\documentclass[11pt,letterpaper,reqno]{amsart}
\usepackage{luatex85}
\usepackage{preamble}
\pdfoutput=1

\title[Deformations of triangulated categories with $t$-structures]{Deformations of triangulated categories with $t$-structures via derived injectives}

\author{F. Genovese}

\author{W. Lowen} 
\address[Wendy Lowen]{Universiteit Antwerpen, Departement Wiskunde, Middelheimcampus,
Middelheimlaan 1,
2020 Antwerp, Belgium}
\email{wendy.lowen@uantwerpen.be}

\author{J. Symons}
\address[Julie Symons]{Universiteit Antwerpen, Departement Wiskunde, Middelheimcampus,
Middelheimlaan 1,
2020 Antwerp, Belgium}
\email{julie.symons@uantwerpen.be}

\author{M. Van den Bergh}

\thanks{
This project has received funding from the European Research Council (ERC) under the European Union’s Horizon 2020 research and innovation programme (grant agreement No. 817762). The third-named author holds a PhD fellowship from the Research Foundation - Flanders (FWO)
}

\makeatletter
\@namedef{subjclassname@2020}{
  \textup{2020} Mathematics Subject Classification}
\makeatother


\makeindex 
\begin{document}

\begin{abstract}
This paper provides the final ingredient in the development of the deformation theory of pretriangulated dg-categories endowed with a nice $t$-structure, which was initiated by the authors and is modeled after the previously developed 
deformation theory of abelian categories. We show how to extend a $t$-structure on a pretriangulated dg-category to its dg-derived category so that the Yoneda embedding becomes $t$-exact. We construct several equivalences between deformation problems; in particular, we prove a deformation equivalence between the bounded $t$-deformations of a bounded $t$-dg-category on the one hand, and dg-deformations of the dg-category of derived injective ind-dg-objects on the other hand. Since this latter dg-category is cohomologically concentrated in nonpositive degrees, we do not encounter curvature.
\end{abstract}

\maketitle
\vspace{-0,5cm}

\tableofcontents

% 0. Conventions
\vspace{-0.5cm}
\section*{Conventions}
%\addcontentsline{toc}{section}{Conventions}
We fix a morphism of commutative dg-rings $\theta: R \to S$ strictly concentrated in nonpositive degrees with kernel $K$ that satisfies the following conditions:
\begin{enumerate}
    \item $\theta$ is strictly surjective;
    \item $R$ and $S$ are \ti{homotopically coherent}\index{homotopically coherent}, meaning that $H^{0}(R)$ and $H^{0}(S)$ are coherent commutative rings and that for all $i \in \mb{Z}$,
    \begin{align*}
        H^{i}(R) \in \modu(H^{0}(R)), \\
        H^{i}(S) \in \modu(H^{0}(S)),
    \end{align*}
    where $\modu(-)$ stands for the category of finitely presented modules;
    \item $S$ and $K$ are finitely presented in $\mc{Z}^{0}(\dgm(R))$. In particular, they are \ti{homotopically finitely presented}\index{homotopically finitely presented} as dg-$R$-modules, meaning that for all $i\in \mb{Z}$,
    $$H^{i}(S), H^{i}(K) \in \modu(H^{0}(R));$$
    \item $K$ is nilpotent of order $n > 0$.
\end{enumerate}

\begin{remark*}
    This agrees with the setup of \cite{GLVdB2}[\S 5.3.1] with the exception of (3), which is necessary for \S \ref{pardefset6} and \Cref{appendix - Universe}. If we take $S = k$ to be a field of characteristic zero, it also agrees with the setup of \cite{CalaqueGrivaux}[\S 1.1] if one works up to homotopy equivalence (one has to strictify to ensure for example the strict surjectivity of (1)).
\end{remark*}

The morphism $\theta$ will be the base change morphism over which we will deform.

\ver

Concerning notation, we write $\DGCat(R)$ for the category of small dg-$R$-categories, with morphisms the dg-$R$-functors. It has a model structure where the weak equivalences are the quasi-equivalences, i.e.\ the dg-functors~$F$ inducing quasi-isomorphisms between internal homs and so that $H^{0}(F)$ is essentially surjective, see \cite{BMR}[\S 3.2]. We denote its homotopy category by $\Hqe(R)$\index{$\Hqe(R)$}. Every small dg-$R$-category $\A$ then has a cofibrant replacement $Q(\A) \to \A$\index{$\Hqe(R)$!cofibrant replacement $Q(\A)$}, which in particular has 
$R$-h-projective $\operatorname{Hom}$-sets. Omitting the base dg-ring, one has the category $\DGCat$ of small dg-categories with as homotopy category $\Hqe$.

We denote the category of dg-$R$-modules by $\dgm(R)$ and, given an $R$-linear dg-category $\A$, we write $\dgm_{R}(\A)$ or $\dgm(\A)$\index{$\dgm(\A)$} for the dg-$R$-category of (right) dg-$\A$-modules, i.e. contravariant dg-functors $\A \to \dgm(R)$. Left dg-$\A$-modules are covariant dg-functors $\A \to \dgm(R)$. Dg-module will always mean right dg-module. We will use $\sim$ to denote quasi-isomorphisms.

\ver
%We will use $\sim$ to denote quasi-isomorphisms, $\approx$ to denote homotopy equivalences and $\cong$ for isomorphisms between objects. In contrast, we will still use $\cong$ to denote quasi-equivalences between categories since these will occur much more often than equivalences.

To take care of set-theoretical problems that might arise in category theory, we work with (ZFCU), meaning the Zermelo-Fraenkel axioms of set theory (ZF), the axiom of choice (C) and the universe axiom (U). We moreover fix a universe $\mc{U}$ that contains $\mb{N}$. For more background on universes, $\mc{U}$-sets and $\mc{U}$-(small) categories, we refer to \cite{AbelianDeformations}[\S 2.1]. For the notion of $\mc{U}$-(small) dg-category, we refer to \Cref{appendix - Universe}. Note that these size restrictions are necessary to ensure that for example the homotopy ind-dg-completion has the expected properties.

% 0. Intro
\section*{Introduction} 

% I Start with motivation
\paragraph{\ti{I. Motivation}} 
The present paper is the final one in a series that began with \cite{GLVdB1}, followed by \cite{GLVdB2}. As mentioned in the introduction of \cite{GLVdB1}, `derived deformations' were introduced to provide a deformation-theoretic interpretation of the higher Hochschild cohomology groups $\{ \tn{HH}^{n}(\mb{A}) \}_{n \geq 2}$ of an abelian category $\mb{A}$ that is linear over a field $k$, seeing as for $n \geq 3$ these groups no longer allow an interpretation in terms of (flat) abelian deformations of $\mb{A}$. 

This goal is met in this paper, as we show that $\tn{HH}^{n}(\mb{A})$ for $n \geq 3$ corresponds to a certain novel type of deformations of $\D_{\dg}^{b}(\mb{A})$ along the dg-$k$-algebra morphism $\theta_{2-n}: k[\epsilon]/(\epsilon^{2}) \to k$ where $\vert \epsilon \vert = 2-n$, see \Cref{appendix - Hochschild}. To obtain this interpretation, we require an equivalence of deformations analogous to the one in the abelian setup \cite{AbelianDeformations},
\begin{equation} \label{eq - thm ab}
    \Def_{\mb{A}}^{\tn{ab}}(\gamma) \cong \Def_{\Inj(\Ind(\mb{A}))}^{\tn{lin}}(\gamma),
\end{equation}
where flat abelian deformations of an $S'$-linear abelian category $\mb{A}$ along a suitable ring morphism $\gamma: R' \to S'$ appear on the left-hand side, and linear deformations of the $S'$-linear category of injectives of $\Ind(\mb{A})$ along $\gamma$ appear on the right-hand side. We prove the following `derived' version,
\begin{thm*}
    Let $\A$ be a strongly pretriangulated $S$-linear dg-category with a bounded $t$-structure and $\theta: R \to S$ be a suitable morphism of dg-rings (see Conventions). Then there is an equivalence
    \begin{equation} \label{eq - thm}
        \Def_{\A}^{t,b}(\theta) \cong \Def_{\DGInj(\Ind^{\dg,Q,+}(\A))}^{\dg}(\theta)
    \end{equation}
    that is natural in $\theta$.
\end{thm*}
On the left-hand side, we now have so-called bounded \ti{$t$-deformations} (\S \ref{pardefset3}) of $\A$ along a suitable dg-morphism $\theta: R \to S$, and on the right-hand side, we have dg-deformations of the dg-category of derived injectives of the \ti{homotopy ind-dg-completion} $\Ind^{\dg,Q,+}(\A)$ endowed with the \ti{induced $t$-structure} (\S \ref{subsection - homotopy ind-dg-completion}-\ref{subsection - induced t-structure}) along $\theta$. This deformation equivalence is the main result of this paper (see \Cref{cor - recap defequiv}), and it is proven by composing several equivalences of deformation pseudofunctors (\S \ref{pardefset5}-\ref{pardefset7}). Since $\Ind^{\dg,Q,+}(\A) \cong \hproj^{+}(\A)$ (\S \ref{subsection - homotopy ind-dg-completion}), we can depict the result as:
\[\begin{tikzcd}
    {\Def^{t,\tn{+lcG}}_{\hproj^{+}(\A)}} & {\Def^{t,\DGInj}_{\hproj^{+}(\A)}} & {\Def^{\dg,\DGInj}_{\DGInj(\hproj^{+}(\A))}} \\
	{\Def^{t,b}_{\A}} && {\Def_{\DGInj(\hproj^{+}(\A))}^{\dg}.}
	\arrow["\cong", from=1-1, to=1-2]
    \arrow["{\ref{cor - equivalence t-def t-locally coherent}}"', from=1-1, to=1-2]
	\arrow["{\textcolor{white}{n} \cong}"', from=2-1, to=1-1]
    \arrow["{\ref{thm - equivalent deformation pseudofunctors bounded} \textcolor{white}{n}}", from=2-1, to=1-1]
	\arrow["\cong"', from=1-3, to=1-2]
    \arrow["{\ref{diagramdef}}", from=1-3, to=1-2]
	\arrow["{\cong \textcolor{white}{n}}"', from=1-3, to=2-3]
    \arrow["{\textcolor{white}{n} \ref{diagramdef}}", from=1-3, to=2-3]
\end{tikzcd}\]

As a consequence, the Hochschild complex $\Cc(\A)$ governs the deformation theory of the $t$-dg-category $\A$. In particular, we have the following result (see \Cref{cor - HHn bounded tdg}).

\begin{corollary*}
    Let $\A$ be an essentially small strongly pretriangulated $k$-linear dg-category with a bounded $t$-structure. For $n \geq 2$,
    \begin{equation*} 
        \tn{HH}^{n}_{\dg}(\A) \cong \begin{Bmatrix} \tn{bounded $t$-deformations} \\ \tn{of $\A$ along } \theta_{2-n} \end{Bmatrix}/ \sim.
    \end{equation*}
\end{corollary*}

Note that this is a highly nontrivial result, as in general the Hochschild cohomology of a dg-category is known to describe curved $A_{\infty}$-deformations, a phenomenon known as the curvature problem \cite{KellerWendy}, \cite{KellerWendyNicolas}, \cite{LVdBCurv}.
However, thanks to the existence of a $B_{\infty}$-quasi-isomorphism between the Hochschild complexes (see \Cref{thm - HH A DGInjInddgA agree}),
$$\Cc(\A) \cong \Cc(\DGInj(\Ind^{\dg,Q,+}(\A))),$$
the corollary follows easily from the theorem above.
Since the dg-category of derived injectives is concentrated in nonpositive degrees, we effectively sidestepped curvature.

In future work, we will investigate the relation to a different and somewhat orthogonal approach to curvature in the absence of $t$-structures, which makes use of alternative derived categories and was initiated in \cite{LehmannLowen}, \cite{Lehmann}.\\

%Recently, another solution was proposed in \cite{LehmannLowen} for infinitesimal dg-deformations of a dg-algebra, i.e. along the $k$-algebra morphism $k[\epsilon]/(\epsilon^{2}) \to k$ for $\vert \epsilon \vert = 0$. This approach is somewhat orthogonal to ours, as a new type of derived category is used -- a `categorified' square zero extension --, while we depart from the classic derived category and its $t$-structure. As both approaches yield interpretations of the Hochschild cohomology of the dg-algebra, a future comparison is in order. \\

% II First say what is already done
\paragraph{\ti{II. Background}} `Derived deformations' originate from \cite{GLVdB2}. There, the authors introduced deformations of pretriangulated dg-categories with a $t$-structure. These are the \ti{$t$-deformations} and they are defined as lifts under the right adjoint to the restriction of scalars-functor for $\theta: R\rightarrow S$, which preserves the $t$-structure. Given a $t$-dg-category, one may consider its full dg-subcategory of  \textit{derived injectives}, which will be recalled in \S \ref{pardefset2} (in the derived category of an abelian category $\mb{A}$ endowed with the standard $t$-structure, the derived injectives are just the usual injectives in $\mb{A}$, concentrated in degree zero). If there are enough derived injectives, they can be used to reconstruct the whole $t$-dg-category if it belongs to $\Hqe^{t+}_{\Pi}$, i.e. if it has a left bounded, non-degenerate $t$-structure that is closed under countable products, as was shown in \cite{GLVdB1}. We can deform the subcategory on the derived injectives as a dg-category using the natural notion of \ti{dg-deformation}, now using the left adjoint to the restriction of scalars.

Using the reconstruction theorems from \cite{GLVdB1}, the authors of \cite{GLVdB2} were able to show that dg-deformations of dg-categories of derived injectives induce $t$-deformations of the $t$-dg-categories associated to them. In this paper, we aim to show the converse to ultimately prove \eqref{eq - thm}. \\

% III Explain our approach + explain the structure of the paper: two parts
\paragraph{\ti{III. Outline}} Our approach is inspired by the abelian story of \cite{AbelianDeformations}. In loc.\ cit.\ an equivalence of deformation pseudofunctors is constructed between flat abelian deformations of a small abelian category (which in general does not have enough injectives) and flat abelian deformations of its ind-completion, a locally coherent Grothendieck category that has enough injectives. In our setup, the role of the small category is taken on by an essentially small bounded strongly pretriangulated $t$-dg-category (which in general does not have enough derived injectives). For the big category, we need a $t$-dg-analogue of a locally coherent Grothendieck category. This is the focus of the first chapter.

\ver

\paragraph{\tb{Chapter 1}} We start by treating filtered weighted colimits and enriched ind-completions in \S \ref{subsection - weighted colim}-\ref{subsection - enriched ind-completions}. We then specialize to the cosmos of dg-$R$-modules and filtered homotopy dg-colimits in \S \ref{subsection - homotopy ind-dg-completion} in order to construct a dg-enhancement of the derived category of a pretriangulated dg-category in terms of filtered homotopy dg-colimits of representable objects, see \Cref{thm - homotopy dg completion closed}. If the pretriangulated dg-category one starts out from is strongly pretriangulated, this can be viewed as the \ti{homotopy ind-dg-completion} since it satisfies the appropriate universal property. It is also the dg-category that will serve as the `big category' above, i.e. the $t$-dg-analogue of the ind-completion in the aforementioned abelian setting. This requires endowing the dg-derived category of a pretriangulated dg-category with a $t$-structure that is compatible with the original one. To this end, we treat some preliminaries on $t$-structures in \S \ref{subsection - t-structures}, after which we extend the $t$-structure from a strongly pretriangulated dg-category to its homotopy ind-dg-completion or dg-derived category in \S \ref{subsection - induced t-structure}, see \Cref{thm - t-structure inddg}. The \ti{induced $t$-structure} makes the Yoneda embedding $t$-exact and its heart is the abelian ind-completion of the original heart. However, non-degeneracy is not preserved, see \Cref{ex - degenerate induced t-structure}. Finally, we come to the definition of a \ti{left bounded locally coherent Grothendieck $t$-dg-category} in \S \ref{subsection - t-lc} together with an equivalence of categories -- \Cref{thm - categorical hfpb inddg} --, which is what we set out to do. To show that left bounded locally coherent Grothendieck $t$-(dg-)categories have enough derived injectives, we need a variation to Brown representability's theorem in the left bounded setup, where the coaisle has small coproducts. This is the content of \Cref{appendix - Brown}.

\ver
With the machinery from the first chapter available, we go on to prove the deformation equivalences that will combine into \eqref{eq - thm} in the second chapter.
\ver

\paragraph{\tb{Chapter 2}} We start by setting the scene in \S \ref{pardefset1}-\ref{pardefset3}: we introduce the ($t$-)dg-categories we will be considering, as well as the corresponding notions of deformations, and recall the necessary results of \cite{GLVdB1} and \cite{GLVdB2}. Next, we complete the deformation equivalence between $t$-deformations in $\Hqe^{t+}_{\Pi}$ of \cite{GLVdB1}[Theorem 1.4] and dg-deformations of the associated dg-categories of derived injectives in \S \ref{pardefset4}-\ref{pardefset5}. We then show that such a $t$-deformation preserves left bounded locally coherent Grothendieck $t$-dg-categories in \S \ref{pardefset6}. To move to this setting, we use the results of Chapter 1 (in particular \Cref{thm - categorical hfpb inddg}) to prove a deformation equivalence between bounded $t$-deformations of essentially small strongly pretriangulated bounded $t$-dg-categories and left bounded locally coherent Grothendieck $t$-deformations of their dg-derived categories in \S \ref{pardefset7}. This provides the final piece needed to prove \eqref{eq - thm}.

% 1 --------------------------------------------------------------------

\section{Extending \texorpdfstring{$t$}{t}-structures to the derived category} \label{section - extend t-structures}

The aim in this chapter is to extend the $t$-structure on a pretriangulated dg-category $\A$ to its derived category $\D(\A)$ so that the embedding becomes $t$-exact (\S \ref{subsection - induced t-structure}). This we will achieve by introducing a dg-enhancement of the derived category $\D(\A)$ in which the objects are filtered homotopy dg-colimits of representable objects (\S \ref{subsection - homotopy ind-dg-completion}). These filtered homotopy dg-colimits are filtered weighted colimits where the weight is a resolution of the constant diagram, as will be recalled in \S\ref{subsection - weighted colim} and \S\ref{subsection - homotopy ind-dg-completion}. Since the filtered homotopy dg-colimit is quasi-isomorphic to the filtered dg-colimit, which uses the constant diagram, we start by studying filtered dg-colimits with constant weights in \S\ref{subsection - enriched ind-completions}. This will then prove useful for \S\ref{subsection - homotopy ind-dg-completion}.

After extending the $t$-structure in \S \ref{subsection - induced t-structure}, we ask whether non-degeneracy is preserved. This proves the case when the $t$-structure on $\A$ is bounded. The functor $\A \mapsto \D_{\dg}(\A)$ then forms an equivalence of categories with the left bounded locally coherent Grothendieck $t$-dg-categories, introduced in \S \ref{subsection - t-lc}.

% 1.1 ------------------------------------------------------------------

\subsection{Enriched category theory and weighted (co)limits} \label{subsection - weighted colim}

For background on enriched category theory, we refer to \cite{Kelly}. Below, we will recall weighted (co)limits following \cite[\S 2-3]{Kelly}, which are introduced there under the term indexed (co)limits. In enriched category theory, ordinary (co)limits defined in terms of the representability of (co)cones do not suffice, so the wider notion of weighted (co)limit was introduced.

Let $\V$ be a \ti{cosmos}\index{cosmos}, i.e. a bicomplete closed symmetric monoidal category. 

\begin{definition} \label{def - weighted colim}
    Let $\A$ be a $\V$-category, $\mc{I}$ be a small $\V$-category and $F: \mc{I} \to \A$, $W: \mc{I} \to \V$, $W': \mc{I}^{\op} \to \V$ be $\V$-functors. We call $W$ and $W'$ the \ti{weights}\index{weight}. 
    \begin{itemize}
        \item[-] If the $\V$-functor 
    $\A^{\op} \to \V: A \mapsto \V-\Cat(\mc{I},\V)(W,\A(A,F(-)))$
    allows a representation
    $$\A(A,\varprojlim{}^{W} F) \cong \V-\Cat(\mc{I},\V)(W,\A(A,F(-)))$$
    with counit
    $\lambda: W \to \A(\varprojlim{}^{W} F,F(-))$,
    then we call $(\varprojlim{}^{W} F, \lambda)$ the \ti{weighted limit of $F$ with weight $W$}\index{weighted (co)limit}. 
        \item[-] If the $\V$-functor 
    $\A \to \V: A \mapsto \V-\Cat(\mc{I}^{\op},\V)(W',\A(F(-),A))$
    allows a representation
    $$\A(\colim{}^{W'} F, A) \cong \V-\Cat(\mc{I},\V)(W',\A(F(-),A)),$$
    with unit
    $\lambda': W' \to \A(F(-),\colim{}^{W'} F)$,
    then we call $(\colim{}^{W'} F, \lambda')$ the \ti{weighted colimit of $F$ with weight $W'$}. 
    \end{itemize}
\end{definition}

\begin{remark}
    Observe that for a small $\V$-category $\mc{I}$, there is no way of defining a canonical, constant weight $W: \mc{I} \to \V$ that maps everything to the monoidal unit $I$. This is the weight we would want to plug into \Cref{def - weighted colim} to obtain ordinary (co)limits and it is precisely why weighted (co)limits are needed in the enriched setup. If the index category $\mc{I}$ is of a particular form however, this can be done, as we will observe in \Cref{ex - constant diagram}.
\end{remark}

Next, we mention some special cases of weighted limits and colimits.

\begin{example}
    Let $\mc{I} = \V\{*\}$ be the unit $\V$-category, with one object $*$ and as inner hom the monoidal unit $I_{\V}$. Then we can consider the $\V$-functors $W: \V\{*\} \to \V$ and $F: \V\{*\} \to \A$ that are fully determined by the image of $*$, so by objects $V \in \V$ and $A \in \A$ respectively, in the sense that on the level of morphisms they are given by the identity,
    \begin{align*}
        \mc{Z}^{0}(\V(I_{\V},\V(V,V))) \cong \mc{Z}^{0}(\V(V,V)), \\
        \mc{Z}^{0}(\V(I_{\V},\A(A,A))) \cong \mc{Z}^{0}(\A(A,A)).
    \end{align*}
    We therefore write $c_{V}$ and $c_{A}$ instead of $W$ and $F$.
    
    \begin{enumerate}
        \item The \ti{cotensor product} or \ti{internal hom} of $V$ and $A$ is defined as
        $$[V,A] := \varprojlim{}^{c_{V}} c_{A}.$$

        \item The \ti{tensor product} of $V$ and $A$ is defined as
        $$V \otimes A := \colim{}^{c_{V}} c_{A}.$$
    \end{enumerate}
\end{example}

\begin{definition}
    A $\V$-category $\A$ is \ti{cotensored} or said to \ti{admit cotensor products} (resp. \ti{tensored} and \ti{admitting tensor products}) if $[V,A]$ (resp. $V \otimes A$) exists for all $V \in \V, A \in \A$.\index{weighted (co)limit!(co)tensored}
\end{definition}

\begin{example}
Let $F$ be a $\V$-functor $\B^{\op} \otimes \B \to \A$ for some $\V$-category $\B$ and consider the weight $W = \homom_{\B}: \B^{\op} \otimes \B \to \V$.
    \begin{enumerate}
        \item The \ti{end of $F$} is defined as
        $$\int_{B \in \B} F(B,B) := \varprojlim{}^{\homom_{\B}}F.$$\index{weighted (co)limit!(co)end}
        \item The \ti{coend of $F$} is defined as
        $$\int^{B \in \B} F(B,B) := \colim{}^{\homom_{\B}} F.$$
    \end{enumerate}
\end{example}

One can show that if a $\V$-category is cotensored (resp. tensored), all weighted colimits (resp. weighted limits) with that target category exist and can be computed using ends and cotensors (resp. coends and tensors), see \cite{Kelly}[(3.69),(3.70)].

\begin{proposition} \label{prop - tensor inner}
    Let $W: \mc{I} \to \V$, $W': \mc{I}^{\op} \to \V$ and $F: \mc{I} \to \A$ be $\V$-functors and $\mc{I}$ be a small $\V$-category.
    \begin{enumerate}
        \item If $\A$ is cotensored, then $$\varprojlim{}^{W} F \cong \int_{i \in \mc{I}} [W(i),F(i)].$$
        We will also use the notation $[W, F]$ and call this the \ti{inner hom}.\index{weighted (co)limit!inner hom and tensor product}
        \item If $\A$ is tensored, then $$\colim{}^{W'} F \cong \int^{i \in \mc{I}^{\op}} W'(i) \otimes F(i).$$ We will also use the notation $W' \otimes_{\mc{I}} F$ and call this the \ti{tensor product}.
    \end{enumerate}
\end{proposition}

\begin{examples}
    \begin{enumerate}
        \item The cosmos $\V$ itself is both tensored and cotensored, where the tensor product is the monoidal product and the cotensor product is the internal hom.

        \item Let $\V = \dgm(R)$ for a dg-ring $R$ and let $\A$ be a dg-$R$-category. Then the dg-category $\dgm(\A)$ of dg-$\A$-modules is both tensored and cotensored, see \eqref{eq - dgmA tensor} and \eqref{eq - dgmA cotensor}. In particular, it is bicomplete.
    \end{enumerate}
\end{examples}

\subsection{Enriched ind-completions} \label{subsection - enriched ind-completions}

Let $\V$ be a cosmos so that the functor corepresented by the monoidal unit, $\V(I_{\V},-)$, preserves filtered constantly weighted colimits. Let $\A$ be a $\V$-category. We start by fixing terminology.

\begin{definition}
    A \ti{diagram in $\A$} is a $\V$-functor $X: \mc{I} \to \A$ from a small $\V$-category $\mc{I}$ to $\A$. $I$ is \ti{filtered} if it is the free $\V$-category $\V\{P\}$ generated by a filtered poset $(P,\preceq)$. An \ti{ind-$\V$-object in $\A$} (resp. \ti{pro-$\V$-object in $\A$}) is then a covariant (resp. contravariant) $\V$-functor $\V\{P\} \to \A$.\index{diagram} \index{diagram!ind-$\V$-object} \index{filtered $\V$-(co)limit!pro-$\V$-object}
\end{definition}

\begin{example} \label{ex - constant diagram}
    Given a poset $(P,\preceq)$, we can consider the \ti{constant diagram}\index{diagram!constant diagram} $$c_{P}: \V\{P\} \to \V,$$ which is the $\V$-functor corresponding to the functor that maps everything to the monoidal unit $I_{\V}$ under the adjunction $\V(-): \Cat \leftrightharpoons \V-\Cat: U(-)$ between the free and forgetful functor.
    %, i.e. to
   % \[\begin{tikzcd}[row sep=small]
	%P & {U(\V):i} & {I_{\V}} \\
	%& {\textcolor{white}{U(\V):}j} & {I_{\V}.}
	%\arrow[from=1-1, to=1-2]
	%\arrow["\preceq"', shift left=7, from=1-2, to=2-2]
	%\arrow["\id", from=1-3, to=2-3]
	%\arrow[maps to, from=1-2, to=1-3]
	%\arrow[maps to, from=2-2, to=2-3]
%\end{tikzcd}\]
    Since $P^{\op}$ remains a poset, we can also consider the constant diagram $c_{P^{\op}}: \V\{P\}^{\op} \to \V$, where we identified $\V\{P^{\op}\} \cong \V\{P\}^{\op}$.
\end{example}

These constant diagrams will serve as our weights in this subsection. In this case, weighted colimits are just ordinary colimits.

\begin{definition} \label{def - filtered enriched colimit}
    Consider an ind-$\V$-object $X: \V\{P\} \to \A$ and a pro-$\V$-object $X': \V\{P\} \to \A$ in $\A$. If they exist, we call 
     \begin{align*}
         \varprojlim{}^{c_{P^{\op}}} X', \\
         \tn{resp. \textcolor{white}{l}} \colim{}^{c_{P^{\op}}} X,
     \end{align*}
     the \ti{filtered $\V$-limit of $X'$}\index{filtered $\V$-(co)limit} (resp. \ti{filtered $\V$-colimit of $X$}). We say that $\A$ is \emph{closed under filtered $\V$-colimits} (resp. \emph{closed under filtered $\V$-limits}) if these exist for all ind-$\V$-objects in $\A$ (resp. pro-$\V$-objects in $\A$). \index{filtered $\V$-(co)limit!closed under filtered $\V$-(co)limits}
\end{definition}

\begin{remark} \label{remark - yoneda filtered colimit exists}
    We can always compose a diagram $X: \mc{I} \to \A$ with the enriched Yoneda embedding $\Y$ so that it becomes a diagram in $\V^{\A^{\op}}$, which is bicomplete. The weighted (co)limit then always exists. Note that if $X$ is an ind-$\V$-object in $\A$ and its filtered $\V$-colimit exists, it is the object representing
    $$\varprojlim{}^{c_{P^{\op}}} \V^{\A^{\op}}(X,-)$$
    because Yoneda reflects limits. Similarly, if $X$ is a pro-$\V$-object in $\A$ and its filtered $\V$-limit exists, it is the object representing
    $$\varprojlim{}^{c_{P^{\tn{op}}}} \V^{\A^{\op}}(-,X).$$
\end{remark}

\begin{definition}
    Let $\A$ be an essentially small $\V$-category. We define the \ti{$\V$-category $\Ind\Ob^{\V}(\A)$ of ind-$\V$-objects in $\A$} as the data: \index{filtered $\V$-(co)limit!ind-$\V$-object!$\Ind\Ob^{\V}(\A)$}
    \begin{itemize}
        \item The objects are the ind-$\V$-objects in $\A$;
        \item For any two objects $X:\V\{P_{X}\} \to \A, Z:\V\{P_{Z}\} \to \A$, \begin{align*}
            \Ind\Ob^{\V}(\A)(X, Z) :&= \V^{\A^{\op}}( \colim{}^{c_{P_{X}^{\op}}} \Y 
        \circ X, \colim{}^{c_{P_{Z}^{\op}}} \Y \circ Z ) \\
        &\cong \varprojlim_{i \in P_{X}} \colim_{j \in P_{Z}} \A(X_{i},Z_{j}),
        \end{align*}
        where we abbreviated $X(i)$ and $Z(j)$ to $X_{i}$ and $Z_{j}$ respectively.
    \end{itemize}
\end{definition}

There are natural fully faithful $\V$-functors,
\[\begin{tikzcd}[row sep = small]
	{\iota_{\A}:\A} & {\Ind\Ob^{\V}(\A):A} & {[c_{A}: \V\{*\} \longrightarrow \A: * \longmapsto A]}
	\arrow[from=1-1, to=1-2]
	\arrow[maps to, from=1-2, to=1-3]
\end{tikzcd}\]
which is moreover injective on objects, and\index{ind-$\V$-completion!$\Ll_{\A}$}
\[\begin{tikzcd}[row sep = small]
	{\Ll_{\A}:\Ind\Ob^{\V}(\A)} & {\V^{\A^{\op}}: X} & {\Ll_{\A}(X) := \colim {}^{c_{P_{X}^{\op}}} \Y \circ X.}
	\arrow[from=1-1, to=1-2]
	\arrow[maps to, from=1-2, to=1-3]
\end{tikzcd}\] 

\begin{definition}
    Let $\A$ be an essentially small $\V$-category. The \ti{ind-$\V$-completion of $\A$}, denoted by $\Ind^{\V}(\A)$, is the essential image of $\Ll_{\A}$.\index{ind-$\V$-completion}\index{ind-$\V$-completion!$\Ind^{\V}(\A)$}
\end{definition}

%Consequently, the $\V$-functors $\iota_{\A}$ and $\Ll_{\A}$ fit into the following commutative diagram,
%\begin{equation} \label{eq - Yoneda embedding ind} \begin{tikzcd}[row sep = small]
	%\A & {\Ind\Ob^{\V}(\A)} & {\Ind^{\V}(\A) \subseteq \V^{\A^{\op}}.}
	%\arrow["{\iota_{A}}", hook, from=1-1, to=1-2]
	%\arrow["\Ll_{\A}", from=1-2, to=1-3]
	%\arrow["\Y"', hook, curve={height=20pt}, from=1-1, to=1-3]
%\end{tikzcd}\end{equation}

By definition, the ind-$\V$-completion of $\A$ is a full $\V$-subcategory of $\V^{\A^{\op}}$ and $\Ind\Ob^{\V}(\A) \cong \Ind^{\V}(\A)$.
The following theorem justifies its name.

\begin{thm} \label{thm - enriched completion closed}
    $\Ind^{\V}(\A)$ has all filtered $\V$-colimits.
    \begin{proof} 
        We have to show that the filtered $\V$-colimit of any ind-$\V$-object in $\Ind^{\V}(\A)$ lies in the essential image of $\Ll_{\A}$. So consider such an ind-$\V$-object,
        \[\begin{tikzcd}[row sep=small]
	{X: \V\{P\}} & {\Ind^{\V}(\A):i} & {\Ll_{\A}(X_{i}: \V\{P_{i}\} \to \A)} \\
	& {\textcolor{white}{ \Ind^{\V}(\A):} j} & {\Ll_{\A}(X_{j}: \V\{P_{j}\} \to \A).}
	\arrow[from=1-1, to=1-2]
	\arrow[maps to, from=1-2, to=1-3]
	\arrow[maps to, from=2-2, to=2-3]
	\arrow["\preceq"', shift left=10, from=1-2, to=2-2]
	\arrow["{X(i\preceq j)}", from=1-3, to=2-3]
\end{tikzcd}\]
        The idea is to define a filtered poset $\ol{P}$ that combines all the posets $\{ P_{i}, i \in P \}$ and then distill from the ind-$\V$-objects $X_{i}$ a new ind-$\V$-object $\ol{X}: \V\{\ol{P}\} \to \A$. 
        
        We define $\ol{P}$ as the set $\bigcup_{i \in P} \{i\} \times P_{i}$ and endow it with the relation
        \begin{align} \label{eq - new poset V} & \tn{ } (i,t_{i}) \preceq (j,t_{j}) \nonumber \\
        \Leftrightarrow &\begin{cases} (i = j) \wedge (t_{i} \preceq t_{j}) \\ 
        (i \prec j) \wedge (\exists \psi^{i}_{j}: X_{i,t_{i}} \to X_{j,t_{j}} \tn{ s.t. }X(i \prec j) \circ \mu_{i,t_{i}} = \mu_{j,t_{j}} \circ \Y(\psi^{i}_{j})), \end{cases}
        \end{align}
        where for each $i \in P$, $\mu_{i}$ is the dinatural transformation between $\homom_{\V\{P_{i}\}}(-,-)$ and $\dgm(\A)(c_{P_{i}}(-) \otimes \Y \circ X_{i}(-),\Ll_{\A}(X_{i}))$ in $\V^{\V\{P_{i}\}^{\op} \otimes \V\{P_{i}\}}$ that accompanies the weighted colimit (which is a coend). As $c_{P_{i}}$ is constant, this amounts to a $\V$-natural transformation $\mu_{i}$ between $\homom_{\V\{P_{i}\}}$ and $\dgm(\A)(\Y \circ X_{i}(-),\Ll_{\A}(X_{i}))$ in $\V^{\V\{P_{i}\}}$, where the first $\V$-functor maps each $t_{i}$ to $\V\{P_{i}\}(t_{i},t_{i}) = I_{\V}$, the monoidal unit. Consequently, $\mu_{i}$ translates to a family of maps in $\V^{\A^{\op}}$ $$\{ \mu_{i,t_{i}}: \Y(X_{i,t_{i}}) \to \Ll_{\A}(X) \}_{t_{i} \in P_{i}}$$
        so that $\mu_{j,t_{j}} \circ \Y(X_{i}(t_{i} \prec t_{j})) = \mu_{i,t_{i}}$ for all $t_{i} \prec t_{j}$ in $P_{i}$. The condition on $\psi^{i}_{j}$ is then the commutativity of the following diagram,
        \[\begin{tikzcd}[row sep = small]
	{\Y(X_{i,t_{i}})} & {\Ll_{\A}(X_{i})} \\
	{\Y(X_{j,t_{j}})} & {\Ll_{\A}(X_{j}).}
	\arrow["{\Y(\psi^{i}_{j})}"', dashed, from=1-1, to=2-1]
	\arrow["{\mu_{j,t_{j}}}"', from=2-1, to=2-2]
	\arrow["{\mu_{i,t_{i}}}", from=1-1, to=1-2]
	\arrow["{X(i \prec j)}", from=1-2, to=2-2]
\end{tikzcd}\]
        One can check that \eqref{eq - new poset V} defines a partial order on $\ol{P}$. To see that $(\ol{P},\preceq)$ is filtered, consider elements $(i,t_{i})$ and $(j,t_{j})$. We distinguish three cases:
        \begin{enumerate}
            \item $\ul{i = j}$: Then we can use the filteredness of $P_{i}$ to find an element $t_{k} \in P_{i}$ so that $t_{i} \preceq t_{k}$ and $t_{j} \preceq t_{k}$. Then $(i,t_{i}) \preceq (i,t_{k})$ and $(j,t_{j}) \preceq (i,t_{k})$.
            
            \item $\ul{(i \prec j) \vee (j \prec i)}$: We treat the case of $i \prec j$. Consider the map 
            \[\begin{tikzcd}[row sep = small]
	{\Y(X_{i,t_{i}})} & {\Ll_{\A}(X_{i})} \\
	& {\Ll_{\A}(X_{j}).}
	\arrow["{\mu_{i,t_{i}}}", from=1-1, to=1-2]
	\arrow["{X(i \prec j)}", from=1-2, to=2-2]
\end{tikzcd}\]
            This is an element of
            \begin{align*}
                (\V^{\A^{\op}})_{0}(\Y(X_{i,t_{i}}), \Ll_{\A}(X_{j})) &= \V(I_{\V},\V^{\A^{\op}}(\Y(X_{i,t_{i}}), \Ll_{\A}(X_{j}))) \\
                &\cong \V(I_{\V},\Ll_{\A}(X_{j})(X_{i,t_{i}})) \\
                &= \V(I_{\V},(\colim{}^{c_{P_{j}^{\op}}} \Y \circ X_{j})(X_{i,t_{i}})) \\
                &\cong \V(I_{\V},\colim{}^{c_{P_{j}^{\op}}} \A(X_{i,t_{i}},X_{j}(-))) \\
                &\cong \colim{}_{t_{j} \in P_{j}} \A_{0}(X_{i,t_{i}},X_{j,t_{j}}),
            \end{align*}
            by our assumption on $\V(I_{\V},-)$. So we can use the description of filtered colimits in $\Set$. Consequently, $X(i \prec j) \circ \mu_{i,t_{i}}$ corresponds to (the equivalence class of) a map $\psi^{i}_{j} \in \A(X_{i,t_{i}}, X_{j,t_{j}^{i}})$, which makes the diagram below commute.
            \[\begin{tikzcd}[row sep = small]
	{\Y(X_{i,t_{i}})} & {\Ll_{\A}(X_{i})} \\
	{\Y(X_{j,t_{j}^{i}})} & {\Ll_{\A}(X_{j}).}
	\arrow["{\mu_{i,t_{i}}}", from=1-1, to=1-2]
	\arrow["{X(i \prec j)}", from=1-2, to=2-2]
	\arrow["{\mu_{j,t_{j}^{i}}}"', from=2-1, to=2-2]
	\arrow["{\Y(\psi^{i}_{j_{i}})}"', from=1-1, to=2-1]
\end{tikzcd}\]
            Due to $P_{j}$ being filtered, there exists some $t' \in P_{j}$ so that $t_{j}^{i} \preceq t'$ and $t_{j} \preceq t'$. The following diagram then commutes
            \[\begin{tikzcd}[row sep = small]
	{\Y(X_{i,t_{i}})} & {\Ll_{\A}(X_{i})} \\
	{\Y(X_{j,t_{j}^{i}})} & {\Ll_{\A}(X_{j}).} \\
	{\Y(X_{j,t'})}
	\arrow["{\mu_{i,t_{i}}}", from=1-1, to=1-2]
	\arrow["{X(i \prec j)}", from=1-2, to=2-2]
	\arrow["{\mu_{j,t_{j}^{i}}}"', from=2-1, to=2-2]
	\arrow["{\Y(\psi^{i}_{j})}"', from=1-1, to=2-1]
	\arrow["{\Y(X_{j}(t_{j}^{i}\preceq t'))}"', from=2-1, to=3-1]
	\arrow["{\mu_{j,t'}}"', from=3-1, to=2-2]
\end{tikzcd}\]
            Hence $(j,t_{j}) \preceq (j,t')$ and $(i,t_{i}) \preceq (j,t')$ via $X_{j}(t^{i}_{j} \preceq t') \circ \psi^{i}_{j_{i}}$.
            
            \item \underline{Otherwise}: Since $P$ is filtered, we can find some $k \in P$ so that $i \preceq k$ and $j \preceq k$. As in the second case, we can then find maps $\psi^{i}_{k_{i}}$ and $\psi^{j}_{k_{j}}$ that make the following diagrams commute:
            \[\begin{tikzcd}[row sep = small]
	{\Y(X_{i,t_{i}})} & {\Ll_{\A}(X_{i})} && {\Y(X_{j,t_{j}})} & {\Ll_{\A}(X_{j})} \\
	{\Y(X_{k,t_{k}^{i}})} & {\Ll_{\A}(X_{k})} && {\Y(X_{k,t_{k}^{j}})} & {\Ll_{\A}(X_{k}).}
	\arrow["{\mu_{i,t_{i}}}", from=1-1, to=1-2]
	\arrow["{X(i \prec k)}", from=1-2, to=2-2]
	\arrow["{\mu_{k,t_{k}^{i}}}"', from=2-1, to=2-2]
	\arrow["{\Y(\psi^{i}_{k_{i}})}"', from=1-1, to=2-1]
	\arrow["{\mu_{j,t_{j}}}", from=1-4, to=1-5]
	\arrow["{X(j \prec k)}", from=1-5, to=2-5]
	\arrow["{\mu_{k,t_{k}^{j}}}"', from=2-4, to=2-5]
	\arrow["{\Y(\psi^{j}_{k_{j}})}"', from=1-4, to=2-4]
\end{tikzcd}\]
            Now let $t_{k} \in P_{k}$ so that $t_{k}^{i} \preceq t_{k}$ and $t_{k}^{j} \preceq t_{k}$. Then we can conclude from the commutativity of the diagrams below that $(i,t_{i}) \preceq (k,t_{k})$ and $(j,t_{j}) \preceq (k,t_{k})$.
            \[\begin{tikzcd}[row sep = small]
	{\Y(X_{i,t_{i}})} & {\Ll_{\A}(X_{i})} && {\Y(X_{j,t_{j}})} & {\Ll_{\A}(X_{j})} \\
	{\Y(X_{k,t_{k}^{i}})} & {\Ll_{\A}(X_{k})} && {\Y(X_{k,t_{k}^{j}})} & {\Ll_{\A}(X_{k})} \\
	{\Y(X_{k,t_{k}})} &&& {\Y(X_{k,t_{k}})}
	\arrow["{\mu_{i,t_{i}}}", from=1-1, to=1-2]
	\arrow["{X(i \prec k)}", from=1-2, to=2-2]
	\arrow["{\mu_{k,t_{k}^{i}}}"', from=2-1, to=2-2]
	\arrow["{\Y(\psi^{i}_{k_{i}})}"', from=1-1, to=2-1]
	\arrow["{\mu_{j,t_{j}}}", from=1-4, to=1-5]
	\arrow["{X(j \prec k)}", from=1-5, to=2-5]
	\arrow["{\mu_{k,t_{k}^{j}}}"', from=2-4, to=2-5]
	\arrow["{\Y(\psi^{j}_{k_{j}})}"', from=1-4, to=2-4]
	\arrow["{\mu_{t,t_{k}}}"', from=3-1, to=2-2]
	\arrow["{\mu_{t,t_{k}}}"', from=3-4, to=2-5]
	\arrow["{\Y(X_{k}(t^{i}_{k} \preceq t_{k}))}"', from=2-1, to=3-1]
	\arrow["{\Y(X_{k}(t^{j}_{k} \preceq t_{k}))}"', from=2-4, to=3-4]
\end{tikzcd}\]
            
        \end{enumerate}
        So we constructed a filtered poset $\ol{P}$ and can now define the $\V$-functor
        \[\begin{tikzcd}[row sep = small, ampersand replacement=\&]
	{\ol{X}: \V\{\ol{P}\}} \& {\A:(i,t_{i})} \& {X_{i,t_{i}}} \\
	\& {\textcolor{white}{\A:}(j,t_{j})} \& {X_{j,t_{j}},}
	\arrow[from=1-1, to=1-2]
	\arrow["\preceq"', shift left=5, from=1-2, to=2-2]
	\arrow[maps to, from=1-2, to=1-3]
	\arrow[maps to, from=2-2, to=2-3]
	\arrow["{\quad \begin{cases} X_{i}(t_{i} \leq t_{j}), &i=j \\ \psi^{i}_{j}, &i \prec j \end{cases}}", from=1-3, to=2-3]
\end{tikzcd}\]
        where we abbreviated $X_{i}(t_{i})$ to $X_{i,t_{i}}$ and used $\V(-): \Cat \leftrightharpoons \V-\Cat: U(-)$. Using the natural maps $\Ll_{\A}(X_{i}) \to \colim^{c_{P^{\op}}} X$ and $\Y \circ \ol{X}(i,t_{i}) \to \Ll_{\A}(\ol{X})$, we conclude that the filtered $\V$-colimit of $X$ is isomorphic to $\Ll_{\A}(\ol{X})$ and thus contained in the essential image of $\Ll_{\A}(-)$.
    \end{proof}
\end{thm}

As in the non-enriched case, the ind-$\V$-completion satisfies a universal property.

\begin{corollary} \label{cor - universal prop enriched}
    For any $\V$-functor $F: \A \to \B$ where $\A$ is essentially small and $\B$ has all filtered $\V$-colimits, there exists a unique $\V$-functor $\Ind^{\V}(F): \Ind^{\V}(\A) \to \B$ that makes the diagram below commute:
    \begin{equation} \label{eq - universal prop hodg} \begin{tikzcd}[row sep=small]
	\A & {\B.} \\
	{\Ind^{\V}(\A)}
	\arrow["F", from=1-1, to=1-2]
	\arrow["{\Ind^{\V}(F)}"', pos=0.3, shift right=2, dashed, from=2-1, to=1-2]
	\arrow["{\Y}"', hook, from=1-1, to=2-1]
\end{tikzcd}\end{equation}
    \begin{proof}
        One can check that the following defines a $\V$-functor,
        \[\begin{tikzcd}[row sep = small]
	{\Ind^{\V}(F):\Ind^{\V}(\A)} & {\B: \Ll_{\A}(X)} & {\colim{}^{c_{P_{X}^{\op}}} F \circ X} 
	%& {\textcolor{white}{\B:}\Ll_{\A}(Z)} & {\colim{}^{c_{P_{Z}^{\op}}} F \circ Z,}
	\arrow[from=1-1, to=1-2]
	\arrow[maps to, from=1-2, to=1-3]
	%\arrow[from=1-2, to=2-2]
	%\arrow[maps to, from=2-2, to=2-3]
	%\arrow[shift left=5, from=1-3, to=2-3]
\end{tikzcd}\]
    and that this is the unique one to make diagram \eqref{eq - universal prop hodg} commute.
    \end{proof}
\end{corollary}

% 1.3 ------------------------------------------------------------------

\subsection{An enhancement of the derived category via filtered homotopy dg-colimits or the homotopy ind-dg-completion} \label{subsection - homotopy ind-dg-completion}

We now take the cosmos $\V$ to be $\dgm(R)$. Note that $\mc{Z}^{0}(-)$ indeed preserves filtered constantly weighted colimits so that the results of \S\ref{subsection - enriched ind-completions} apply. In what follows, we will often abbreviate $\V = \dgm(R)$ to $R$, for example by writing $R\{P\}$ instead of $\dgm(R)\{P\}$. We will also use the term \ti{ind-dg-object} (resp. \ti{pro-dg-object}) instead of ind-$\dgm(R)$-object (resp. proj-$\dgm(R)$-object) and \ti{filtered dg-(co)limit} instead of filtered $\dgm(R)$-(co)limit. \index{filtered homotopy dg-(co)limit!ind-dg-object} \index{filtered homotopy dg-(co)limit!pro-dg-object}

We start by recalling filtered homotopy dg-(co)limits. These are filtered weighted (co)limits that are stable under quasi-isomorphisms between the diagrams. 
%the homotopy of the diagram category is not Fun(H(I),H(A)] or similar, you cannot get representables of the morphisms in the homotopy category to be a functorial diagram.
The idea is simple: as was recalled in \S\ref{subsection - weighted colim}, a weighted dg-(co)limit is a tensor product (resp. inner hom), see \Cref{prop - tensor inner}. Hence, we simply resolve one of the two components, usually the weight. So consider a dg-category $\A$. For a filtered poset $(P,\preceq)$, let 
\begin{align*}
    Q(c_{P^{\op}}) &\to c_{P^{\op}}, \\
    Q(c_{P}) &\to c_{P},
\end{align*}
be fixed h-projective\footnote{That we need a h-projective resolution instead of a h-flat resolution, will become clear from \Cref{thm - homotopy dg completion closed}.} resolutions of the constant diagrams (cf. \Cref{ex - constant diagram}), e.g. the bar resolution\footnote{To obtain the bar resolution of a dg-module $M: \A \to \dgm(R)$, one tensors $M$ with the bar-resolution of $\A$, which is an $\A-\A$-dg-bimodule, over $\A$. See for example \cite{KellerWendy}[Proof of Proposition 2.2] for the bar resolution of a dg-category $\A$ as an $\A-\A$-dg-bimodule.}.\index{bar resolution}

\begin{definition} \label{def - filtered homotopy dg-colimit} 
     Consider an ind-dg-object $X: R\{P\} \to \A$ and a pro-dg-object $X': R\{P\} \to \A$ in $\A$. If they exist, we call
     \begin{align*}
         \varprojlim{}^{Q(c_{P^{\op}})} X', \\
         \tn{resp. \textcolor{white}{l}} \colim{}^{Q(c_{P^{\op}})} X,
     \end{align*}
     the \ti{filtered homotopy dg-limit of $X'$}\index{filtered homotopy dg-(co)limit}\index{filtered homotopy dg-(co)limit!closed under filtered homotopy dg-(co)limits} (resp. \ti{filtered homotopy dg-colimit of $X$}) \cite{FrancescoPrivate}.
    % \begin{itemize}
    %     \item[-] $\varprojlim^{c_{P^{\op}}} X'$ the \ti{filtered homotopy dg-limit of $X'$};    \index{filtered homotopy dg-(co)limit}\index{filtered homotopy dg-(co)limit!closed under filtered homotopy dg-(co)limits}
    %     \item[-] $\colim^{c_{P^{\op}}} X$ the \ti{filtered homotopy dg-colimit of $X$} \cite{FrancescoPrivate}.
    % \end{itemize}
    We say that $\A$ is \emph{closed under filtered homotopy dg-colimits} (resp. \emph{closed under filtered homotopy dg-limits}) if these exist for all ind-dg-objects in $\A$ (resp. pro-dg-objects in $\A$).
\end{definition}

\begin{remark} \label{remark - filtered homotopy dg-colimits}
    Analogous to \Cref{remark - yoneda filtered colimit exists}, we can compose a diagram $X: \mc{I} \to \A$ with the dg-Yoneda embedding $\Y$ so that it becomes a diagram in $\dgm(\A)$. The filtered homotopy (dg-)colimit then always exists. Note that if $X$ is an ind-dg-object in $\A$ and its filtered homotopy dg-colimit exists, it is the object representing
    $$\varprojlim{}^{Q(c_{P^{\op}})} \dgm(\A)(X,-)$$
    because $\Y$ reflects limits. Similarly, if $X$ is a pro-dg-object in $\A$ and its filtered homotopy dg-limit exists, it is the object representing
    $$\varprojlim{}^{Q(c_{P^{\tn{op}}})} \dgm(\A)(-,X).$$
\end{remark}

%Should I say something about Lurie and Antieau?
% Lurie's ind-objects are defined in HTT \S 5.3
% Lurie's ind-completion is defined in HTT \S 5.3.5.1
% Lurie's universal property is HTT Prop 5.3.5.10
% Lurie's ind completion remains stable is DAGI Prop 4.5
% Lurie's ind completion inherits t-structures + properties thereof in SAG Lemma C.2.4.3
% or Antieau's version in \cite{Antieau}[Proposition 2.13]

Note that, up to homotopy equivalence, the filtered homotopy dg-(co)limits of \Cref{def - filtered homotopy dg-colimit} are independent of the choice of h-projective resolution. In order to compare homotopy colimits with the corresponding ordinary weighted colimits, we make use of the following observations:

\begin{lemma}\label{lemmahflat}
     The constant weight $c_{P^{\op}}$ is h-flat in $\dgm(R\{P\}^{\op})$. 
\end{lemma}

\begin{proof}
    Consider an acyclic dg-$R\{P\}^{\op}$-module $N$. Then $$c_{P^{\op}} \otimes_{R\{P\}} N = \colim{}^{c_{P^{\op}}} N.$$
            Since filtered colimits are exact in $\mc{Z}^{0}(\dgm(R))$, we have that
            \begin{equation*}
                H^{k}(c_{P^{\op}} \otimes_{R\{P\}} N) \cong \colim_{i \in R\{P\}} H^{k}(N_{i}) = 0. \qedhere
            \end{equation*}
            \end{proof}

%The following lemma addresses the unimportance of this choice by showing that filtered homotopy dg-colimits are determined up to homotopy equivalence.

\begin{corollary} \label{lemma - independency hocolim}
    Consider an ind-dg-object $X: R\{P\} \to \A$.
    %Given resolutions $Q(c_{P^{\op}}) \to c_{P^{\op}}$ and $Q'(c_{P^{\op}}) \to c_{P^{\op}}$, $$Q(c_{P^{\op}}) \otimes_{R\{P\}} X \approx Q'(c_{P^{\op}}) \otimes_{R\{P\}} X,$$ if the weighted colimits exist.
    Given a resolution $Q(c_{P^{\op}}) \to c_{P^{\op}}$, we have a quasi-isomorphism $$Q(c_{P^{\op}}) \otimes_{R\{P\}} X \sim c_{P^{\op}} \otimes_{R\{P\}} X,$$ if the weighted colimits exist.
\end{corollary}
\begin{proof}
            %This follows from the fact that resolutions are homotopically equivalent and additive functors preserve homotopy equivalences.
            %We start by showing that $c_{P^{\op}}$ is h-flat in $\dgm(R\{P\}^{\op})$. Consider hereto an acyclic dg-$R\{P\}^{\op}$-module $N$. Then $$c_{P^{\op}} \otimes_{R\{P\}} N = \colim{}^{c_{P^{\op}}} N.$$
            %Since filtered colimits are exact in $\mc{Z}^{0}(\dgm(R))$, we have that
            %$$H^{k}(c_{P^{\op}} \otimes_{R\{P\}} N) \cong \colim_{i \in R\{P\}} H^{k}(N_{i}) = 0.$$
            %So $c_{P^{\op}}$ is indeed h-flat. We then obtain the desired quasi-isomorphism as the composition of the following quasi-isomorphisms for 
            For a h-flat resolution $Q(X) \to X$, we have
            \begin{align*}
                Q(c_{P^{\op}}) \otimes_{R\{P\}} X &\sim Q(c_{P^{\op}}) \otimes_{R\{P\}} Q(X) \\
                &\sim c_{P^{\op}} \otimes_{R\{P\}} Q(X) \\
                &\sim c_{P^{\op}} \otimes_{R\{P\}} X. \qedhere
            \end{align*}
    \end{proof}

Let $\A$ be an essentially small dg-category. As before, we can consider the dg-category $\Ind\Ob^{\dg}(\A)$ of ind-dg-objects in $\A$\index{filtered homotopy dg-(co)limit!ind-dg-object!$\Ind\Ob^{\dg}(\A)$}, which came with a natural fully faithful dg-functor $\Ll_{\A}$ that mapped an ind-dg-object $X$ in $\A$ to the filtered dg-colimit of $X$. Similarly, we have the following dg-functor\index{homotopy ind-dg-completion!$\Ll_{\A}^{Q}$}
\[\begin{tikzcd}[row sep = small]
	{\Ll_{\A}^{Q}:\Ind\Ob^{\dg}(\A)} & {\dgm(\A): X} & {\Ll_{\A}^{Q}(X) := \colim {}^{Q(c_{P_{X}^{\op}})} \Y \circ X,}
	\arrow[from=1-1, to=1-2]
	\arrow[maps to, from=1-2, to=1-3]
\end{tikzcd}\] 
where the resolution $Q(c_{P_{X}^{\op}})$ is the bar resolution. Note that it is no longer fully faithful since 
\begin{equation} \label{eq - LAQ ff}
    \dgm(\A)(\Ll_{\A}^{Q}(X), \Ll_{\A}^{Q}(X')) \cong \varprojlim { }^{Q(c_{P_{X}^{\op}})} \colim { }^{Q(c_{P_{X'}^{\op}})} \A(X_{i},X'_{j}).
\end{equation}

%In analogy with \S\ref{subsection - enriched ind-completions}, we wish to define `the homotopy ind-dg-completion' as the essential image of $\Ll_{\A}^{Q}$. By \Cref{thm - homotopy dg completion closed} and \Cref{cor - universal prop inddgho}, this will require the $\A$ is strongly pretriangulated. So we continue under this assumption.

Motivated by \Cref{thm - homotopy dg completion closed} and \Cref{cor - universal prop inddgho}, we restrict to strongly pretriangulated dg-categories in the following definition.

\begin{definition}
    The \ti{homotopy ind-dg-completion} of an essentially small strongly pretriangulated dg-category $\A$, denoted by $\Ind^{\dg,Q}(\A)$, is the essential image of $\Ll_{\A}^{Q}$.\index{homotopy ind-dg-completion}\index{homotopy ind-dg-completion!$\Ind^{\dg,Q}(\A)$}
\end{definition}

\begin{remark} \label{remark - inddg and inddgho}
    $\Ind^{\dg,Q}(\A)$ is quasi-equivalent to the quasi-essential image of $\Ll_{\A}^{Q}$ and, by \Cref{lemma - independency hocolim} (1), the latter contains the filtered homotopy dg-colimits for all possible resolutions of the weight. In fact, it is $\hproj(\A)$, see \Cref{thm - homotopy dg completion closed}.
\end{remark}

We again have an embedding of $\A$ into $\Ind^{\dg,Q}(\A)$ given by Yoneda, which makes the following diagram commute:
        \begin{equation} \label{eq - cdiagram embedding LAQ} \begin{tikzcd}[row sep = small]
	\A & {\Ind\Ob^{\dg}(\A)} & {\Ind^{\dg,Q}(\A) \subseteq \dgm(\A).}
	\arrow["{\iota_{A}}", hook, from=1-1, to=1-2]
	\arrow["\Ll_{\A}^{Q}", from=1-2, to=1-3]
	\arrow["\Y"', hook, curve={height=12pt}, from=1-1, to=1-3]
\end{tikzcd}\end{equation}
This requires remarking that if $P$ is a singleton (or more general discrete), then $c_{P}$ is already h-projective\footnote{For $P = \coprod_{S} \{*\}$ and $N: R\{P\} \to \dgm(R)$ acyclic, $\dgm(R\{P\}^{\op})(c_{P},N) \cong \coprod_{i \in S} N_{i}$ is again acyclic because filtered colimits are exact in $\mc{Z}^{0}(\dgm(R))$.} as a dg-$R\{P\}^{\op}$-module. As desired, we have the following result.

\begin{thm} \label{thm - homotopy dg completion closed}
    Let $\A$ be an essentially small strongly pretriangulated dg-category. The following hold:
    \begin{enumerate}
        \item We have that $\Ind^{\dg,Q}(\A) \subseteq \hproj(\A)$;
        \item The embedding $\Ind^{\dg,Q}(\A) \hookrightarrow \hproj(\A)$ is a quasi-equivalence;
        \item $\Ind^{\dg,Q}(\A)$ is pretriangulated;
          \item $\Ind^{\dg,Q}(\A)$ is closed under filtered homotopy dg-colimits up to homotopy equivalence and $\hproj(\A)$ is strictly closed under them.
    \end{enumerate}
    \begin{proof}
    For (1), we start by remarking that for $P = \{\ast\}$ and the constant diagram $c_A: R\{P\} \rightarrow \A:\ast \mapsto A$, the weight $c_P$ is h-projective and we have a homotopy equivalence
    $\Y(A) = \Ll_{\A}(c_{A}) \sim \Ll_{\A}^{Q}(c_{A})$. 
    Next, we show (for (2) and (4)) that a filtered homotopy dg-colimit of a diagram $X: R\{P\} \to \hproj(\A)$ remains h-projective. Consider hereto an acyclic dg-$\A$-module $N$. Then $$\dgm(\A)(\colim{}^{Q(c_{P^{\op}})} X, N) \cong \varprojlim{}^{Q(c_{P^{\op}})} \dgm(\A)(X_{i},N).$$ So we found a filtered weighted limit of acyclics. Using the inner hom-description of \Cref{prop - tensor inner}, this weighted limit rewrites as
        \begin{align*}
            \varprojlim{}^{Q(c_{P^{\op}})} \dgm(\A)(X_{i},N) = \dgm(R\{P\}^{\op})(Q(c_{P^{\op}}),\dgm(\A)(X_{i},N)).
        \end{align*}
        Since $Q(c_{P^{\op}})$ is h-projective, this is again acyclic, rendering $\colim{}^{Q(c_{P^{\op}})} X$ h-projective. Consequently, we have the embedding $$\Ind^{\dg,Q}(\A) \hookrightarrow \hproj(\A).$$
        There remains to be shown that this is quasi-essentially surjective to conclude (2). So let $M \in \hproj(\A)$. Since $\hproj(\A)$ is quasi-equivalent to $\SF(\A)$, the dg-category of semi-free dg-$\A$-modules, $M$ is homotopy equivalent to a semi-free dg-$\A$-module. So we may just assume that $M$ is semi-free. Then there is a filtration
        $$0 \subseteq F_{0} \subseteq F_{1} \subseteq ... \subseteq M$$
        of dg-$\A$-modules so that $M = \bigcup_{i \geq 0} F_{i}$ and $F_{i}/F_{i-1} \cong \coprod_{j \in \Omega_{i}} \Y(A_{j})[k_{j}]$ for all $i \geq 0$ (we set $F_{-1} = 0$). We will prove by induction on $i$ that each $F_{i}$ is a filtered colimit of perfect dg-$\A$-modules. 
        \begin{itemize}
            \item $i = 0$. Since $F_{0} = \coprod_{j \in \Omega_{0}} \Y(A_{j})[k_{j}]$ and coproducts are filtered colimits (the poset consists of the finite subsets of $\Omega_{0}$), it is already a filtered colimit of perfect dg-$\A$-modules.

            \item $i-1 \Rightarrow i$. Consider the short exact sequence
            \[\begin{tikzcd}[row sep = small]
	{F_{i-1}} & {F_{i}} & {F_{i}/F_{i-1}.}
	\arrow["\pi", two heads, from=1-2, to=1-3]
	\arrow["\gamma", hook, from=1-1, to=1-2]
\end{tikzcd}\]
            Recall that $F_{i-1} = \coprod_{l \in \Omega_{i-1}} \Y(B_{l})[h_{l}]$ and $F_{i}/F_{i-1}=\coprod_{j \in \Omega_{i}} \Y(A_{j})[k_{j}]$. Fix some $j \in \Omega_{i}$ and consider the closed, degree 0 monomorphism
\[\begin{tikzcd}[row sep = small]
	{ F_{i-1}} & {F_{i}} & {F_{i}/F_{i-1}} \\
	&& {\Y(A_{j})[k_{j}].}
	\arrow["\pi", two heads, from=1-2, to=1-3]
	\arrow["\gamma", hook, from=1-1, to=1-2]
	\arrow["f"', hook, from=2-3, to=1-3]
\end{tikzcd}\]
            By Yoneda, $f \in \mc{Z}^{-k_{j}}F_{i}/F_{i-1}(A_{j})$. Since $\pi_{A_{j}}$ is epi and of degree zero, we find a preimage $\tilde{f} \in F_{i}^{-k_{j}}(A_{j})$, which corresponds to another monomorphism of degree zero:
\[\begin{tikzcd}[row sep = small]
	{ F_{i-1}} & {F_{i}} & {F_{i}/F_{i-1}} \\
	&& {\Y(A_{j})[k_{j}].}
	\arrow["\pi", two heads, from=1-2, to=1-3]
	\arrow["\gamma", hook, from=1-1, to=1-2]
	\arrow["f"', hook, from=2-3, to=1-3]
	\arrow["{\tilde{f}}", hook, from=2-3, to=1-2]
\end{tikzcd}\]
            Note that $\tilde{f}$ will in general not be closed, i.e. $d_{F_{i}} \circ \tilde{f} - \tilde{f} \circ d_{\Y(A_{j})[k_{j}]}$ might differ from zero. Still, $\pi$ and $f$ are closed so that
            \begin{align*}
                \pi \circ (d_{F_{i}} \circ \tilde{f} - \tilde{f} \circ d_{\Y(A_{j})[k_{j}]}) &= d_{F_{i}} \circ \pi \circ \tilde{f} - \pi \circ \tilde{f} \circ d_{\Y(A_{j})[k_{j}]} \\
                &= d_{F_{i}} \circ f - f \circ d_{\Y(A_{j})[k_{j}]} \\
                &= 0.
            \end{align*}
            We conclude that $d_{F_{i}(A_{j})}(\tilde{f}) \in \Kernel(\pi)$. By exactness and the fact that $\gamma$ is of degree zero, we then find an element $x \in F_{i-1}^{-k_{j}+1}(A_{j})$ so that $\gamma(x) = d_{F_{i}(A_{j})}(\tilde{f})$. Since $\gamma$ is mono and
            $$\gamma_{A_{j}}(d_{F_{i-1}}(x)) = d_{F_{i}}(\gamma_{A_{j}}(x)) = d_{F_{i}}^{2}(\tilde{f}) = 0,$$ $x$ lies in $\mc{Z}^{-k_{j}+1}F_{i-1}(A_{j})$. It then corresponds to a closed morphism $\Y(A_{j})[k_{j}] \to F_{i-1}$ of degree $1$ so that $\gamma \circ x = d_{F_{i}} \circ f - f \circ d_{\Y(A_{j})[k_{j}]}$. As $F_{i-1} = \coprod_{l \in \Omega_{i-1}} \Y(A_{l})[k_{l}]$, we then find some finite subset $W_{j} \subseteq \Omega_{i-1}$ so that $x$ factors over $\coprod_{l \in W_{j}} \Y(B_{l})[h_{l}]$. This is depicted in the following commutative diagram,
            \[\begin{tikzcd}[row sep = small]
	{ F_{i-1}} & {F_{i}} & {F_{i}/F_{i-1}} \\
	{\coprod_{l \in W_{j}} \Y(B_{l})[h_{l}]} && {\Y(A_{j})[k_{j}].} \\
	& {\Y(A_{j})[k_{j}]}
	\arrow["\pi", two heads, from=1-2, to=1-3]
	\arrow["\gamma", hook, from=1-1, to=1-2]
	\arrow["f"', hook, from=2-3, to=1-3]
	\arrow["{\tilde{f}}", hook, from=2-3, to=1-2]
	\arrow["{d_{F_{i}} \circ \tilde{f} - \tilde{f} \circ d_{\Y(A_{j})[k_{j}]}}"'{pos=0.2}, from=3-2, to=1-2]
	\arrow["\mu", hook', from=2-1, to=1-1]
	\arrow["x", from=3-2, to=2-1]
\end{tikzcd}\]
            We now define the dg-$\A$-module $M_{j}$ as
            $$\left( \Y(A_{j})[k_{j}] \oplus \coprod_{l \in W_{j}} \Y(B_{l})[h_{l}],  \begin{pmatrix} d_{\Y(A_{j})[k_{j}]} & 0 \\ x & d_{\coprod_{l \in W_{j}} \Y(B_{l})[h_{l}]} \end{pmatrix} \right).$$
            The degree zero monomorphisms $\tilde{f}$ and $\gamma \circ \mu$ induce a degree zero monomorphism $M_{j} \hookrightarrow F_{i}$ and we claim that this is a closed map. Indeed, we have that
            \begin{align*}
                &(\tilde{f}, \gamma \circ \mu) \circ \begin{pmatrix} d_{\Y(A_{j})[k_{j}]} & 0 \\ x & d_{\coprod_{l \in W_{j}} \Y(B_{l})[h_{l}]} \end{pmatrix} \\
                = \tn{ }&(\tilde{f} \circ d_{\Y(A_{j})[k_{j}]} + \gamma \circ \mu \circ x, \gamma \circ d_{\coprod_{l \in W_{j}} \Y(B_{l})[h_{l}]} ) \\
                = \tn{ }&(d_{F_{i}} \circ \tilde{f}, d_{F_{i}} \circ \gamma).
            \end{align*}
            So we conclude that each $M_{j}$ is a dg-submodule of $F_{i}$ and by construction, it is perfect. Given a finite set $S \subseteq \Omega_{i}$, we then define $M_{S} := \coprod_{j \in S} M_{j}$, which is again a perfect dg-submodule of $F_{i}$. This yields a monomorphism from the filtered colimit $$\xi: \colim_{S \subseteq^{\tn{fin}} \Omega_{i}} M_{S} \hookrightarrow F_{i}.$$
            We also define for $l \in \Omega_{i-1}$ the perfect dg-submodule $N_{l} := \Y(B_{l})[h_{l}]$ and for each finite set $T \subseteq \Omega_{i-1}$, $N_{T} := \coprod_{l \in T} N_{l}$. This yields another monomorphism $$\chi: F_{i-1} = \colim_{T \subseteq^{\tn{fin}} \Omega_{i-1}} N_{T} \hookrightarrow F_{i}.$$
            The monomorphisms $\xi$ and $\chi$ induce a monomorphism
            $$\colim_{T \subseteq^{\tn{fin}} \Omega_{i-1}} N_{T} \oplus \colim_{S \subseteq^{\tn{fin}} \Omega_{i}} M_{S} \hookrightarrow F_{i}.$$
            Since filtered colimits commute with finite limits, the source is a filtered colimit of objects in $\Perf(\A)$. The map is moreover an epimorphism, because $F_{i} \cong F_{i-1} \oplus F_{i}/F_{i-1}$ as graded $\A^{\#}$-modules. We conclude that $F_{i}$ is a filtered colimit of perfect dg-$\A$-modules.
        \end{itemize}

        So each $F_{i}$ is a filtered colimit of perfect dg-$\A$-modules and $M$ is a filtered colimit of the $F_{i}$'s. From the proof of \Cref{thm - enriched completion closed} for $\V = \dgm(R)$ then follows that $M$ is a filtered colimit of perfect dg-$\A$-modules. Since perfect dg-$\A$-modules are finite extensions of finite coproducts of shifts of representables and $\A$ is strongly pretriangulated, we can pull all these operations through Yoneda and find that each perfect dg-$\A$-modules is isomorphic to a representable dg-$\A$-module (i.e. we totalise as in \cite{BondalKapranov}). Consequently, $M$ is isomorphic to a filtered colimit of representables, say $\colim{}^{c_{P^{\op}}} \Y \circ X$. By \Cref{lemma - independency hocolim} (2), there is a quasi-isomorphism
        \begin{equation*} %\label{eq - qis 3}
            \colim{}^{c_{P^{\op}}} \Y \circ X \sim \colim{}^{Q(c_{P^{\op}})} \Y \circ X = \Ll_{\A}^{Q}(X).
        \end{equation*}
        So we found a quasi-isomorphism $M \to \Ll_{\A}^{Q}(X)$. Since both dg-$\A$-modules are h-projective, this is a homotopy equivalence. We conclude that any h-projective dg-$\A$-module is homotopy equivalent to an object in the essential image of $\Ll_{\A}^{Q}(-)$. Finally, (3) and the first part of (4) follow from (2) and the fact that $\hproj(\A)$ is strongly pretriangulated.
    \end{proof}
\end{thm}

\begin{remark} \label{remark - strongly pretriang}
    We needed the strongly pretriangulatedness of $\A$ in \Cref{thm - homotopy dg completion closed} in order for the totalisation of perfect dg-modules to be $\mc{Z}^{0}$-functorial. For a general pretriangulated dg-category $\A$, extra homotopies are involved and it is unclear how to deal with them.
\end{remark}

\Cref{thm - homotopy dg completion closed} immediately yields

\begin{corollary} \label{cor - inddgQ compactly generated}
    Let $\A$ be an essentially small strongly pretriangulated dg-category. Then $\Ind^{\dg,Q}(\A)$ is a dg-enhancement of the derived category $\D(\A)$, i.e. there is an equivalence of triangulated categories $$H^{0}(\Ind^{\dg,Q}(\A)) \cong \D(\A),$$ and $H^{0}(\Ind^{\dg,Q}(\A))$ is compactly generated by the Yoneda image of $\A$.
\end{corollary}

\begin{remark} \label{remark - qeq models}
    So we found two quasi-equivalent dg-enhancements of the derived category $\D(\A)$ of an essentially small strongly pretriangulated dg-category $\A$,
    $$\Ind^{\dg,Q}(\A) \hookrightarrow \hproj(\A).$$
    Throughout this paper, we will use these quasi-equivalent models interchangeably: the smaller $\Ind^{\dg,Q}(\A)$ is the better choice when defining a functor out of it since objects have a concrete description, but it is less flexible because it is only closed under filtered homotopy dg-colimits up to homotopy equivalence by \Cref{thm - homotopy dg completion closed} (4). That is why the larger dg-category $\hproj(\A)$ is the better choice when mapping into it, because it is actually closed under filtered homotopy dg-colimits, again by \Cref{thm - homotopy dg completion closed} (4).
\end{remark}

%\begin{remark} \label{prop - classic hocolim}
    %In a $H^{0}(R)$-linear triangulated category $\mc{T}$ with all small coproducts (which are examples of filtered colimits), one can also consider sequential homotopy colimits, see for example \cite{NeemanBook}[\S 1.6]. This homotopy colimit of a diagram $Z: H^{0}(R)\{\mb{N}\} \to \mc{T}$ is usually denoted by $\hocolim Z_{n}$. If $\mc{T} = \D(\A)$ for $\A$ strongly pretriangulated and $X: R\{\mb{N}\} \to \hproj \A$ an ind-dg-object, then the filtered homotopy dg-colimit of $H^{0}(X)$ also exists and we claim that the two are isomorphic in $\D(\A)$. To see this, observe that $$\colim{}^{c_{\mb{N}^{\op}}} X = \Coker(\id - \tn{shift}),$$ where $\id - \tn{shift}$ is the first morphism in the exact triangle that defines $\hocolim H^{0}(X)_{n}$. Since a short exact sequence of dg-$\A$-modules gives rise to an exact triangle in $\D(\A)$, $\colim{}^{c_{\mb{N}^{\op}}} X$ and $\hocolim H^{0}(X)_{n}$ are isomorphic in $\D(\A)$. The claim then follows from \Cref{lemma - independency hocolim} (2).
%\end{remark}

Another corollary of \Cref{thm - homotopy dg completion closed} is the universal property.

\begin{corollary} \label{cor - universal prop inddgho}
    Let $\A$ be an essentially small strongly pretriangulated dg-category. For any dg-functor $F: \A \to \B$ where $\B$ has all filtered homotopy colimits of diagrams $F \circ X$ for an ind-dg-object $X$ in $\A$, there exists a unique dg-functor $\Ind^{\dg,Q}(F)$ that makes the diagram below commute,\index{homotopy ind-dg-completion!$\Ind^{\dg,Q}(F)$}
    \begin{equation} \label{eq - universal prop enriched} \begin{tikzcd}[row sep=small]
	\A & {\B.} \\
	{\Ind^{\dg,Q}(\A)}
	\arrow["F", from=1-1, to=1-2]
	\arrow["{\Ind^{\dg,Q}(F)}"', pos=0.3, shift right=2, dashed, from=2-1, to=1-2]
	\arrow["{\Y}"', hook, from=1-1, to=2-1]
\end{tikzcd}\end{equation}
    \begin{proof}
        The following defines a dg-functor,
        \[\begin{tikzcd}[row sep = small]
	{\Ind^{\dg,Q}(F):\Ind^{\dg,Q}(\A)} & {\B: \Ll_{\A}^{Q}(X)} & {\colim{}^{Q(c_{P_{X}^{\op}})} F \circ X.} \\
	& {\textcolor{white}{\B:}\Ll_{\A}^{Q}(X')} & {\colim{}^{Q(c_{P_{X'}^{\op}})} F \circ X',}
	\arrow[from=1-1, to=1-2]
	\arrow[maps to, from=1-2, to=1-3]
	\arrow[from=1-2, to=2-2]
	\arrow[maps to, from=2-2, to=2-3]
	\arrow[shift left=5, from=1-3, to=2-3]
\end{tikzcd}\]
    and it is the unique one to make diagram \eqref{eq - universal prop enriched} commute.
    \end{proof}
\end{corollary}

\begin{corollary} \label{cor - D(F)}
    Let $\A$ and $\B$ be essentially small strongly pretriangulated dg-categories. For any dg-functor $F: \A \to \B$, there exists a unique dg-functor $\D(F)$ that makes the diagram below commute,\index{homotopy ind-dg-completion!$\D(F)$}
    \[\begin{tikzcd}[row sep =small]
	\A & \B \\
	{\Ind^{\dg,Q}(\A)} & {\Ind^{\dg,Q}(\B).}
	\arrow["F", from=1-1, to=1-2]
	\arrow["\Y"', hook, from=1-1, to=2-1]
	\arrow["\Y", hook, from=1-2, to=2-2]
	\arrow["{\D(F)}", from=2-1, to=2-2]
\end{tikzcd}\]
    If we use the quasi-equivalent models $\hproj(\A)$ and $\hproj(\B)$, then $\D(F)$ corresponds to the left adjoint to $F_{*}: \hproj(\B) \to \hproj(\A)$, namely
    $$\hproj(\A) \to \hproj(\B): M \mapsto M \otimes_{\A} \B(-,F(-)).$$
    \begin{proof}
        The first part of the statement follows from \Cref{cor - universal prop inddgho}. For the second part, it suffices to remark that 
        \begin{equation*}
            \A(-,A) \otimes_{\A} \B(-,F(-)) \cong \B(-,F(A)). \qedhere
        \end{equation*}
    \end{proof}
\end{corollary}

\begin{remark} \label{remark - not strongly pretriangulated}
     If $\A$ is pretriangulated but not strongly pretriangulated, it is unclear whether the quasi-essential image of $\Ll_{\A}^{Q}(-)$ is $\hproj(\A)$, so we do not want to name it the homotopy ind-dg-completion of $\A$. However, we can always consider the quasi-equivalence $\A \hookrightarrow \pretr(\A)$ as in \cite{BondalKapranov}, where the latter is strongly pretriangulated by \cite{BondalKapranov}[\S 3 Proposition 1]. Then we still have an embedding $\A \hookrightarrow \Ind^{\dg,Q}(\pretr(\A))$, which is all we were after. Since we are only interested in pretriangulated $t$-dg-categories up to quasi-equivalence because their derived deformations agree, this suffices.
\end{remark}

% 1.4 ----------------------------------------------------------------
\subsection{About \texorpdfstring{$t$}{t}-structures} \label{subsection - t-structures}

Let $\mc{T}$ be a triangulated category endowed with a $t$-structure $(\mc{T}_{\leq 0},\mc{T}_{\geq 0})$\index{$t$-structure $(\mc{T}_{\leq 0},\mc{T}_{\geq 0})$} with as heart $\mc{T}^{\heartsuit}$\index{$t$-structure $(\mc{T}_{\leq 0},\mc{T}_{\geq 0})$!heart $\mc{T}^{\heartsuit}$}. In this subsection, we will fix notation and mention some properties and results regarding $t$-structures. See \cite{BBD}[\S 1.3] for more background on $t$-structures.

For an integer $i$, we denote the $i$-th cohomology functor induced by the $t$-structure by
$$H^{i}_{t} = [i]\tau_{\leq i}\tau_{\geq i}: \mc{T} \to \mc{T}^{\heartsuit},$$
where $\tau_{\leq i}$ and $\tau_{\geq i}$ are the truncation functors\index{$t$-structure $(\mc{T}_{\leq 0},\mc{T}_{\geq 0})$!truncation functors $\tau_{\leq n}, \tau_{\geq n}$}\index{$t$-structure $(\mc{T}_{\leq 0},\mc{T}_{\geq 0})$!$i$-th cohomology $H^{i}_{t}$}. There are inclusions
\begin{align} \label{eq - inclusions}
    \mc{T}_{\leq n} \subseteq \{ X \in \mc{T} \mid H^{i}_{t}(X) = 0, \forall i > n \}, \\
    \mc{T}_{\geq n} \subseteq \{ X \in \mc{T} \mid H^{i}_{t}(X) = 0, \forall i < n \}. \nonumber
\end{align}
We say that the $t$-structure $(\mc{T}_{\leq 0},\mc{T}_{\geq 0})$ is \ti{left} (resp. \ti{right}) \ti{separated}\index{$t$-structure $(\mc{T}_{\leq 0},\mc{T}_{\geq 0})$!separated, left/right} if 
\begin{align*}
    &\bigcap_{n \geq 0} \mc{T}_{\leq -n} = 0, \\
    \tn{resp.} \quad &\bigcap_{n \geq 0} \mc{T}_{\geq n} = 0.
\end{align*}
We call $(\mc{T}_{\leq 0},\mc{T}_{\geq 0})$ \ti{non-degenerate}\index{$t$-structure $(\mc{T}_{\leq 0},\mc{T}_{\geq 0})$!(non-)degenerate} if it is both left and right separated. In this case, an object $X$ is zero if and only if $H^{i}_{t}(X) = 0$ for all $i \in \mb{Z}$ and the inclusions of \eqref{eq - inclusions} are equalities \cite{BBD}[Proposition 1.3.7].

Given another triangulated category $\mc{T}'$ with a $t$-structure $(\mc{T}'_{\leq 0},\mc{T}'_{\geq 0})$ and a triangulated functor $F: \mc{T} \to \mc{T}'$, we call $F$ \ti{left} (resp. \ti{right}) \ti{$t$-exact}\index{$t$-structure $(\mc{T}_{\leq 0},\mc{T}_{\geq 0})$!$t$-exact, left/right} if 
\begin{align*}
    &F(\mc{T}_{\leq n}) \subseteq \mc{T}'_{\leq n}, \\
    \tn{resp.} \quad &F(\mc{T}_{\geq n}) \subseteq \mc{T}'_{\geq n}.
\end{align*}
If $F$ is both right and left $t$-exact, we call $F$ \ti{$t$-exact}. This is equivalent to $F$ commuting with the truncation functors, see \cite{BBD}[after 1.3.19]. Recall also from \cite{BBD}[Proposition 1.3.17 (iii)] that for an adjoint pair of triangulated functors $(F,G)$, $F$ is left $t$-exact if and only if $G$ is right $t$-exact.

There are the following full subcategories of $\mc{T}$,
\begin{align*}
    \mc{T}^{-} &= \bigcup_{n \geq 0} \mc{T}_{\leq n}, \\
    \mc{T}^{+} &= \bigcup_{n \geq 0} \mc{T}_{\geq -n}, \\
    \mc{T}^{b} &= \mc{T}^{-} \cap \mc{T}^{+}.
\end{align*}
These are strictly full triangulated subcategories of $\mc{T}$ that are closed under direct summands and the $t$-structure on $\mc{T}$ induces $t$-structures on $\mc{T}^{-}, \mc{T}^{+}$ and $\mc{T}^{b}$ so that the inclusions into $\mc{T}$ are $t$-exact. In particular, the hearts of all three $t$-structures coincide. We then call a $t$-structure $(\mc{T}_{\leq 0},\mc{T}_{\geq 0})$ on $\mc{T}$ \ti{left} (resp. \ti{right} and simply) \ti{bounded} if the inclusion $\mc{T}^{+} \hookrightarrow \mc{T}$ (resp. $\mc{T}^{-} \hookrightarrow \mc{T}$ and $\mc{T}^{b} \hookrightarrow \mc{T}$) is an equivalence.\index{$t$-structure $(\mc{T}_{\leq 0},\mc{T}_{\geq 0})$!bounded, left/right}

As illustrated by \S \ref{subsection - homotopy ind-dg-completion}, we will often consider homotopy (co)limits inside pretriangulated dg-categories with a $t$-structure (on the homotopy category). In a triangulated category $\mc{T}$ with countable coproducts, there is already a notion of (sequential) homotopy (co)limit, see \cite{NeemanBook}[\S 1.6]: Given a covariant functor $X: \mathbb{N} \to \mc{T}$ (resp. contravariant functor $X': \mathbb{N} \to \mc{T}$), its \ti{(sequential) homotopy colimit} (resp. \ti{(sequential) homotopy limit})\index{homotopy (co)limit (sequential)}, denoted by $\hocolim_{n \geq 0} X_{n}$ (resp. $\holim_{n \geq 0} X'_{n}$), is defined via the exact triangle
\begin{align} \label{eq - triangle Neeman}
    &\coprod_{n \geq 0} X_{n} \xrightarrow{1 - \tn{shift}} \coprod_{n \geq 0} \rightarrow \hocolim_{n \geq 0} X_{n} \xrightarrow{+} \quad, \\
    \tn{resp.} \quad &\Pi_{n \geq 0} X'_{n} \xrightarrow{1 - \tn{shift}} \Pi_{n \geq 0} X'_{n} \rightarrow \holim_{n \geq 0} X'_{n} \xrightarrow{+} \quad, \nonumber
\end{align}
where $\tn{shift}$ is induced by $\{X(n \leq n+1) \}_{n \geq 0}$ (resp. $\{X'(n \leq n+1) \}_{n \geq 0}$). Note that the homotopy colimit is unique up to non-canonical isomorphism. Let us justify the overlap in the use of the term homotopy colimit:

\begin{proposition} \label{prop - classic hocolim}
    Let $\A$ be an essentially small $R$-linear pretriangulated dg-category. For an ind-dg-object $X: R\{\mb{N}\} \to \hproj(\A)$, there is an isomorphism in $\D(\A)$,
    \begin{equation} \label{eq - classic hocolim}
        \colim{}^{Q(c_{\mb{N}^{\op}})} X \cong \hocolim_{n \geq 0} X_{n}.
    \end{equation}
    \begin{proof}
        Observe that in $\D(\A)$, $$\colim{}^{c_{\mb{N}^{\op}}} X = \Coker(1 - \tn{shift}),$$ where $1 - \tn{shift}$ is the first morphism in the exact triangle \eqref{eq - triangle Neeman} that defines the homotopy colimit of $X$. Using \Cref{lemma - independency hocolim} and the fact that a short exact sequence of dg-$\A$-modules gives rise to an exact triangle in $\D(\A)$, we find that
        \begin{equation*}
            \colim{}^{Q(c_{\mb{N}^{\op}})} X \cong \hocolim_{n \geq 0} X_{n}. \qedhere
        \end{equation*}
    \end{proof}
\end{proposition}

% Of course, a dual result holds for pro-dg-objects and sequential homotopy limits. 
Recall from \cite{GLVdB1}[Definition 3.8] that a $t$-structure $(\mc{T}_{\leq 0},\mc{T}_{\geq 0})$ is said to be \ti{closed under countable coproducts} (resp. \ti{products}) if the aisle $\mc{T}_{\leq 0}$ (resp. coaisle $\mc{T}_{\geq 0}$) is closed under countable coproducts (resp. products). Note that this entails both the existence in $\T$ and the inclusion into the coaisle (resp. aisle).\index{$t$-structure $(\mc{T}_{\leq 0},\mc{T}_{\geq 0})$!closed under countable (co)products} \index{$t$-structure $(\mc{T}_{\leq 0},\mc{T}_{\geq 0})$!(co)aisle $\T_{\leq 0}$ and $\T_{\geq 0}$}

\begin{remark} \label{remark - closed products}
    The coaisle $\T_{\geq 0}$ is always closed under products if these exist in $\T$ because of the adjunction $\tau_{\geq 0} \dashv \iota_{\geq 0}$ and the fact that right adjoints preserve limits. Similarly, the aisle $\T_{\leq 0}$ is always closed under coproducts if they exist in $\T$.
\end{remark}

If (sequential) homotopy (co)limits exist, it is natural to wonder whether the homotopy (co)limits induced by the truncation functors exist and yield the object itself. Ergo, whether an object in the triangulated category is completely determined by its truncations. We hereto introduce the notions of \ti{left} and \ti{right completeness}.\index{$t$-structure $(\mc{T}_{\leq 0},\mc{T}_{\geq 0})$!complete, left/right}

\begin{definition} \label{def - right complete}
    A triangulated category $\mc{T}$ with a $t$-structure $(\mc{T}_{\leq 0},\mc{T}_{\geq 0})$ is said to be \ti{right complete} if
    \begin{enumerate}
        \item[(1)] the coaisle $\T_{\geq 0}$ is closed under countable coproducts;

        % all sequential homotopy colimits of diagrams with values in a coaisle that stabilize in cohomology, i.e. the homotopy colimit of any diagram $X: \mb{N} \to \T_{\geq 0}$ \ti{stabilizing in cohomology}\index{homotopy (co)limit (sequential)!stabilizing in cohomology} exists. This condition entails that for each $N > 0$, we can find an integer $n_{N} \geq 0$ so that $H^{i}_{t}(X(n \leq n+1))$ induces
        % \begin{equation} \label{eq - stabilizing cond}
        %     H^{i}_{t}(X(n)) \cong H^{i}_{t}(X(n+1)), \quad \forall n \geq n_{N}, \forall i \leq N;
        % \end{equation}
        
        \item[(2)] $H^{0}_{t}: \T \to \T^{\heartsuit}$ preserves countable coproducts; 
        
        \item[(3)] for each object $Z \in \T_{\geq 0}$, we have that 
        \begin{equation} \label{eq - right complete}
            \hocolim_{n \geq 0} \tau_{\leq n}Z \cong Z
        \end{equation} in $\T$ via any map arising from (T3) in
        \begin{equation} \label{eq - T3 hocolim} \begin{tikzcd}[row sep=small]
	{\coprod_{n \geq 0} \tau_{\leq n}Z} & {\coprod_{n \geq 0} \tau_{\leq n}Z} & {\hocolim_{n \geq 0} \tau_{\leq n}Z} & {} \\
	0 & Z & Z & {\quad.}
	\arrow["{1-\tn{shift}}", from=1-1, to=1-2]
	\arrow[from=1-1, to=2-1]
	\arrow[from=1-2, to=1-3]
	\arrow[from=1-2, to=2-2]
	\arrow["{+}", from=1-3, to=1-4]
	\arrow[dashed, from=1-3, to=2-3]
	\arrow[from=2-1, to=2-2]
	\arrow[Rightarrow, no head, from=2-2, to=2-3]
	\arrow["{+}", from=2-3, to=2-4]
\end{tikzcd}\end{equation}
    \end{enumerate}
    By using the aisle $\T_{\leq 0}$, the truncation functors $\tau_{\geq -n}$ and sequential homotopy limits instead, we obtain the dual notion of \ti{left completeness}.
\end{definition}

\begin{remark} \label{remark - H0t hocolim to colim AB5}
    \begin{enumerate}
        \item Note that condition (1) in \Cref{def - right complete} implies that all sequential homotopy colimits of diagrams $X: \mb{N} \to \T_{\geq 0}$ exist.
        \item If in \Cref{def - right complete} we moreover have that $\T^{\heartsuit}$ satisfies AB5, then condition (2) implies that $H^{0}_{t}: \T \to \T^{\heartsuit}$ maps sequential homotopy colimits to colimits, see for example \Cite{NeemanBokstedt}[Remark 2.2]. 
    \end{enumerate}
\end{remark}

\begin{remark}
    In the usual definition of right completeness, see for example \Cite{LurieDAGI}[\S 7] and \cite{NeemanNLC}[\S 1], one requests \eqref{eq - right complete} to hold for all $Z \in \T$, so one replaces $\T_{\geq 0}$ by $\T$ in \Cref{def - right complete}. However, since $\tau_{\leq m-1}(\tau_{\leq n}Z)$ is eventually constant, it is clear that the necessary information to recover $Z$ is gathered in $\tau_{\geq m}Z$. In particular, if all countable coproducts in $\T$ exist and its heart $\T^{\heartsuit}$ satisfies AB5, then \eqref{eq - right complete} holds for all of $\T$ as soon as it holds for $\T_{\geq 0}$ as we now show:

    Consider $Z \in \T$ and a map $\varphi: \hocolim_{n \geq 0} \tau_{\leq n}Z \to Z$ arising as in \eqref{eq - T3 hocolim}. Then we have a morphism of exact triangles,
    \[\begin{tikzcd}[column sep = small, row sep = small, font=\small]
	{\tau_{\leq -1}Z} & Z & {\tau_{\geq 0}Z} & { } \\
	{\tau_{\leq -1}(\hocolim_{n \geq 0} \tau_{\leq n}Z)} & {\hocolim_{n \geq 0} \tau_{\leq n} Z} & {\tau_{\geq 0}(\hocolim_{n \geq 0} \tau_{\leq n}Z)} & {\quad.}
	\arrow[from=1-1, to=1-2]
	\arrow[from=1-2, to=1-3]
	\arrow["{+}", from=1-3, to=1-4]
	\arrow["{ }", from=2-1, to=1-1]
	\arrow[from=2-1, to=2-2]
	\arrow["\varphi"', from=2-2, to=1-2]
	\arrow[from=2-2, to=2-3]
	\arrow["{ }", from=2-3, to=1-3]
	\arrow["{+}", from=2-3, to=2-4]
\end{tikzcd}\]
    The leftmost map is an isomorphism because $\tau_{\leq -1}$ commutes with homotopy colimits. That the rightmost map is an isomorphism follows from the fact that $\T$ satisfies \Cref{def - right complete} and \Cref{remark - H0t hocolim to colim AB5} (2). Then $\varphi$ is also an isomorphism.
    % \Cref{def - right complete} agrees with \cite{NeemanNLC}[\S 1], where left completeness is discussed for a derived category of an abelian category $\D(\mb{A})$. Only condition (3) is mentioned there, but $\D(\mb{A})$ has countable coproducts so (1) holds and the standard $t$-structure (cf. \Cref{ex - standard t-structures}) satisfies (2).
\end{remark}

There is a relation between right separatedness and right completeness (resp. left separatedness and left completeness) as in \cite{LurieDAGI}[Proposition 1.2.1.19].

\begin{proposition} \label{prop - Lurie} 
    Let $\T$ be a triangulated category with a $t$-structure $(\mc{T}_{\leq 0},\mc{T}_{\geq 0})$. Suppose that $\mc{T}_{\geq 0}$ is closed under countable coproducts, the functor $H^{0}_{t}: \T \to \T^{\heartsuit}$ preserves countable coproducts and AB5 holds in the heart $\T^{\heartsuit}$.
    %that homotopy colimits of a countable family of exact triangles form an exact triangle.
%     given diagrams $A,B,C:\mathbb{N} \to \T$ whose coproducts exist in $\T$ and a family of exact triangles in $\T$,
%     $$\{A_{n} \to B_{n} \to C_{n} \xrightarrow{+} \quad \}_{n \in \mb{N}},$$
%     there is an exact triangle
%     $$\hocolim_{n} A_{n} \to \hocolim_{n} B_{n} \to \hocolim_{n} C_{n} \xrightarrow{+} \quad,$$
%     that makes the diagram below commute:
%     \[\begin{tikzcd}[row sep =small]
% 	{\coprod_{n} A_{n}} & {\coprod_{n} B_{n}} & {\coprod_{n} C_{n}} & \quad \\
% 	{\coprod_{n} A_{n}} & {\coprod_{n} B_{n}} & {\coprod_{n} C_{n}} & \quad \\
% 	{\hocolim_{n} A_{n}} & {\hocolim_{n} B_{n}} & {\hocolim_{n} C_{n}} & {\quad.} \\
% 	\quad & \quad & \quad
% 	\arrow[from=1-1, to=1-2]
% 	\arrow[from=1-1, to=2-1]
% 	\arrow[from=1-2, to=1-3]
% 	\arrow[from=1-2, to=2-2]
% 	\arrow["{+}", from=1-3, to=1-4]
% 	\arrow[from=1-3, to=2-3]
% 	\arrow[from=2-1, to=2-2]
% 	\arrow[from=2-1, to=3-1]
% 	\arrow[from=2-2, to=2-3]
% 	\arrow[from=2-2, to=3-2]
% 	\arrow["{+}", from=2-3, to=2-4]
% 	\arrow[from=2-3, to=3-3]
% 	\arrow[from=3-1, to=3-2]
% 	\arrow["{+}", from=3-1, to=4-1]
% 	\arrow[from=3-2, to=3-3]
% 	\arrow["{+}", from=3-2, to=4-2]
% 	\arrow["{+}", from=3-3, to=3-4]
% 	\arrow["{+}", from=3-3, to=4-3]
% \end{tikzcd}\]
    Then the following are equivalent:
    \begin{enumerate}
        \item $\mc{T}$ is right complete for the $t$-structure;
        \item $\bigcap_{n \geq 0} \mc{T}_{\geq n} = 0$.
    \end{enumerate}
    \begin{proof}
        Assume first that $\mc{T}$ is right complete and consider $X \in \bigcap_{n \geq 0} \mc{T}_{\geq n}$. Then $\tau_{\leq n} X = 0$ for all $n \geq 0$. Since $X \in \T_{\geq 0}$ and using \Cref{remark - H0t hocolim to colim AB5} (1), $$X \cong \hocolim_{n \geq 0} \tau_{\leq n} X = 0.$$

        Conversely, assume that (2) holds. Due to our hypotheses, we only need to show condition (3) in \Cref{def - right complete}. Let $X$ be an arbitrary object in $\mc{T}_{\geq 0}$ and $\varphi$ a map arising from (T3) in
        \[\begin{tikzcd}[row sep = small]
	{\coprod_{n \geq 0} \tau_{\leq n}X} & {\coprod_{n \geq 0} \tau_{\leq n}X} & {\hocolim_{n \geq 0} \tau_{\leq n}X} & {} \\
	0 & X & X & {\quad.}
	\arrow["{1-\tn{shift}}", from=1-1, to=1-2]
	\arrow[from=1-1, to=2-1]
	\arrow["\alpha", from=1-2, to=1-3]
	\arrow[from=1-2, to=2-2]
	\arrow["{+}", from=1-3, to=1-4]
	\arrow["\varphi"', dashed, from=1-3, to=2-3]
	\arrow[from=2-1, to=2-2]
	\arrow[Rightarrow, no head, from=2-2, to=2-3]
	\arrow["{+}", from=2-3, to=2-4]
\end{tikzcd}\]
        Then $\Cone(\varphi) \in \T_{\geq 0}$, so by (2) it suffices to show that $H^{i}_{t}(\varphi)$ is an isomorphism for all $i \geq 0$. Since $H^{i}_{t}$ preserves countable coproducts and AB5 holds in $\T^{\heartsuit}$, $H^{i}_{t}$ maps sequential homotopy colimits to sequential colimits by \Cref{remark - H0t hocolim to colim AB5} (2). We thus find
        \begin{align*}
            H^{i}_{t}(\hocolim_{n \geq 0} \tau_{\leq n } X) &\cong \colim_{n \geq 0} H^{i}_{t}(\tau_{\leq n} X) \\
            &\cong \colim_{n \geq i} H^{i}_{t}(\tau_{\leq n} X) \\
            &\cong \colim_{n \geq i} H^{i}_{t}(X) \\
            &\cong H^{i}_{t}(X). \qedhere
        \end{align*}
    \end{proof}
\end{proposition}

% \begin{remark} 
%     \Cref{prop - Lurie} is the dual of \cite{LurieDAGI}[Proposition 1.2.1.19], where we added the condition that homotopy colimits of a countable family of exact triangles complete to an exact triangle, due to the lack of an enhancement. Note that this condition is only needed to show that $(2)$ implies $(1)$ in \Cref{prop - Lurie}. We will almost exclusively work with a dg-enhancement of the derived category, where the homotopy colimit coincides with the homotopy dg-colimit (cf. \Cref{prop - classic hocolim}), which is exact.
% \end{remark}

Next, we mention some examples of $t$-structures.

\begin{example} \label{ex - standard t-structures}
        Let $\A$ be an essentially small dg-category that is cohomologically concentrated in nonpositive degrees. Then the derived category $\D(\A)$ has a natural non-degenerate $t$-structure,\index{$t$-structure $(\mc{T}_{\leq 0},\mc{T}_{\geq 0})$!standard $t$-structure}
        \begin{align*}
            \D(\A)_{\leq 0} &= \{ X \in \D(\A) \mid H^{i}(X) = 0, \forall i > 0 \}, \\
            \D(\A)_{\geq 0} &= \{ X \in \D(\A) \mid H^{i}(X) = 0, \forall i < 0 \}.
        \end{align*}
        Its heart is equivalent to $\Mod(H^{0}(\A)^{\heartsuit})$, see \cite{GLVdB2}[Proposition 3.1.1], and the truncation functors are given by
        \begin{align} \label{eq - standard truncations}
            (\tau_{\leq n} X)(A) &= \tau_{\leq n}X(A) = \left( \bigoplus_{i < n} X(A)^{i} \oplus \Kernel(d_{X(A)}^{n}), d_{X(A)}\right), \\
            (\tau_{\geq n} X)(A) &= \tau_{\geq n}X(A) = \left( \bigoplus_{i \geq n} X(A)^{i} \oplus \Image(d_{X(A)}^{n-1}), d_{X(A)}\right). \nonumber
        \end{align}
        This $t$-structure is right complete by \Cref{prop - Lurie}.
\end{example}

\begin{example} \label{ex - trivial $t$-structure}
An example of a poorly behaved $t$-structure on a triangulated category $\mc{T}$ is the trivial one whose aisle (or coaisle) is everything, so $(\mc{T},0)$. It is degenerate, since it is not left separated (resp. right separated), and therefore also not left (resp. right) complete, by the same argument as in the proof of \Cref{prop - Lurie}. %Indeed, if it were say right complete, then for any object $X \in \bigcap_{n \geq 0} \mc{T}_{\geq n}$, we have that $\tau_{\leq n} X = 0$ and therefore also $X \cong \hocolim_{ n\geq 0} \tau_{\leq n} X =0$. Consequently, $\mc{T}$ is right separated\footnote{The converse holds under certain conditions (see for example the proof of \Cref{prop - properties induced t-structure}), namely when $\mc{T}$ has countable coproducts, the $t$-structure is stable onder those and taking homotopy colimits of an ($\mb{N}$-indexed) family of exact triangles form an exact triangle. This is the case in \Cref{ex - standard t-structures}.}.
\end{example}

To give a more interesting example of a poorly behaved $t$-structure, we look to compactly generated triangulated categories (cf. \cite{NeemanBook}). An object $X \in \mc{T}$ is called \ti{compact}\index{compact!compact (triangulated)} if $\mc{T}(X,-)$ preserves small coproducts. We write $\T^{c}$ for the full subcategory of $\mc{T}$ consisting of the compact objects. The triangulated category $\mc{T}$ is then said to be \ti{compactly generated}\index{compact!compactly generated} if there exists a set $\mc{S}$ of compact objects in $\mc{T}$ that \ti{generates}\index{generating set} $\T$, i.e. so that for any object $X \in \mc{T}$:
$$\forall S \in \mc{S}: \mc{T}(S,X) = 0 \quad \Rightarrow \quad X = 0.$$
Compact objects on a compactly generated triangulated category give rise to $t$-structures via an aisle-construction, see \cite{Lopez}[Theorem A.1]. Although these \ti{compactly generated $t$-structures}\index{compact!compactly generated $t$-structure} can be described fairly concretely, they are not necessarily non-degenerate, as the following example illustrates.

\begin{example} \label{ex - conmpactly gen non-degenerate}
    Consider the graded field $k[u,u^{-1}]$ with $u$ of non-zero degree. It compactly generates $\D(k[u,u^{-1}])$ and defines a compactly generated $t$-structure via \cite{Lopez}[Theorem A.1], with as aisle $$\mc{U}_{\leq 0} = \overline{\langle k[u,u^{-1}] \rangle}^{(-\infty,0]}.$$ This is the smallest full subcategory of $\D(k[u,u^{-1}])$ containing $\Sigma^{-i}k[u,u^{-1}]$ for all $i \in (-\infty,0]$ that is closed under arbitrary coproducts, direct summands and extensions, see \cite{NeemanFunctors}[\S 1]. The $t$-structure is degenerate since for any integer $m$, $\Sigma^{\vert u \vert m}k[u,u^{-1}]$ is quasi-isomorphic to $k[u,u^{-1}]$. Consequently, $\mc{U}_{\leq -n} = \mc{U}_{\leq 0}$ so that $$\bigcap_{n \geq 0} \mc{U}_{\leq -n} = \mc{U}_{\leq 0} \neq 0.$$ As before, we conclude that the $t$-structure is not left complete.
\end{example}

\textbf{In the presence of a dg-enhancement.} Throughout this paper, we will work with dg-enhancements, meaning pretriangulated dg-categories. Given a pretriangulated dg-category $\A$ so that $H^{0}(\A)$ has a $t$-structure $(H^{0}(\A)_{\leq 0}, H^{0}(\A)_{\geq 1})$, we write $\A_{\leq 0}$ (resp. $\A_{\geq 1}$ and $\A^{\heartsuit}$) for the full dg-subcategories of $\A$ on the objects of $H^{0}(\A)_{\leq 0}$ (resp. $H^{0}(\A)_{\geq 1}$ and $H^{0}(\A)^{\heartsuit}$). We then speak of a pretriangulated dg-category $\A$ with a $t$-structure $(\A_{\leq 0}, \A_{\geq 1})$, also called a \ti{$t$-dg-category}.\index{$t$-dg-category} For such an $\A$, one can look to full dg-subcategories of $\A$ consisting of objects that are small with respect to the cohomology induced by the $t$-structure, see \cite{GLVdB2}[Definition 4.4.4]:

\begin{definition} \label{def - hfp}
    The full dg-subcategory $\hfp(\A)$ of \ti{homotopically finitely presented objects}\index{homotopically finitely presented}\index{homotopically finitely presented!$\hfp(\A)$} of $\A$ is defined as
    \begin{align*}
    \hfp(\A) &= \{ X \in \A \mid \forall n \in \mb{Z}: H^{n}_{t}(X) \in \fp(H^{0}(\A)^{\heartsuit}) \},
\end{align*}
where $\fp(H^{0}(\A)^{\heartsuit})$ denotes the full subcategory of $H^{0}(\A)^{\heartsuit}$ on the finitely presented objects\index{finitely presented object}. We also define\index{homotopically finitely presented! bounded, $\hfp^{b}(\A)$}
\begin{align*}
    %\hfp^{+}(\A) &:= \hfp(\A) \cap \A^{+} = \hfp(\A^{+}), \\
    \hfp^{b}(\A) &:= \hfp(\A) \cap \A^{b} = \hfp(\A^{b}). %\\
    %\hfp^{-}(\A) &:= \hfp(\A) \cap \A^{-} = \hfp(\A^{-}).
\end{align*}
\end{definition}

Lastly, we recall a result of \cite{GLVdB2}[\S 3.4] that allows one to transfer $t$-structures to functor categories. Given a pretriangulated dg-category $\A$ with a $t$-structure and a small dg-$R$-category $\mc{I}$ cohomologically concentrated in nonpositive degrees, where one of the two is h-flat, one can endow the dg-category of quasi-functors $\RHom(I,\A)$ with a natural $t$-structure. We hereto briefly recall quasi-functors, but refer to \cite{GLVdB2}[\S 2] for more background. Recall that an $\mc{I}$-$\A$-bimodule $F$ is an $R$-linear dg-functor $\A^{\op} \otimes \mc{I} \to \dgm(R)$, and that it is \ti{right quasi-representable}\index{quasi-functor}\index{quasi-functor!$\hproj^{\rqr}(\mc{I},\A)$}\index{quasi-functor!right quasi-representable}\index{quasi-functor!$\RHom(\mc{I},\A)$} if $$F(-,i) \sim \A(-,\Phi_{F}(i)), \forall i \in \mc{I}.$$
Note that this quasi-isomorphism is not natural. Still, $F$ induces a graded functor $H^{*}(F): H^{*}(\mc{I}) \to H^{*}(\A)$, motivating the slogan that right quasi-representable bimodules are generalised morphisms from $\mc{I}$ to $\A$. We therefore call them \ti{quasi-functors}. We denote the dg-category of $\mc{I}$-$\A$-bimodules by $\dgm(\mc{I},\A)$ and the full dg-subcategory on the right quasi-representable dg-bimodules by $\dgm^{\rqr}(\mc{I},\A)$. Finally, we define the dg-category of \ti{quasi-functors} as
$$\hproj^{\rqr}(\mc{I},\A) := \hproj(\A,\B) \cap \dgm^{\rqr}(\mc{I},\A).$$
This is the internal hom $\RHom(\mc{I},\A)$ in $\Hqe(R)$, see \cite{GLVdB2}[Corollary 2.3.2], and one has that $$\D^{\rqr}(\mc{I},\A) := H^{0}(\hproj^{\rqr}(\mc{I},\A)) = \Hqe(R)(\mc{I},\A).$$ The following result is \cite{GLVdB2}[Theorem 3.4.1],

\begin{thm} \label{thm - t-structure quasi-functors}
    Assume that the base dg-ring $R$ is concentrated in nonpositive degrees. Let $\A$ be a pretriangulated dg-category with a $t$-structure $(\A_{\leq 0}, \A_{\geq 0})$ and let $\mc{I}$ be a small dg-category cohomologically concentrated in nonpositive degrees. Assume that $\A$ or $\mc{I}$ is h-flat. Then the dg-category of quasi-functors $\RHom(\mc{I},\A)$ has a $t$-structure such that
    \begin{align*}
        \RHom(\mc{I},\A)_{\leq 0} &= \RHom(\mc{I},\A_{\leq 0}), \\
        \RHom(\mc{I},\A)_{\geq 0} &= \RHom(\mc{I},\A_{\geq 0}).
    \end{align*}
    The heart of this $t$-structure can be identified with $\Fun(H^{0}(\mc{I}),H^{0}(\A)^{\heartsuit})$. 
\end{thm}

\begin{remark} \label{remark - truncation quasi-functors}
    Recall from the proof of \cite{GLVdB2}[Theorem 3.4.1] that the homotopy category of $\RHom(\mc{I},\A)$ is $$H^{0}(\RHom(\mc{I},\A)) = H^{0}(\hproj^{\rqr}(\mc{I},\A)) = \D^{\rqr}(\mc{I},\tau^{\leq 0}\A).$$ 
    Here, $\tau^{\leq 0}\A$ is the dg-category with the same objects as $\A$ but with inner homs the smart truncations $\tau_{\leq 0}\A(X,Y)$ of the inner homs $\A(X,Y)$ in $\dgm(R)$, see \cite{GLVdB2}[\S 2.5]. The reason being that the truncation functors of the $t$-structure on $H^{0}(\A)$ do not lift to quasi-functors $\A \to \A_{\leq n}$ and $\A \to \A_{\geq n}$ in general, but instead to quasi-functors
    \begin{align*}
        \tau_{\leq n}&: \tau^{\leq 0}\A \to \tau^{\leq 0}\A_{\leq n}, \\
        \tau_{\geq n}&: \tau^{\leq 0}\A \to \tau^{\leq 0}\A_{\geq n},
    \end{align*}
    that are adjoint to the inclusion dg-functors $\iota_{\leq n}$ and $\iota_{\geq n}$ respectively, see \cite{GLVdB2}[Proposition 3.3.1].\index{$t$-dg-category!truncation quasi-functors $\tau_{\leq n}, \tau_{\geq n}$}
    In some cases, for example when $\A = \dgm(R)$ and the $t$-structure is the standard one (cf. \Cref{ex - standard t-structures}), the truncation functors are dg-functorial on $\tau^{\leq 0}\A$. In that case, \Cref{thm - t-structure quasi-functors} is obvious.
\end{remark}

% 1.5 ----------------------------------------------------------------
\subsection{An induced \texorpdfstring{$t$}{t}-structure on the derived category} \label{subsection - induced t-structure}

Below, $\A$ denotes an essentially small strongly pretriangulated dg-category with a $t$-structure $(\A_{\leq 0}, \A_{\geq 1})$. As mentioned before, our aim is to endow the derived category $\D(\A)$ with a $t$-structure so that the Yoneda embedding $H^{0}(\A) \hookrightarrow \D(\A)$ becomes $t$-exact. We first require a technical lemma that allows us to exchange taking the zero-th cohomology with a weighted filtered dg-limit under certain circumstances, seeing as filtered dg-colimits are exact in $\mc{Z}^{0}(\dgm(R))$ but filtered dg-limits are not.

\begin{lemma} \label{lemma - appendix hocolim prop 4.2.3}
    Let $\mc{I}$ be a small dg-$R$-category cohomologically concentrated in nonpositive degrees. Let $X: \mc{I} \to \dgm(R)$ be a contravariant dg-functor so that $H^{k}(X) = 0$ whenever $k <0$. Let $W: \mc{I} \to \dgm(R)$ be a weight so that $H^{k}(W) = 0$ whenever $k > 0$. Then the natural morphism
    \begin{equation} \label{eq - morphism iso}
        H^{0}(\varprojlim{}^{W} X) \to \varprojlim{}^{H^{0}(W)} H^{0}(X)
    \end{equation}
    given by taking the zero-th cohomology, is an isomorphism.
    \begin{proof}
        Note that
        \begin{align*}
            X &\in \RHom(\mc{I},\dgm(R))_{\geq 0}, \\
            W &\in \RHom(\mc{I},\dgm(R))_{\leq 0},
        \end{align*}
        for the $t$-structure of \Cref{thm - t-structure quasi-functors} induced by the standard $t$-structure. Since $\A = \dgm(R)$, \Cref{remark - truncation quasi-functors} applies. Then \eqref{eq - morphism iso} follows from the description of the weighted limit as the inner hom and the adjunctions of the truncation functors.
    \end{proof}
    % \begin{proof} Niet refereren naar t-structuur op quasi-functors, t-structuur op dgm(R) is Z0-functorieel dus we kunnen gewoon puntsgewijs werken.
    
    %     Using the inner hom-description of weighted limits of \S \ref{subsection - weighted colim}, we have
    %     \begin{align*}
    %         \varprojlim{}^{H^{0}(W)} H^{0}(X) &= \Mod(H^{0}(I^{\op}))(H^{0}(Q(W)),H^{0}(X)), \\
    %         H^{0}(\varprojlim{}^{W} X) &= H^{0}(\varprojlim{}^{Q(W)} X) = H^{0}(\dgm(I^{\op}))(Q(W),X).
    %     \end{align*}
    %     Since $Q(W)$ is h-projective, we also have that $$H^{0}(\dgm(I^{\op}))(Q(W),X) \cong \D(I^{\op})(Q(W),X).$$
    %     So the morphism of \eqref{eq - morphism iso} factorizes over $\D(I^{\op})(Q(W),X)$. It thus suffices to show that it is an isomorphism. Since $I$ is cohomologically concentrated in nonpositive degrees, so is $I^{\op}$. Consequently, $\D(I^{\op})$ can be endowed with the standard non-degenerate $t$-structure of \Cref{ex - standard t-structures} whose heart is $\Mod(H^{0}(I^{\op}))$. By definition of this $t$-structure and by our assumptions, $Q(W) \in \D(I^{\op})_{\leq 0}$ and $X \in \D(I^{\op})_{\geq 0}$. Therefore,
    %     \begin{align*}
    %         \D(I^{\op})(Q(W),X) &\cong \D(I^{\op})(\tau_{\leq 0} Q(W),\tau_{\geq 0} X) \\
    %         &\cong \D(I^{\op})^{\heartsuit}(H^{0}_{t}(Q(W)),H^{0}_{t}(X)) \\
    %         &\cong \Mod(H^{0}(I^{\op}))(H^{0}(Q(W)),H^{0}(X)),
    %     \end{align*}
    %     where we used the adjunctions $\iota_{\leq 0} \dashv \tau_{\leq 0}$ and $\tau_{\geq 0} \dashv \iota_{\geq 0}$.
    % \end{proof}
\end{lemma}

Using the dg-enhancement of \Cref{cor - inddgQ compactly generated}, we can now endow the derived category of $\A$ with a $t$-structure.

\begin{proposition} \label{prop - pre-t-structure}
    Let $\A$ be a pretriangulated dg-category with a $t$-structure $(\A_{\leq 0},\A_{\geq 1})$ and let $\iota: \A \hookrightarrow \B$ be an embedding of pretriangulated dg-categories. Then the pair of full subcategories 
    \begin{align*}
        H^{0}(\B)_{\geq 1} &:= H^{0}(\A)_{\leq 0}^{\perp}, \\
        H^{0}(\B)_{\leq 0} &:= {}^{\perp}(H^{0}(\A)_{\leq 0}^{\perp}),
    \end{align*}
    defines a \emph{pre-$t$-structure} on $H^{0}(\B)$, in the sense that it satisfies
    \begin{itemize}
        \item[$(T1)$.] $\forall X \in H^{0}(\B)_{\leq 0}, Z \in H^{0}(\B)_{\geq 1}$: $H^{0}(\B)(X,Z) = 0$;
        \item[$(T2)$.] $H^{0}(\B)_{\leq 0}[1] \subseteq H^{0}(\B)_{\leq 0}$ and $H^{0}(\B)_{\geq 1}[-1] \subseteq H^{0}(\B)_{\geq 1}$.
    \end{itemize}
    \begin{proof}
        $(T1)$ follows immediately from the definition of the aisle $H^{0}(\B)_{\leq 0}$ as the left orthogonal to the coaisle $H^{0}(\B)_{\geq 1}$. For $(T2)$, consider first $Z \in H^{0}(\B)_{\geq 1}$ and $A \in H^{0}(\A)_{\leq 0}$. Then 
        \begin{align*}
            H^{0}(\B)(\iota(A),Z[-1]) &\cong H^{0}(\B)(\iota(A[1]),Z) = 0
        \end{align*}
        because $A[1] \in H^{0}(\A)_{\leq 0}$. Consequently, $H^{0}(\B)_{\geq 1}[-1] \subseteq H^{0}(\B)_{\geq 1}$. Then follows that $H^{0}(\B)_{\leq 0}[1] \subseteq H^{0}(\B)_{\leq 0}$ because
        \begin{align*}
            H^{0}(\B)(X[1],Z) &\cong H^{0}(\B)(X,Z[-1]) = 0
        \end{align*}
        for any $X \in H^{0}(\B)_{\leq 0}$ and $Z \in H^{0}(\B)_{\geq 1}$.
    \end{proof}
\end{proposition}

\begin{thm} \label{thm - t-structure inddg}
    Let $\A$ be an essentially small strongly pretriangulated dg-category with a $t$-structure $(\A_{\leq 0},\A_{\geq 1})$ and consider the embedding into the homotopy ind-dg-completion $\Y: \A \hookrightarrow \Ind^{\dg,Q}(\A)$. Then 
    \begin{enumerate}
        \item the pre-$t$-structure of \Cref{prop - pre-t-structure},
        \begin{align} \label{eq - induced t-structure 1}
                \D(\A)_{\geq 1} &:= H^{0}(\A)_{\leq 0}^{\perp}, \\
        \D(\A)_{\leq 0} &:= {}^{\perp}(H^{0}(\A)_{\leq 0}^{\perp}), \nonumber
    \end{align}
        extends to a $t$-structure on $\D(\A)$, i.e. the pair moreover satisfies
        $(T3)$: $\forall X \in \D(\A)$, there exists an exact triangle
        \[\begin{tikzcd}
	U & X & V & {\quad,}
	\arrow[from=1-1, to=1-2]
	\arrow[from=1-2, to=1-3]
	\arrow["{+}", from=1-3, to=1-4]
\end{tikzcd}\]
        with $U \in \D(\A)_{\leq 0}$ and $V \in \D(\A)_{\geq 1}$;

        \item this $t$-structure makes the Yoneda embedding $\Y: H^{0}(\A) \hookrightarrow \D(\A)$ $t$-exact;

        \item both the aisle $\D(\A)_{\leq 0}$ and coaisle $\D(\A)_{\geq 1}$ are closed under filtered (homotopy) dg-colimits;
        
        \item the truncation functors preserve filtered (homotopy) dg-colimits.
    \end{enumerate}
    \begin{proof}
        Throughout the proof, we will be using \Cref{lemma - independency hocolim} (2) to identify filtered homotopy dg-colimits and filtered dg-colimits.

        We start with $(1)$, so $(T3)$. Consider hereto an arbitrary object $\Ll_{\A}(X)$ in $\D(\A)$ with $X: R\{P\} \to \A$ an ind-dg-object. Since $R$ is strictly concentrated in nonpositive degrees, so is $R\{P\}$. So $X$ goes into $\tau^{\leq 0} \A$, see \Cref{remark - truncation quasi-functors}. We can then postcompose $X: R\{P\} \to \tau^{\leq 0} \A$ with the truncation quasi-functors of \Cref{remark - truncation quasi-functors},
        \begin{align*}
            \tau_{\leq n}: \tau^{\leq 0} \A &\to \tau^{\leq 0} (\A_{\leq n}), \\
            \tau_{\geq n}: \tau^{\leq 0} \A &\to \tau^{\leq 0} (\A_{\geq n}).
        \end{align*}
        This yields quasi-functors
        \begin{align*}
            \tau_{\leq n}X: R\{P\} \to \tau^{\leq 0} (\A_{\leq n}), \\
            \tau_{\geq n}X: R\{P\} \to \tau^{\leq 0} (\A_{\geq n}).
        \end{align*}
        By \Cref{thm - t-structure quasi-functors}, there is an exact triangle in $\D^{\rqr}(R\{P\},\tau^{\leq 0}\A)$,
        \begin{equation} \label{eq - exact triangle A} \begin{tikzcd}[column sep = small]
	{\iota_{\leq 0}\tau_{\leq 0}X} & X & {\iota_{\geq 1}\tau_{\geq 1}X} & {\quad.}
	\arrow[from=1-1, to=1-2]
	\arrow[from=1-2, to=1-3]
	\arrow["{+}", from=1-3, to=1-4]
\end{tikzcd}\end{equation}
        Since $c_{P^{\op}}$ is h-flat by \Cref{lemmahflat}, we have the triangulated functor
        $$c_{P^{\op}} \otimes_{R\{P\}} -: \D^{\rqr}(R\{P\},\tau^{\leq 0}\A) \to \D(\tau^{\leq 0}\A).$$
        So \eqref{eq - exact triangle A} gives rise to an exact triangle in $\D(\A)$ via $\D(\tau^{\leq 0}\A) \hookrightarrow \D(\A)$,
        \begin{equation} \label{eq - exact triangle B} \begin{tikzcd}[column sep = small]
	{c_{P^{\op}} \otimes_{R\{P\}}\iota_{\leq 0}\tau_{\leq 0}X} & {\Ll_{\A}(X)} & {c_{P^{\op}} \otimes_{R\{P\}}\iota_{\geq 1}\tau_{\geq 1}X} & {\quad.}
	\arrow[from=1-1, to=1-2]
	\arrow[from=1-2, to=1-3]
	\arrow["{+}", from=1-3, to=1-4]
\end{tikzcd}\end{equation}
        So, using once again \Cref{lemma - independency hocolim} (2), we see that the truncation functors are given by
        \begin{align} \label{eq - induced truncation functors}
            \tau_{\leq 0} (\Ll_{\A}^{Q}(X)) = \colim{}^{Q(c_{P^{\op}})} \iota_{\leq 0}\tau_{\leq 0}X \in \D(\A)_{\leq 0}, \\
            \tau_{\geq 1} (\Ll_{\A}^{Q}(X)) = \colim{}^{Q(c_{P^{\op}})} \iota_{\geq 1}\tau_{\geq 1}X \in \D(\A)_{\geq 1}, \nonumber
        \end{align}
        with $\iota_{\leq 0}, \tau_{\leq 0}, \iota_{\geq 1}$ and $\tau_{\geq 1}$ the truncation (quasi-)functors on $\RHom(R\{P\},\A)$ of \Cref{thm - t-structure quasi-functors}. We still need to show that \eqref{eq - exact triangle B} satisfies the conditions of $(T3)$, i.e. that
        \begin{align*}
            &(i) \quad \tn{ } \forall U \in \D^{\rqr}(R\{P\},\tau^{\leq 0}(\A_{\leq 0})): \tn{ } \colim{}^{c_{P^{\op}}} U \in \D(\A)_{\leq 0}, \\
            &(ii) \quad \forall V \in \D^{\rqr}(R\{P\},\tau^{\leq 0}(\A_{\geq 1})): \tn{ } \colim{}^{c_{P^{\op}}} V \in \D(\A)_{\geq 1}.
        \end{align*}
        %, i.e. that
        % \begin{align*}
        %     &(i) \quad \tn{ }c_{P^{\op}} \otimes_{R\{P\}}\iota_{\leq 0}\tau_{\leq 0}X = \colim{}^{c_{P^{\op}}} \iota_{\leq 0}\tau_{\leq 0}X \in \D(\A)_{\leq 0}, \\
        %     &(ii) \quad c_{P^{\op}} \otimes_{R\{P\}}\iota_{\geq 1}\tau_{\geq 1}X = \colim{}^{c_{P^{\op}}} \iota_{\geq 1}\tau_{\geq 1}X \in \D(\A)_{\geq 1}.
        % \end{align*}

                \begin{enumerate}
            \item[$(ii)$] Take $A \in H^{0}(\A)_{\leq 0}$. Then
            \begin{align*}
                \D(\A)(\Y(A),\colim{}^{c_{P^{\op}}} V) &\cong H^{0}\left((\colim{}^{c_{P^{\op}}} V)(A)\right) \\
                &\cong H^{0}\left(\colim{}^{c_{P^{\op}}} (V)(A)\right) \\
                &\cong \colim_{j} H^{0}(V)_{j}(A) \\
                &= 0,
            \end{align*}
            because $\mc{Z}^{0}(\dgm(R))$ is $\tn{AB}5$ and $V \in \D^{\rqr}(R\{P\},\tau^{\leq 0}(\A_{\geq 1}))$.
        
            \item[$(i)$] Take $\Ll_{\A}^{Q}(Z) \in \D(\A)_{\geq 1}$. Then 
            \begin{align} \label{eq - cohomol varprojlim}
                & \quad \D(\A)(\colim{}^{c_{P^{\op}}} U,\Ll_{\A}^{Q}(Z)) \\ \nonumber
                &\cong \D(\A)(\colim{}^{c_{P^{\op}}} Q(U),\Ll_{\A}^{Q}(Z)) \\ \nonumber
                &\cong H^{0}\left(\varprojlim{}^{c_{P^{\op}}} \hproj(\A)(Q(U)_{-},\Ll_{\A}^{Q}(Z)\right).
            \end{align}
            We now wish to bring the filtered dg-limit outside, as we did in $(ii)$. This cannot be done in general since filtered limits are not exact in $\mc{Z}^{0}(\dgm(R))$. In our setting however, we can apply \Cref{lemma - appendix hocolim prop 4.2.3}. Indeed, $I = R\{P\}^{\op}$ is cohomologically concentrated in nonpositive degrees and the the diagram $X$ is $\hproj(\A)(Q(U)_{-},\Ll_{\A}^{Q}(Z))$. So there remains to be checked that it has no negative cohomology. If $k <0$, then 
            \begin{align*}
                &\quad H^{k}\left(\hproj(\A)(Q(U)_{j},\Ll_{\A}^{Q}(Z) \right) \\
                &\cong \D(\A)(Q(U)_{j},\Ll_{\A}^{Q}(Z)[k]) \\
                &\cong \D(\A)(U_{j},\Ll_{\A}^{Q}(Z)[k]) \\
                &= \tn{}0,
            \end{align*}
            by the quasi-isomorphism $Q(U)_{j} \sim U_{j} \sim \A(-,A_{j})$ for some $A_{j} \in \A_{\leq 0}$, together with $(T2)$. We can then continue rewriting \eqref{eq - cohomol varprojlim} as
            \begin{align*}
                & \quad \varprojlim{}^{H^{0}(c_{P}^{\op})} H^{0}\left(\hproj(\A)(Q(U)_{j},\Ll_{\A}^{Q}(Z)\right) \\
                &\cong \varprojlim{}^{H^{0}(c_{P}^{\op})} \D(\A)(Q(U)_{j},\Ll_{\A}^{Q}(Z)) \\
                &\cong \varprojlim{}^{H^{0}(c_{P}^{\op})} \D(\A)(\A(-,A_{j}),\Ll_{\A}^{Q}(Z)) \\
                &= 0.
            \end{align*}
        \end{enumerate}

        Concerning (2), so $t$-exactness of $\Y: H^{0}(\A) \hookrightarrow \D(\A)$: this follows from the commutativity of \eqref{eq - cdiagram embedding LAQ} together with $\iota_{\leq 0}\tau_{\leq 0} c_{A} \sim c_{\tau_{\leq 0} A}$ and $\iota_{\geq 1}\tau_{\geq 1} c_{A} \sim c_{\tau_{\geq 1} A}$. 
        
        For (3), remark first that the aisles are always closed under filtered homotopy dg-colimits since the $\iota_{\leq m}$ are left adjoints, see \Cref{remark - closed products}. That the coaisles are closed under them, follows as in (ii).

        Finally, (4) follows from the description of the truncation functors of \eqref{eq - induced truncation functors}, \Cref{lemma - independency hocolim} (2) and \Cref{thm - enriched completion closed}.
    \end{proof}
\end{thm}

\begin{definition}
    We call the $t$-structure of \eqref{eq - induced t-structure 1} the \ti{induced $t$-structure}.\index{induced $t$-structure}
\end{definition}

Let us look to the heart of the induced $t$-structure.

\begin{proposition} \label{prop - heart inddg}
    If we endow $\D(\A)$ with the induced $t$-structure, then $$\D(\A)^{\heartsuit} = H^{0}(\Ind^{\dg,Q}(A))^{\heartsuit} \cong \Ind(H^{0}(\A)^{\heartsuit})$$
    and $H^{0}_{t} \circ \Y \cong \Y_{\tn{ab}} \circ H^{0}_{t}$. In particular, the heart is Grothendieck abelian. Furthermore, the cohomological functor $H^{0}_{t} = \tau_{\leq 0}\tau_{\geq 0}: \D(\A) \to \Ind(H^{0}(\A)^{\heartsuit})$ sends filtered (homotopy) dg-colimits to filtered colimits.
    \begin{proof}
        Consider an object $\Ll_{\A}^{Q}(X) \in \D(\A)^{\heartsuit}$. By the description of the truncation functors of \eqref{eq - induced truncation functors}, we have that $\Ll_{\A}^{Q}(X) \cong \colim{}^{Q} \iota_{\leq 0}\tau_{\leq 0}X$ and moreover that $\Ll_{\A}^{Q}(X) \cong \colim{}^{Q} H^{0}_{t} \circ X$. We then use \Cref{thm - t-structure quasi-functors} and \cite{GLVdB2}[Remark 3.4.5] to conclude.
    \end{proof}
\end{proposition}

We do not seem to have $t$-exactness in the universal property of \Cref{cor - universal prop inddgho} in general: even if $F$ in \eqref{eq - universal prop enriched} is $t$-exact, it is unclear whether $\Ind^{\dg,Q}(F)$ is, because taking the homotopy ind-dg-completion strictly enlarges the category. However, the slightly weaker \Cref{cor - D(F)} does respect $t$-exactness:

\begin{corollary} \label{cor - D(F) t-exact}
    Let $F: \A \to \B$ be a $t$-exact functor between essentially small strongly pretriangulated $t$-dg-categories. Then $F$ extends uniquely to a $t$-exact dg-functor $\D(F)$ that makes the diagram below commute.
    \[\begin{tikzcd}[row sep = small]
	\A & \B \\
	{\Ind^{\dg,Q}(\A)} & {\Ind^{\dg,Q}(\B).}
	\arrow["F", from=1-1, to=1-2]
	\arrow["\Y"', hook, from=1-1, to=2-1]
	\arrow["\Y", hook, from=1-2, to=2-2]
	\arrow["{\D(F)}", from=2-1, to=2-2]
\end{tikzcd}\]
    \begin{proof}
        By \Cref{cor - D(F)}, we only have to show $t$-exactness of $\D(F)$. Consider hereto $\colim{}^{Q(c_{P}^{\op})} \Y \circ X \in \D(\A)_{\geq 1}$. Then
        $$\colim{}^{Q(c_{P}^{\op})} \Y \circ X \cong \colim{}^{Q(c_{P}^{\op})} \tau_{\geq 1} \circ X.$$
        Since the object on the right may not live in $\Ind^{\dg,Q}(\A)$, but in $\hproj(\A)$, we use the second description of $\D(F)$ from \Cref{cor - D(F)}. We find that
        \begin{align*}
            \D(F)(\colim{}^{Q(c_{P}^{\op})} \Y \circ X) &\cong \D(F)( \colim{}^{Q(c_{P}^{\op})} \tau_{\geq 1} \circ X) \\
            &\cong \colim{}^{Q(c_{P}^{\op})} \tau_{\geq 1} \circ X \otimes_{\A} \B(-,F(-)) \\
            &\cong \colim{}^{Q(c_{P}^{\op})} ((\tau_{\geq 1} \circ X)_{i} \otimes_{\A} \B(-,F(-))) \\
            &\cong \colim{}^{Q(c_{P}^{\op})} \tau_{\geq 1} \circ F \circ X
        \end{align*}
        since by $t$-exactness of $F$,
        \begin{align*}
            (\tau_{\geq 1} \circ X)_{i} \otimes_{\A} \B(-,F(-)) &\cong \A(-,\tau_{\geq 1} X_{i}) \otimes_{\A} \B(-,F(-)) \\
            &\cong \B(-,F(\tau_{\geq 1} X_{i})) \\
            &\cong \B(-,\tau_{\geq 1}(F(X_{i}))) \\
            &\cong (\tau_{\geq 1} \circ F \circ X)_{i}.
        \end{align*}
        The argument for left $t$-exactness is analogous.
    \end{proof}
\end{corollary}

Next, we ask what properties of the $t$-structure on $\A$ are inherited by the induced $t$-structure on $\D(\A)$. Here, we run into some trouble. We start with a positive result:

\begin{proposition} \label{prop - properties induced t-structure}
    Let $\A$ be an essentially small strongly pretriangulated dg-category with a bounded or bounded above $t$-structure. Then the induced $t$-structure on $\D^{+}(\A)$ is non-degenerate and right complete.
    \begin{proof}
        Note that for any strongly pretriangulated $t$-dg-category $\mc{C}$,
        \begin{align*}
            \bigcap_{n \geq 0} \D(\mc{C})_{\geq n} &= \bigcap_{n \geq 0} (\mc{C}_{\leq n-1}^{\perp}) = (\bigcup_{n \geq 0} \mc{C}_{\leq n-1})^{\perp} = (\mc{C}^{-})^{\perp}, \\
            \bigcap_{n \geq 0} \D(\mc{C})_{\leq -n} &= \bigcap_{n \geq 0} {}^{\perp}(\mc{C}_{\leq -n}^{\perp}) = {}^{\perp}(\bigcup_{n \geq 0} \mc{C}_{\leq -n}^{\perp}) = {}^{\perp}(\bigcup_{n \geq 0} \D(\mc{C})_{\geq n}) = {}^{\perp}\D^{+}(\mc{C}),
        \end{align*}
        where we actually mean $H^{0}(\mc{C})$ everywhere. Consequently, the $t$-structure is left separated if we consider only the left bounded part $\D^{+}(\mc{A})$. It is clearly also right separated if $\A$ is bounded or bounded above because $\D(\A)$ is compactly generated by $\A \cong \A^{-}$ in that case. Right completeness then follows by \Cref{prop - Lurie}, \Cref{thm - t-structure inddg} (3) and \Cref{prop - heart inddg}.
    \end{proof}
\end{proposition}

That the induced $t$-structure can be degenerate without the assumptions of \Cref{prop - properties induced t-structure}, is illustrated by the following example\footnote{This agrees with \Cite{Antieau}[Remark below Proposition 2.13] using \Cref{prop - Lurie} and its dual.}.

\begin{example} \label{ex - degenerate induced t-structure}
    Consider the derived category $\D^{+}(k)$ for a field $k$, endowed with the non-degenerate standard $t$-structure of \Cref{ex - standard t-structures}, and enlarge the universe so that $\D^{+}(k)$ is essentially small (one can avoid this by considering $\D^{+}(\Vector^{\tn{fd}}(k))$ instead). Then we claim that the induced $t$-structure on $\D^{+}(\D^{+}(k)) = H^{0}(\Ind^{\dg,Q,+}(\D_{\dg}^{+}(k)))$ is degenerate. %Consider hereto $\bigoplus_{n \in \mb{N}} \Y(k)[-n]$ and $\Y(\bigoplus_{n \in \mb{N}}  k[-n])$. 
    Consider hereto the natural map,
    $$\xi: \bigoplus_{n \in \mb{N}} \Y(k)[-n] \longrightarrow \Y(\bigoplus_{n \in \mb{N}}  k[-n]).$$
    From the properties of the induced $t$-structure follows that $\tau_{\leq n}\xi$ is an isomorphism for all $n \geq 0$. Consequently, $\Cone(\xi) \in \bigcap_{n \geq 0} \D^{+}(\D^{+}(k))_{\geq n}$. Since $\Y(\bigoplus_{n \in \mb{N}}  k[-n])$ is compact and $\bigoplus_{n \in \mb{N}} \Y(k)[-n]$ is not -- indeed, consider the natural map $$\bigoplus_{m \geq 0} \homom(\bigoplus_{n \in \mb{N}} \Y(k)[-n],\Y(k)[-m]) \to \homom(\bigoplus_{n \in \mb{N}} \Y(k)[-n],\bigoplus_{m \in \mb{N}} \Y(k)[-m])$$ and observe that the identity is not reached -- $\Cone(\xi)$ is different from zero. Thus, the induced $t$-structure is not right separated and therefore degenerate.
\end{example}

A natural question is whether the standard $t$-structure is induced. Under certain circumstances, this is the case:

\begin{example}
    Let $R$ be a ring of finite global dimension. By the Auslander-Buchsbaum-Serre theorem,
    $$\Perf(R) \cong \D^{b}(R).$$
    Combining this with \cite{Beilinson}[Lemma 1], we find equivalences
    $$\D(R) \cong \D(\Perf(R)) \cong \D(\D^{b}(R))$$
    that are $t$-exact for the standard $t$-structure and the induced one.
    
    A similar result holds on the geometric side. Let $X$ be a separated, noetherian scheme over $\mb{C}$ of finite Krull dimension, and so that $\coh(X)$ has enough locally free sheaves (see \cite{OrlovSing}[\S 1.2]). If $X$ is regular, then its singularity category vanishes, ergo
    $$\Perf(X) \cong \D^{b}(X),$$
    see for example \cite{OrlovSing}[Remark 1.9]. By combining this with \cite{BondalVdB} and \cite{Beilinson}[Lemma 1], we find $t$-exact equivalences
    $$\D(X) \cong \D(\Perf(X)) \cong \D(\D^{b}(X)).$$
\end{example}

\subsection{Left bounded locally coherent Grothendieck \texorpdfstring{$t$}{t}-dg-categories}
\label{subsection - t-lc}

Assume that $\A$ is an essentially small strongly pretriangulated dg-category with a bounded $t$-structure. Then we find ourselves in the setting of \Cref{prop - properties induced t-structure} so that the induced $t$-structure on the homotopy ind-dg-completion $\Ind^{\dg,Q,+}(\A)$ is non-degenerate and right complete. By \Cref{thm - t-structure inddg}, we have a $t$-exact embedding $$\Y: \A \hookrightarrow \Ind^{\dg,Q,+}(\A)$$ given by Yoneda. As a result, we can apply any construction defined in $\Ind^{\dg,Q,+}(\A)$ -- and perhaps not in the smaller dg-category, due to for example the lack of enough (derived) injectives -- to objects of $\A$ by embedding them. To conclude whether the result is again `small', we wish to characterize the quasi-essential image of $\A$ in $\Ind^{\dg,Q,+}(\A)$.

\begin{proposition} \label{prop - hfpb or colim}
    Let $\A$ be an essentially small strongly pretriangulated dg-category with a bounded $t$-structure and $\C$ be an essentially small strongly pretriangulated $t$-dg-category so that $H^{0}(\C)$ is Karoubian. Then
    \begin{enumerate}
        \item $\A \cong \hfp^{b}(\Ind^{\dg,Q,+}(\A))$;
        \item $H^{0}(\A) \cong \D(\A)^{c}$;
        \item $H^{0}(\C) \cong \D(\C)^{c}$.
    \end{enumerate}
    %Let $\C$ be a strongly pretriangulated dg-category with a bounded above $t$-structure that is left complete. Then
    %\begin{enumerate}
        %\item[(3)] $\C \cong \hfp^{-}(\Ind^{\dg,Q}(\C))$.
    %\end{enumerate}
    \begin{proof}
        Using \Cref{prop - heart inddg} and \cite{AbelianDeformations}[\S 2.2], which states that $$\fp(\Ind( H^{0}(\A)^{\heartsuit})) \cong H^{0}(\A)^{\heartsuit},$$ we can rewrite $\hfp^{b}(\Ind^{\dg,Q},+(\A))$ as
        \begin{align*}
            \biggl\{ X \in \Ind^{\dg,Q,+}(\A) \bigg| \begin{matrix} \forall n \in \mb{Z}: H^{n}_{t}(X) \in H^{0}(\A)^{\heartsuit}, \\
            \exists M \geq 0: (\tau_{\leq -M}X = 0) \wedge (\tau_{\geq M+1}X = 0) \end{matrix} \biggl\}.
        \end{align*}
        
        \begin{enumerate}  
            % bounded
            \item $\A$ clearly embeds into $\hfp^{b}(\Ind^{\dg,Q,+}(\A))$. Conversely, consider $X$ in $\hfp^{b}(\Ind^{\dg,Q,+}(\A))$ and $M \geq 0$ so that $\tau_{\leq -M}X = 0$ and $\tau_{\geq M+1}X = 0$. We will show via an inductive argument that $X \in \A$:
            \begin{itemize}
                \item[-] If $M = 0$, then $X = 0$ and is certainly contained in $\A$.
                
                \item[-] Let $M > 0$. As $\tau_{\leq -M}X = 0$, we have that $X \cong \tau_{\geq -M+1}X$. Then 
                \begin{equation} \label{eq - (a)}
                    \tau_{\leq -M+1}X \cong H^{-M+1}_{t}(X)[M-1] \in H^{0}(\A).
                \end{equation}
                Since $\tau_{\geq M+1}X = 0$, we also have that $X \cong \tau_{\leq M}X$ so that
                \begin{equation} \label{eq - (b)}
                    \tau_{\geq M}X \cong H^{M}_{t}(X)[-M] \in H^{0}(\A).
                \end{equation}
                If we combine \eqref{eq - (a)} with the exactness of
                \[\begin{tikzcd}
	{\tau_{\leq -M+1}X} & {\tau_{\leq -M+2}X} & {H^{-M+2}_{t}(X)[M-2]} & {\quad,}
	\arrow[from=1-1, to=1-2]
	\arrow[from=1-2, to=1-3]
	\arrow["{+}", from=1-3, to=1-4]
\end{tikzcd}\]
                we find that $\tau_{\leq -M+2}X \in H^{0}(\A)$. We can repeat this argument to eventually conclude that $\tau_{\leq M-1}X \in H^{0}(\A)$. The following exact triangle and \eqref{eq - (b)} then show that $X \in \A$,
                \[\begin{tikzcd}
	{\tau_{\leq M-1}X} & X & {\tau_{\geq M}X} & {\quad.}
	\arrow[from=1-1, to=1-2]
	\arrow[from=1-2, to=1-3]
	\arrow["{+}", from=1-3, to=1-4]
\end{tikzcd}\]
                \end{itemize}

            \item Recall from \Cref{cor - inddgQ compactly generated} that $\D(\A)$ is compactly generated by the Yoneda image of $H^{0}(\A)$. This is equivalent to $H^{0}(\A)$ classically generating $\D(\A)^{c}$ (cf. \cite{NeemanGrothendieckDuality}[Theorem 2.1 and Remark 2.2]), meaning that $\D(\A)^{c}$ is the smallest triangulated subcategory of $\D(\A)$ that contains $H^{0}(\A)$ and is closed under direct summands. It thus suffices to show that $H^{0}(\A)$ is closed under direct summands. Since the truncation functors commute with biproducts and finitely presented objects are closed under direct summands (see for example \cite{AdamekRosicky}[Proposition 1.3]), this is clear from $(1)$. 

            \item We use the same argument as in (2), now using the fact that $H^{0}(\C)$ is Karoubian and therefore closed under direct summands.
            \qedhere
        \end{enumerate}
    \end{proof}
\end{proposition}

\Cref{prop - hfpb or colim} (1) can be restated as a categorical correspondence, which should be compared to \cite{AbelianDeformations}[\S 2.2].

\begin{definition} \label{def - left tlc}
    A triangulated category $\T$ with a $t$-structure $(\T_{\leq 0}, \T_{\geq 0})$ is \ti{left bounded locally coherent Grothendieck} if
    \begin{enumerate}
        \item the $t$-structure is left bounded and right complete;
        \item $\T^{\heartsuit}$ is locally coherent Grothendieck abelian;
        \item $H^{0}_{t}: \T \to \T^{\heartsuit}$ preserves coproducts;
        \item $\fp(\T^{\heartsuit}) \subseteq \T^{c}$.
    \end{enumerate}
    Note that by $\iota_{\leq 0} \dashv \tau_{\leq 0}$ and the truncation exact triangles, (4) is equivalent to
    \begin{enumerate}
    \item[(4')] $\forall X \in \fp(\T^{\heartsuit}), i \in \mb{Z}: H^{i}(\D(X,-)): \T^{\heartsuit} \to \Ab$ preserves coproducts.
    \end{enumerate}
    A $t$-dg-category $\D$ with a $t$-structure is then a \ti{left bounded locally coherent Grothendieck $t$-dg-category}\index{$t$-dg-category!left bounded locally coherent Grothendieck} if $H^{0}(\D)$ with that $t$-structure is left bounded locally coherent Grothendieck.
\end{definition}

\begin{remark}
    In \cite{JuliaFrancesco}[Setup 3.1.1], a definition was given for unbounded Grothendieck $t$-(dg-)categories, though the term itself was not yet used there. A natural definition of an unbounded locally coherent Grothendieck $t$-dg-category would then be a Grothendieck $t$-dg-category $\A$ with the extra conditions:
    \begin{itemize}
        \item[-] $H^{0}(\A)$ is compactly generated;
        \item[-] the generators $\mc{U}$ of \cite{JuliaFrancesco}[Setup 3.1.1] can be chosen inside $H^{0}(\A)^{c}$;
        \item[-] the compact objects inherit the $t$-structure from $\A$.
    \end{itemize}
    However, this notion does not seem to work well in the unbounded case where the two possible candidates differ: the unique dg-enhancement of $\D(\Ind(\mb{A}))$ for $\mb{A}$ abelian is not necessarily compactly generated (\cite{KrauseAuslander}), while the $t$-structure on $\Ind^{\dg,Q}(\D^{b}(\mb{A}))$ may be degenerate (\Cref{ex - degenerate induced t-structure}). In the left bounded setup, the two agree, see \Cref{cor - equiv induced derived}.
\end{remark}

\begin{lemma} \label{lemma - hfpb compact} 
    Let $\T$ be a triangulated category with a $t$-structure that is left bounded locally coherent Grothendieck. Then $\hfp^{b}(\T) \subseteq \T^{c}$.
    \begin{proof}
        Since shifts and extensions of compact objects remain compact and we are given that $\fp(\T^{\heartsuit}) \subseteq \T^{c}$, the statement follows by considering the truncation exact triangles.
    \end{proof}
\end{lemma}

\begin{corollary} \label{cor - lcGt enough DGInj}
    Let $\T$ be a triangulated category with a $t$-structure that is left bounded locally coherent Grothendieck. Then $\T$ has enough derived injectives.
    \begin{proof}
        This follows by condition (1) in \Cref{def - right complete}, \Cref{lemma - hfpb compact}, \Cref{def - left tlc} and \Cref{cor - enough DGInj}.
    \end{proof}
\end{corollary}

\begin{thm} \label{thm - categorical hfpb inddg} 
    The following is an equivalence of categories between the full subcategories of $\Hqe(R)$,
\[\begin{tikzcd}[column sep = small]
	\begin{array}{c} \begin{Bmatrix} \textnormal{essentially small strongly} \\ \textnormal{pretriangulated dg-categories } \\ \textnormal{with a bounded $t$-structure} \end{Bmatrix} \end{array} && \begin{array}{c} \begin{Bmatrix} \textnormal{left bounded locally} \\ \textnormal{coherent Grothendieck} \\ \textnormal{$t$-dg-categories} \end{Bmatrix}. \end{array}
	\arrow["{\hproj^{+}(-)}", shift left=2, from=1-1, to=1-3]
	\arrow["{\hfp^{b}(-)}", shift left=2, from=1-3, to=1-1]
\end{tikzcd}\]
    \begin{proof}
        By \Cref{prop - hfpb or colim} and \Cref{thm - homotopy dg completion closed} (2), we have that
        $$\A \cong \hfp^{b}(\hproj^{+}(\A))$$
        for $\A$ an essentially small strongly pretriangulated dg-category with a bounded $t$-structure. 
        
        Conversely, let $\C$ be a left bounded locally coherent Grothendieck $t$-dg-category. By \cite{AbelianDeformations}[\S 2.2], $H^{0}(\C)^{\heartsuit} \cong \Ind(\mathbb{A})$ for an abelian category $\mb{A}$. Then $H^{0}(\hfp^{b}(\C))^{\heartsuit} \cong \fp(H^{0}(\C)^{\heartsuit}) \cong \mb{A}$ and there is a $t$-exact embedding $\hfp^{b}(\C) \hookrightarrow \C$. That $\hfp^{b}(\C)$ is small, follows from the smallness of $\mb{A}$, see for example \cite{AbelianDeformations}[\S 2.2]. Consider now the quasi-functor
        \begin{equation}
            \Psi: \C \to \hproj^{+}(\hfp^{b}(\C)): C \mapsto Q^{\hproj}(\C(-,C)\vert_{\hfp^{b}(\C)}).
        \end{equation}
        Note that for $C \in \hfp^{b}(\C)$,
        \begin{align*}
            H^{0}(\Psi)(C) &= Q^{\hproj}(\C(-,C)\vert_{\hfp^{b}(\C)}) \\ 
&= Q^{\hproj}(\hfp^{b}(\C)(-,C)) \\ 
&\sim \hfp^{b}(\C)(-,C).
        \end{align*}
        Consequently, $\Psi$ defines a $t$-exact quasi-equivalence between $\hfp^{b}(\C)$ as a subcategory of respectively $\C$ and $\hproj^{+}(\hfp^{b}(\C))$. We claim that this extends to a $t$-exact quasi-equivalence of the whole $t$-dg-categories.

        To see that $\Psi$ is quasi-fully faithful, consider $C,C' \in \C$. By right completeness, $C \cong \hocolim \tau_{\leq n} C$. Hence,
        \begin{align*}
            \C(\hocolim_{n \geq 0} \tau_{\leq n}C, C') \sim \holim_{n \geq 0}\C(\tau_{\leq n}C, C').
        \end{align*}
        Using \Cref{lemma - hfpb compact}, we also have that
        \begin{align*}
            &\quad \hproj^{+}(\hfp^{b}(\C))(H^{0}(\Psi)(C),H^{0}(\Psi)(C')) \\ 
            &\sim \hproj^{+}(\hfp^{b}(\C))(\hocolim_{n \geq 0} H^{0}(\Psi)(\tau_{\leq n} C),H^{0}(\Psi)(C')) \\      
            &\sim \holim_{n \geq 0}\hproj^{+}(\hfp^{b}(\C))(H^{0}(\Psi)(\tau_{\leq n}C),H^{0}(\Psi)(C')).
        \end{align*}
        Since each $\tau_{\leq n} C$ is bounded, it thus suffices -- using the truncation exact triangles inductively -- to show for $X \in H^{0}(\C)^{\heartsuit}$ that
        \begin{align*}
            \C(X, C') \sim \hproj^{+}(\hfp^{b}(\C))(H^{0}(\Psi)(X),H^{0}(\Psi)(C')).
        \end{align*}
        Since $H^{0}(\C)^{\heartsuit} = \Ind(\mb{A})$, $X = \colim_{i \in I} A_{i}$ is a filtered colimit of $A_{i} \in \mb{A}$. By \Cref{lemma - hfpb compact} and right completeness, we have that
        \begin{align*}
            \C(\colim A_{i},C')
            &\sim \C(\colim{}^{Q} A_{i},C') \\ &\sim \lim{}^{Q} \C(A_{i},C') \\
            &\sim \lim{}^{Q}\hocolim_{m \geq 0} \C(A_{i},\tau_{\leq m}C')
        \end{align*}
        and
        \begin{align*}
            &\quad \hproj^{+}(\hfp^{b}(\C))(H^{0}(\Psi)(X),H^{0}(\Psi)(C')) \\
            &\sim \hproj^{+}(\hfp^{b}(\C))(\colim{}^{Q} \hfp^{b}(\C)(-,A_{i}),H^{0}(\Psi)(C')) \\
            &\sim \lim{}^{Q} \hproj^{+}(\hfp^{b}(\C))(\hfp^{b}(\C)(-,A_{i}),H^{0}(\Psi)(C')) \\
            &\sim \lim{}^{Q} \hocolim_{m \geq 0} \hproj^{+}(\hfp^{b}(\C))(\hfp^{b}(\C)(-,A_{i}),H^{0}(\Psi)(\tau_{\leq m}C')).
        \end{align*}
        So we have reduced, once more, to showing that
        $$\C(A,M) \sim \hproj^{+}(\hfp^{b}(\C))(H^{0}(\Psi)(A), H^{0}(\Psi)(M))$$
        for $A \in \mb{A}$ and $M \in \C^{b}$. Again, using the truncation exact triangles inductively, we can reduce even further to $M \in H^{0}(\C)^{\heartsuit} = \Ind(\mb{A})$ and then $\mb{A}$, in which case the claim follows.

        We continue with $t$-exactness of $H^{0}(\Psi)$. If $C' \in \C_{\geq 1}$, then $\C(C,C')= 0$ for all $C \in \C_{\leq 0}$, in particular this holds for $C \in \hfp^{b}(\C)_{\leq 0}$. Consequently, $$H^{0}(\Psi)(C') \in \hfp^{b}(\C)_{\leq 0}^{\perp} = \hproj^{+}(\hfp^{b}(\C))_{\geq 1}$$ so that $H^{0}(\Psi)$ is right $t$-exact. If $\C \in \C_{\leq 0}$, then $C$ is bounded. Since $H^{0}(\Psi)$ maps $\hfp^{b}(\C)$ to $\hfp^{b}(\C)$ and $H^{0}(\Psi)$ preserves colimits by \Cref{lemma - hfpb compact}, \Cref{thm - homotopy dg completion closed} (4) and \Cref{prop - classic hocolim}, it maps $H^{0}(\C)^{\heartsuit} = \Ind(\mb{A})$ (on)to $\hproj^{+}(\hfp^{b}(\C)) = \Ind(\mb{A})$. We can then inductively show left $t$-exactness by using the truncation exact triangles and right $t$-exactness.

        Lastly, we have to show that $H^{0}(\Psi)$ is essentially surjective. By right completeness of $\hproj^{+}(\hfp^{b}(\C))$ due to \Cref{prop - properties induced t-structure} and the fact that $H^{0}(\Psi)$ preserves colimits by \Cref{lemma - hfpb compact} and \Cref{thm - homotopy dg completion closed} (4) and \Cref{prop - classic hocolim}, it suffices to show this for the bounded part $\hproj^{b}(\hfp^{b}(\C))$. As before, we can again use the truncation exact triangle to reduce to the hearts. Since $\hfp^{b}(\C)$, so in particular $\mb{A}$, is reached and we just remarked that colimits are preserved by $H^{0}(\Psi)$, we have essential surjectivity.
 \end{proof}
\end{thm}

\begin{corollary} \label{cor - equiv induced derived}
    Let $\mb{A}$ be an abelian category so that $\D^{b}_{\dg}(\mb{A})$ is essentially small. Then there is a $t$-exact quasi-equivalence
    $$\D_{\dg}^{+}(\Ind(\mb{A})) \cong \hproj^{+}(\D^{b}_{\dg}(\mb{A})),$$
    where we consider the standard $t$-structure on the left and the induced $t$-structure coming from the standard one on the right.
    \begin{proof}
        It suffices to remark that $\D_{\dg}^{+}(\Ind(\mb{A}))$ is left bounded locally coherent Grothendieck for the standard $t$-structure and that $$\hfp^{b}(\D_{\dg}^{+}(\Ind(\mb{A}))) = \D^{b}(\mb{A}).$$
        The statement then follows from \Cref{thm - categorical hfpb inddg}.
    \end{proof}
\end{corollary}

% 2 --------------------------------------------------------------------

\section{Deformation equivalences} \label{section - deformation equivalences}

In this chapter we present the basic setup of $t$-deformations and their relevant adjoints. We start by recalling the deformation setup for dg-derived categories induced by a dg-deformation of dg-categories in \S \ref{pardefset1}. Next, in \S \ref{pardefset2} and \S \ref{pardefset3}, we recall $t$-deformations with enough derived injectives and how to construct a $t$-deformation from a dg-deformation of the derived injectives, as was shown in \cite{GLVdB2}. In \S \ref{pardefset4}, we treat the converse by constructing an associated dg-deformation between the dg-categories of derived injectives from a $t$-deformation with enough derived injectives, which naturally extends to a right adjoint of this $t$-deformation. These two results we then combine into a deformation equivalence in \S \ref{pardefset5}, covering the case with enough derived injectives. 

To get rid of the condition of the $t$-deformation having enough derived injectives, we first show in \S \ref{pardefset6} that left bounded locally coherent Grothendieck $t$-dg-categories are preserved by $t$-deformations with enough derived injectives. In \S \ref{pardefset7}, we then extend the equivalence of categories of \Cref{thm - categorical hfpb inddg} to a deformation equivalence between the bounded $t$-deformations of an essentially small strongly pretriangulated bounded $t$-dg-category and the left bounded locally coherent Grothendieck $t$-deformations of its dg-derived category.

%In what follows, we will be using the base change constructions of \cite{GLVdB2}[\S 4].

% 1.2 ------------------------------------------------------------------
\subsection{Deformations of dg-derived categories} \label{pardefset1}

Let $\tb{A}$ be an $R$-linear dg-category, which is assumed to be $R$-cofibrant. In particular, for objects $A, A' \in \tb{A}$, $\tb{A}(A,A')$ is cofibrant whence h-projective and h-flat.
The dg-category $\dgm(\tb{A})$ is both tensored and cotensored over $\dgm(R)$, by
\begin{equation} \label{eq - dgmA tensor}
\dgm(R) \otimes \dgm(\tb{A}) \to  \dgm(\tb{A}): (P,M) \mapsto P \otimes_R M
\end{equation}
with $(P \otimes_R M)(A) = P \otimes_R M(A)$,
and
\begin{equation} \label{eq - dgmA cotensor}
\dgm(R) \otimes \dgm(\tb{A}) \to  \dgm(\tb{A}): (P,M) \mapsto \homom_R(P,M)
\end{equation}
with 
\begin{equation}\label{homform}
\homom_R(P,M)(A) = \homom_R(P, M(A)).
\end{equation}
In the sequel, we will make use of the balanced derived bifunctor of the cotensor functor:
\begin{proposition} \label{prop - balancing}
    There is a quasi-functor
    \begin{align*}
        \RHom_R(-,-): \hproj(R)^{\op} \otimes \hinj(\A) \to \D_{\dg}(\A): (P,M) \mapsto \homom_R(P,M)
    \end{align*}
    which computes derived functors in both arguments.
\end{proposition}

\begin{proof}
    This is a classical balancing argument. We are to show that:
    \begin{enumerate}
    \item for $P \in \hproj(R)^{\op}$, $\homom_R(P,-)$ preserves acyclicity;
    \item for $M \in \hinj(\tb{A})$, $\homom_R(-,M)$ preserves acyclicity.
    \end{enumerate}
    Statement (1) immediately follows from \eqref{homform} and the fact that acyclicity of $M$ is pointwise. Statement (2) follows from the same observations in combination with Lemma \ref{lemhinj}.
\end{proof}

\begin{lemma}\label{lemhinj}
For $M \in \hinj(\tb{A})$ and $A \in \tb{A}$, we have that $M(A) \in \hinj(R)$.
\end{lemma}

\begin{proof}
    For fixed $A \in \tb{A}$, consider the evaluation functor $$\ev_{A}: \dgm(\tb{A}) \to \dgm(R): M \mapsto M(A).$$ The claim then easily follows from the existence of the left adjoint $$\dgm(R) \to \dgm(\tb{A}): X \mapsto X \otimes_R \tb{A}(-,A),$$ which is exact since $\tb{A}$ is h-flat.
\end{proof}

Let $\theta: R \rightarrow S$ be as before and consider the dg-$S$-category $S \otimes_{R}^{\mb{L}} \tb{A} = S \otimes_{R} \tb{A}$ and the canonical $\rho: \tb{A}\rightarrow S \otimes_R \tb{A}$.
This type of base change models the most basic notion of deformation for dg-categories:

\begin{definition}\label{defdgdef}
    Let $\tb{B}$ be an $S$-linear dg-category. A \emph{dg-deformation}\index{deformation!dg-deformation} of $\tb{B}$ along $\theta: R \to S$ is an $R$-linear dg-category $\tb{A}$ together with an $R$-linear quasi-functor $\tb{A} \to \tb{B}$ that induces an $S$-linear quasi-equivalence $S \otimes_{R}^{\mb{L}} \tb{A} \cong \tb{B}$.
\end{definition}

In the remainder of this section, we discuss the impact of a dg-deformation -- which for simplicity we assume to be given by $\rho$ -- on the associated dg-derived categories. This will inspire another type of deformation in \S \ref{pardefset3}. 

Consider the induced forgetful functor between dg-categories of dg-modules
$$(-)_R: \dgm_S(S \otimes_R \tb{A}) \longrightarrow \dgm_R(\tb{A}),$$
which has both a left adjoint,
$$S \otimes_R -:  \dgm_S(\tb{A}) \longrightarrow \dgm_R(S \otimes_R\tb{A}): M \longmapsto S \otimes_R M,$$
and a right adjoint,
$$\homom_R(S,-):  \dgm_S(\tb{A}) \longrightarrow \dgm_R(S \otimes_R\tb{A}): M \longmapsto \homom_R(S,M).$$
These functors give rise to corresponding derived adjunctions between the derived categories $\D(\tb{A})$ and $\D(S \otimes_R \tb{A})$. We can realise them on the level of dg-models of the derived categories as follows: Since $(-)_R$ preserves acyclic modules, the functor $S \otimes_R -$ preserves h-projectives, whereas $\homom_R(S,-)$ preserves h-injectives. 
We thus obtain
\begin{equation}\label{tensor}
S \otimes_R -: \hproj_{R}(\tb{A}) \to \hproj_{S}( S \otimes_{R} \tb{A}),
\end{equation}
which has as its right adjoint quasi-functor the composition of the dg-functor $(-)_{R}: \hproj_{S}(S \otimes_{R} \tb{A}) \to \dgm_{R}(\tb{A})$ and the quasi-functor taking h-projective resolutions $Q^{\hproj}: \dgm_{R}(\tb{A}) \to \hproj_{R}(\tb{A})$. Similarly, we obtain
\begin{equation} \label{eq - rhom 1}
    \homom_{R}(S,-): \hinj_{R}(\tb{A}) \to \hinj_{S}( S \otimes_{R} \tb{A}),
\end{equation}
which has as its left adjoint quasi-functor the composition of the dg-functor $(-)_{R}: \hinj_{S}(S \otimes_{R} \tb{A}) \to \dgm_{R}(\tb{A})$ and the quasi-functor taking h-injective resolutions $Q^{\hinj}: \dgm_{R}(\tb{A}) \to \hinj_{R}(\tb{A})$.

Consider now the representable dg-module $\tb{A}(-,A) \in \dgm_{R}(\tb{A})$ for $A \in \tb{A}$. We then have that $S \otimes_R \tb{A}(-,A) = (S \otimes_R \tb{A})(-,A)$ and there is a corresponding commutative diagram,
\begin{equation} \begin{tikzcd} \label{eq - relation dg-def}
	{\tb{A}} & {\hproj_{R}(\tb{A})} \\
	{S \otimes_{R} \tb{A}} & {\hproj_{S}(S \otimes_{R}\tb{A}).}
	\arrow["\Y", from=1-1, to=1-2]
	\arrow["\rho"', from=1-1, to=2-1]
	\arrow["{S \otimes_{R} -}", from=1-2, to=2-2]
	\arrow["\Y"', from=2-1, to=2-2]
\end{tikzcd}\end{equation}
% \begin{equation} \label{eq - relation dg-def}
% \xymatrix{ {\tb{A}} \ar[d]_{\rho} \ar[r]^-\Y & {\hproj_R(\tb{A})} \ar[d]^{S \otimes_R -} \\ {S \otimes_R \tb{A}} \ar[r]_-\Y &  {\hproj_{S}( S \otimes_{R} \tb{A}).}}
% \end{equation}
In other words, the dg-deformation $\rho$ can be recovered by restricting the left adjoint $S \otimes_R -$. This observation will inspire our approach in \S \ref{pardefset3}.

Finally, recall that we have the following base change formula
\begin{equation} \label{eq - base change}
    {\hproj_R(\tb{A})}_{(S)} \cong {\hproj_{S}( S \otimes_{R} \tb{A})},
\end{equation}
where $(-)_{(S)}$ denotes the coextension of scalars, which will be recalled in \S \ref{pardefset3}. This formula expresses that the dg-derived category of $\tb{A}$ can be viewed as a deformation of the dg-derived category of $S \otimes_R \tb{A}$ in an appropriate sense that is reminiscent of abelian deformations.

\subsection{\texorpdfstring{$t$}{t}-dg-categories with enough derived injectives} \label{pardefset2}

Motivated by the story of abelian deformation theory -- which, as mentioned above, the dg-setting is reminiscent of -- the authors looked for (re)construction theorems that link certain dg-categories and $t$-dg-categories (generalising dg-derived categories) in \cite{GLVdB1}. In this subsection, we will specify these categories and recall the correspondence between them.

In the abelian setup, one departs from an abelian category with enough injectives and considers the full linear subcategory on the injective objects, from which one can then reconstruct the whole abelian category, see \cite{AbelianDeformations}[Proposition 6.25].  Here, we need a dg-derived analogue. This is where the $t$-structure comes in: it allows us to lift the notion of injectivity from the heart to the triangulated category.

\begin{definition}
    Let $\mathcal{T}$ be a triangulated category with a $t$-structure $(\mc{T}_{\leq 0}, \mc{T}_{\geq 0})$ and heart $\mathcal{T}^{\heartsuit}$. An object $X \in \mc{T}$ is \ti{derived injective} if it represents the contravariant functor 
    $$\mathcal{T}^{\heartsuit}(H^{0}_{t}(-),I): \mathcal{T}^{\op} \to \Ab$$
    for an injective object $I \in \mathcal{T}^{\heartsuit}$, ergo $\mathcal{T}^{\heartsuit}(H^{0}_{t}(-),I) \cong \mathcal{T}(-,X)$. Because of Yoneda, the notation $X = L(I)$ is justified. $\mc{T}$ is then said to have \ti{enough derived injectives}\index{derived injective}\index{derived injective!enough derived injectives}\index{derived injective!$\DGInj(\A)$} if $\mc{T}^{\heartsuit}$ has enough injectives and such a representable object $L(I)$ exists for every $I \in \Inj (\mc{T}^{\heartsuit})$.
\end{definition}

\begin{examples} \label{ex - enough DGInj}
    \begin{enumerate}
        \item A triangulated category $\mc{T}$ with a $t$-structure whose heart $\mc{T}^{\heartsuit}$ has enough injectives, so that $H^{0}_{t}: \T \to \T^{\heartsuit}$ preserves coproducts and to which Brown's representability theorem \cite{NeemanBook}[Theorem 1.17] or \Cref{thm - Brown} applies, has enough derived injectives, see \Cref{cor - enough DGInj}. This includes left bounded locally coherent Grothendieck triangulated categories, see \Cref{cor - lcGt enough DGInj}.

        \item Let $\D^{+}(\mb{A})$ be the left bounded derived category of a Grothendieck abelian category $\mb{A}$ endowed with the standard $t$-structure of \Cref{ex - standard t-structures}. Then it was shown in \cite{GLVdB2}[Lemma 5.2.1] that the derived injectives in $\D^{+}(\mb{A})$ are precisely the injectives in $\mb{A}$.

        \item If we endow $\D^{+}(\D^{b}(\mb{A}))$ for $\mb{A}$ abelian so that $\D^{b}_{\dg}(\mb{A})$ is essentially small with the $t$-structure induced by the standard $t$-structure of \Cref{ex - standard t-structures} in \Cref{cor - recap defequiv}, then the derived injectives are the injectives in $\Ind(\mb{A})$ by \Cref{cor - equiv induced derived} and \cite{GLVdB2}[Lemma 5.2.1].
    \end{enumerate}
\end{examples}

The definition of a derived injective object immediately results in certain properties \cite{NonFM}[\S 6.1].

\begin{proposition} \label{prop - properties dginj}
    Let $\mc{T}$ be a triangulated category with a $t$-structure $(\mc{T}_{\leq 0}, \mc{T}_{\geq 0})$ and heart $\mathcal{T}^{\heartsuit}$. Let $L(I)$ be a derived injective in $\mc{T}$. Then 
    \begin{enumerate}
        \item $L(I) \in \mc{T}_{\geq 0}$;
        \item $H^{0}_{t}(L(I)) \cong I$;
        \item If $L(J)$ is another derived injective in $\mc{T}$, then
        $$\mc{T}(L(I),L(J)[i]) \cong \begin{cases}
            \mc{T}^{\heartsuit}(I,J), & i = 0, \\
            0, & i > 0.
        \end{cases}$$
        In particular, $$L: \Inj(\mathcal{T}^{\heartsuit}) \to \mc{T}: I \mapsto L(I)$$ defines a fully faithful functor.
    \end{enumerate}
\end{proposition}

%correspondence + two players

\begin{definition}
     Let $\mc{T}$ be a triangulated category with a $t$-structure. The essential image of the functor $L$, which is the full subcategory of derived injectives in $\mc{T}$, is denoted by $\DGInj(\mc{T})$. If $\mc{T} = H^{0}(\B)$ for some dg-category $\B$, we simply write $\DGInj(\B) := \DGInj(H^{0}(\B))$, viewing it as a full dg-subcategory of $\B$.
\end{definition}

\begin{remark}
    Note that by \Cref{prop - properties dginj} (3), the dg-category $\DGInj(\A)$ is cohomologically concentrated in nonpositive degrees.
\end{remark}

Similar to abelian categories with enough injectives, we have the following class of $t$-dg-categories that we will be deforming.

\begin{definition}
    The category $\Hqe^{t+}$\index{$\Hqe^{t+}_{(\Pi)}$} (resp. $\Hqe^{t+}_{\Pi}$) has the following data:
    \begin{itemize}
        \item[-] objects are dg-categories $\B$ endowed with a left bounded, non-degenerate $t$-structure (resp. that is moreover closed under countable products) and with enough derived injectives;

        \item[-] morphisms $\B \to \A$ are morphisms in $\Hqe$ (see Conventions) that have a $t$-exact left adjoint.
    \end{itemize}
\end{definition}

One can then consider $\DGInj(\B)$ for any $\B \in \Hqe^{t+}_{\Pi}$ and wonder what sort of dg-categories one ends up with. In the abelian setup, the categories of injectives of abelian categories with enough injectives are precisely the additive coherent Karoubian categories, see \cite{AbelianDeformations}[Propositions A.13 and A.14]. Here, $\DGInj(-)$ lands in the following category:

\begin{definition}
    The category $\Hqe^{\DGInj}$\index{$\Hqe^{\DGInj}_{(\Pi)}$} (resp. $\Hqe^{\DGInj}_{\Pi}$) has the following data:
    \begin{itemize}
        \item[-] objects are dg-categories $\tb{J}$ that are (left) homotopically locally coherent (see \Cref{def - hlc}) and so that $H^{0}(\tb{J})$ is Karoubian (resp. and $H^{0}(\tb{J})$ is moreover closed under countable products);

        \item[-] morphisms $G: \tb{I} \to \tb{J}$ are morphisms in $\Hqe$ such that for all $J \in \tb{J}$, the $H^{0}(\tb{I})^{\tn{op}}$-module $H^{0}(\tb{J})(J,G(-))$ is finitely presented.
    \end{itemize}
\end{definition}

\begin{definition} \label{def - hlc}
    A dg-$R$-category $\tb{J}$ is called (left) \ti{homotopically locally coherent}\index{homotopically locally coherent (left)} if:
    \begin{itemize}
        \item[-] $\tb{J}$ is cohomologically concentrated in nonpositive degrees;
        \item[-] $H^{0}(\tb{J})$ is an additive and (left) coherent $H^{0}(R)$-linear category;
        \item[-] for all $J \in \tb{J}$, the dg-$\tb{J}^{\tn{op}}$-module $\tb{J}(J,-)$ is homotopically finitely presented, ergo $H^{i}(\tb{J}(J,-))$ is a finitely presented $H^{0}(\tb{J})^{\tn{op}}$-module for all $i \in \mb{Z}$.
    \end{itemize}
\end{definition}

In \cite{GLVdB1}[Theorem 1.4], it was shown that
\begin{equation} \label{eq - DGInj}
    \DGInj: \Hqe^{t+}_{(\Pi)} \to \Hqe^{\DGInj}_{(\Pi)}.
\end{equation}
Conversely, one can construct from a dg-category $\tb{J} \in \Hqe^{\DGInj}_{(\Pi)}$ a dg-category as follows \cite{GLVdB1}[Definition 4.4]:

\begin{definition} \label{def - lb twisted complex}
    Let $\tb{J}$ be a dg-category that is cohomologically concentrated in nonpositive degrees. Let $\tb{J}^{\oplus}$ denote the closure of $\tb{J}$ under finite sums and zero objects in $\dgm(\tb{J})$ via Yoneda $\Y: \tb{J} \hookrightarrow \dgm(\tb{J})$. 
    
    A \ti{left bounded twisted complex over} $\tn{\tb{J}}$\index{left bounded twisted complex}\index{left bounded twisted complex!$\Tw^{+}(\tb{J})$} is a pair $(X,q_{X}) = (\bigoplus_{i \in \mathbb{Z}} X_{i}[-i],q)$ where each $X_{i} \in \tb{J}^{\oplus}$ and $X_{i} = 0$ for $i \ll 0$, and where $$q = \{ q^{j}_{i}: X_{i}[-i] \to X_{j}[-j] \}_{i,j\in \mb{Z}}$$ is of degree $1$ so that $\delta_{\bigoplus_{i \in \mathbb{Z}} X_{i}[-i]} \circ q + q^{2} = 0$ and $q^{j}_{i} = 0$ whenever $i < j+1$. The \ti{dg-category of left bounded twisted complexes over} $\tb{J}$, denoted by $\Tw^{+}(\tb{J})$, then consists of the following data:
    \begin{itemize}
        \item[-] objects are the left bounded twisted complexes over $\tb{J}$;
        \item[-] inner homs $\Tw^{+}(\tb{J})((X,q_{X}),(X',q_{X'}))$ consist in degree $p$ of morphisms $$f = \{f^{j}_{i}: X_{i}[-i] \to X'_{j}[-j]\}_{i,j \in \mb{Z}}$$ each of degree $p$ so that $f^{j}_{i} = 0$ whenever $i < j+p$. The differential is given by
        $$d_{X,X'}(f)^{j}_{i} = (-1)^{j}d_{\tb{J}}(f^{j}_{i}) + \Sigma_{k} (q_{X',k}^{j} f^{k}_{i} - (-1)^{p} f^{j}_{k} q_{X,i}^{k}).$$
    \end{itemize}
\end{definition}

If the dg-category $\tb{J}$ is (left) homotopically locally coherent, one can endow $\Tw^{+}(\tb{J})$ with a $t$-structure, as was shown in \cite{GLVdB1}[Theorem 5.9].

\begin{thm}
    Let $\tn{\tb{J}}$ be a left homotopically locally coherent dg-category. Then $\hproj(\tn{\tb{J}}^{\op})$ (with homotopy category $\D(\tn{\tb{J}}^{\op})$) can be endowed with the standard $t$-structure of \Cref{ex - standard t-structures}. 

    \begin{enumerate}
        \item The full dg-category $\hproj(\tn{\tb{J}}^{\op})^{\hfp} \subseteq \hproj(\tn{\tb{J}}^{\op})$ on the homotopically finitely presented objects (cf. \Cref{def - hfp}) is strongly pretriangulated and has a non-degenerate $t$-structure induced from the standard one on $\hproj(\tn{\tb{J}}^{\op})$, with as heart $\modu(H^{0}(\tn{\tb{J}}^{\op}))$. In other words, the category $\D(\tn{\tb{J}}^{\op})^{\hfp}$ is a triangulated subcategory of $\D(\tn{\tb{J}}^{\op})$ and it has a non-degenerate $t$-structure induced from $\D(\tn{\tb{J}}^{\op})$, with as heart $\modu(H^{0}(\tn{\tb{J}}^{\op}))$.
        \item There is a quasi-equivalence $$\Tw^{+}(\tn{\tb{J}}) \cong \hproj^{-}(\tn{\tb{J}}^{\op})^{\hfp,\op}.$$ In particular, $\Tw^{+}(\tn{\tb{J}})$ has a canonical non-degenerate left bounded $t$-structure with heart $\modu(H^{0}(\tn{\tb{J}}^{\op}))^{\op}$.
    \end{enumerate}
\end{thm}

By \cite{GLVdB1}[Theorem 1.2], this construction defines a functor
\begin{equation} \label{eq - Tw+}
    \Tw^{+}: \Hqe^{\DGInj}_{(\Pi)} \to \Hqe^{t+}_{(\Pi)}.
\end{equation}
Finally, we state the correspondence of \cite{GLVdB1}[Theorem 1.4].

\begin{thm} \label{thm - correspondence}
    The functors \eqref{eq - DGInj} and \eqref{eq - Tw+} form an equivalence of categories when restricted to $\Hqe^{\DGInj}_{\Pi}$ and $\Hqe^{t+}_{\Pi}$,
    \begin{equation} \label{eq - correspondence} \begin{tikzcd}
	{\Hqe^{\DGInj}_{\Pi}} & {\Hqe^{t+}_{\Pi}.}
	\arrow["\cong"{description}, draw=none, from=1-1, to=1-2]
	\arrow["{\Tw^{+}}", curve={height=-18pt}, from=1-1, to=1-2]
	\arrow["\DGInj", curve={height=-18pt}, from=1-2, to=1-1]
\end{tikzcd}\end{equation}
\end{thm}

\subsection{\texorpdfstring{$t$}{t}-deformations with enough derived injectives} \label{pardefset3}

We will now use the base change formula \eqref{eq - base change} and the relation with dg-deformations expressed in \eqref{eq - relation dg-def} as inspiration to define deformations of $t$-dg-categories.

Let $\theta: R \to S$ be as before. Then there are the following base change functors\index{base change functors} of \cite{GLVdB2}[\S 5.1]:
\begin{itemize}
    \item \ti{The restriction of scalars along $\theta$}\index{base change functors!restriction of scalars $(-)_{R}$}. Let $\tb{B}$ be an $S$-linear dg-category. Then there is an underlying $R$-linear dg-category $\tb{B}_{R}$, obtained by restricting along $\theta$. This yields a functor $$(-)_{R}: \Hqe(S) \to \Hqe(R)$$
    that preserves $t$-structures.

    \item \ti{The extension of scalars along $\theta$}\index{base change functors!extension of scalars $S \otimes^{\mb{L}}_{R} -$}. Let $\tb{A}$ be an $R$-linear dg-category. Then there is an $S$-linear dg-category $S \otimes^{\mb{L}}_{R} \tb{A}$ determined up to quasi-equivalence, with the same objects as $\tb{A}$ but as inner homs the derived tensor product of those in $\tb{A}$ with $S$. This yields a functor $$S \otimes^{\mb{L}}_{R} -: \Hqe(R) \to \Hqe(S).$$ 
    Given an $S$-linear dg-category $\tb{B}$, there is moreover a natural $R$-linear quasi-equivalence \cite{GLVdB2}[Proposition 4.3.4],
    \begin{equation} \label{eq - ext}
        \RHom_{R}(\tb{A},\tb{B}_{R}) \cong \RHom_{R}(S \otimes^{\mb{L}}_{R} \tb{A}, \tb{B}).
    \end{equation}

    \item \ti{The coextension of scalars along $\theta$}\index{base change functors!coextension of scalars $(-)_{(S)} = \hproj^{\rqr}(S,Q(-))$}. Let $\tb{A}$ be an $R$-linear dg-category. Then there is an $S$-linear dg-category $\tb{A}_{(S)} = \RHom_{R}(S,\tb{A})$ given by the quasi-functors $S \to \tb{A}$ (cf. \S \ref{subsection - t-structures}), see \cite{GLVdB2}[Lemma 4.2.1], determined up to quasi-equivalence. This yields a functor $$(-)_{(S)}: \Hqe(R) \to \Hqe(S)$$
    that preserves $t$-structures. Concretely, for an $R$-linear $t$-dg-category $\A$, there is a natural $t$-structure on $\A_{(S)}$ that for $\A$ h-flat is given by
    \begin{align} \label{eq - tstructure coextension}
        (\A_{(S)})_{\leq 0} &= \RHom_{R}(S,\A_{\leq 0}), \\
        (\A_{(S)})_{\geq 0} &= \RHom_{R}(S,\A_{\geq 0}), \nonumber
    \end{align}
    and non-degeneracy of the $t$-structure is preserved. Given an $S$-linear dg-category $\tb{B}$, there is moreover a natural $R$-linear quasi-equivalence \cite{GLVdB2}[Theorem 4.2.5],
    \begin{equation} \label{eq - coext}
        \RHom_{R}(\tb{B}_{R},\tb{A}) \cong \RHom_{R}(\tb{B},\tb{A}_{(S)}).
    \end{equation}
\end{itemize}

\begin{remark} \label{remark - quasif RHom}
    Recall from \S \ref{pardefset1} and \eqref{eq - base change} that for an h-flat $R$-linear dg-category $\A$, we have that
    $$\hproj_{R}(\A)_{(S)} \cong \hproj_{S}(S \otimes_{R} \A).$$
    If we denote $F$ for the forgetful functor $\hproj(\A)_{(S)} \to \hproj(\A)$, then we have a right adjoint quasi-functor corresponding to \eqref{eq - rhom 1},
    \[\begin{tikzcd}[font=\small, column sep = small]
	{\hproj(\A)} & {\hinj(\A)} & {\hinj(S \otimes_{R}\A)} & {\hproj(S \otimes_{R} \A) \cong\hproj(\A)_{(S)}.}
	\arrow["{Q^{\hinj}}", from=1-1, to=1-2]
	\arrow["{\homom_{R}(S,-)}", from=1-2, to=1-3]
	\arrow["{Q^{\hproj}}", from=1-3, to=1-4]
\end{tikzcd}\]
    We call this quasi-functor $\RHom_{R}(S,-)$. It could be confused with the derived bifunctor $\RHom_{R}(-,-)$ of \Cref{prop - balancing} with $S \in \dgm(R)$ in the first argument. However, by postcomposing with the counit $\epsilon$ of \eqref{eq - coext}, the two agree; i.e. the following diagram of quasi-functors commutes in $\Hqe(R)$:
    \[\begin{tikzcd}[font=\small, column sep = small]
	{\hproj(\A)} & {\hinj(\A)} & {\hinj(S \otimes_{R}\A)} & {\hproj(S \otimes_{R} \A) \cong \hproj(\A)_{(S)}} \\
	&& {\hinj(\A)} & {\hproj(\A).}
	\arrow["{Q^{\hinj}}", from=1-1, to=1-2]
	\arrow["{\homom_{R}(S,-)}", from=1-2, to=1-3]
	\arrow["{\RHom_{R}(S,-)}"', from=1-2, to=2-3]
	\arrow["{Q^{\hproj}}", from=1-3, to=1-4]
	\arrow["{\epsilon_{\hproj(\A)}}", shift left=5, from=1-4, to=2-4]
	\arrow["{Q^{\hproj}}", from=2-3, to=2-4]
\end{tikzcd}\]
\end{remark}

By \eqref{eq - ext} and \eqref{eq - coext}, we have adjunctions
$$S \otimes^{\mb{L}}_{R} - \quad \dashv \quad (-)_{R} \quad \dashv \quad (-)_{(S)}.$$
These are reminiscent of the base change functors in the abelian setup of \cite{AbelianDeformations}[\S 4], namely
$$H^{0}(S) \otimes_{H^{0}(R)} - \quad \dashv \quad (-)_{H^{0}(R)} \quad \dashv \quad (-)_{(H^{0}(S))}$$
and for an $S$-linear $t$-dg-category $\B$ and $R$-linear $t$-dg-category, we have by \cite{GLVdB2}[(4.51)] that
\begin{align*}
    H^{0}(\B_{R})^{\heartsuit} &\cong H^{0}(\B)^{\heartsuit}_{H^{0}(R)}, \\
    H^{0}(\A_{(S)})^{\heartsuit} &\cong H^{0}(\A)^{\heartsuit}_{(H^{0}(S))},
\end{align*}
so the restriction and coextension of scalars on both levels are compatible. Similar to how abelian deformations were defined in \cite{AbelianDeformations}[Definition 5.2], namely as lifts under the coextension of scalars -- since it preserves abelianness while the extension of scalars does not -- we use the coextension of scalars to define deformations of $t$-dg-categories as in \cite{GLVdB2}[Definition 5.1.2].

\begin{definition}
    Let $\B$ be an $S$-linear $t$-dg-category.
    A \emph{$t$-deformation} of $\B$ along $\theta: R \to S$ is an $R$-linear $t$-dg-category $\A$ together with an $R$-linear quasi-functor $\B \to \A$ that induces a $t$-exact $S$-linear quasi-equivalence $\B \to \A_{(S)}$.\index{deformation!$t$-deformation}
\end{definition}

So we now have a notion of deformation available on each side of \eqref{eq - correspondence}: dg-deformations on the left and $t$-deformations on the right. A natural question is then whether the categories of \eqref{eq - correspondence} are closed under deformations. In \cite{GLVdB2} this was proven for the left side, so dg-categories of derived injectives, and dg-deformations left were shown to induce $t$-deformations on the right.

\begin{thm}\label{thm - GLV2}
    Let $\tn{\tb{J}}$ be an essentially small $S$-linear dg-category in $\Hqe^{\DGInj}_{\Pi}$ and let $G: \tn{\tb{I}} \to \tn{\tb{J}}$ be a dg-deformation along $\theta: R \to S$. Then $\tn{\tb{I}} \in \Hqe^{\DGInj}_{\Pi}$ and the left adjoint to $\Tw^{+}(G): \Tw^{+}(\tn{\tb{J}}) \to \Tw^{+}(\tn{\tb{I}})$ is a $t$-deformation in $\Hqe^{t+}_{\Pi}$.
    \begin{proof}
        The claim without the $\Pi$-subscript follows by \cite{GLVdB2}[Theorem 4], i.e. $\tb{I} \in \Hqe^{\DGInj}$ and the dg-deformation $G: \tb{I} \to \tb{J}$ in $\Hqe^{\DGInj}$ gives rise to a $t$-deformation $F: \Tw^{+}(\tb{J}) \to \Tw^{+}(\tb{I})$ in $\Hqe^{t+}$. That leaves closedness of $H^{0}(\tb{I})$ under countable products. Since $\tb{J} \in \Hqe^{\DGInj}_{\Pi}$, we have by \Cref{thm - correspondence} that $\Tw^{+}(\tb{J}) \in \Hqe^{t+}_{\Pi}$. By \Cref{thmdgdef} (3), \Cref{lemma - countable products} and \Cref{remark - hqet+pi not tlc} below, then $\Tw^{+}(\tb{I}) \in \Hqe^{t+}_{\Pi}$ so that $\tb{I} \in \Hqe^{\DGInj}_{\Pi}$ by \Cref{thm - correspondence}.
    \end{proof}
\end{thm}

To end this subsection, we treat the leading example of a $t$-deformation. This covers \cite{GLVdB2}[Example 5.3.3].

\begin{example} \label{ex - example}
    Let $k$ be a field and take $R$ and $S$ to be $k[\epsilon]/(\epsilon^{n})$ and $k[\epsilon]/(\epsilon^{m})$ with zero differentials, where $1 \leq m \leq n$ and $\epsilon$ is of degree $d \leq 0$. Then the quotient dg-$k$-algebra morphism,
    $$\theta: k[\epsilon]/(\epsilon^{n}) \twoheadrightarrow k[\epsilon]/(\epsilon^{m})$$
    satisfies all the necessary conditions (see Conventions). Since both $R$ and $S$ are concentrated in nonpositive degrees, their derived categories $\D^{+}(S)$ and $\D^{+}(R)$ can be endowed with the natural $t$-structure of \Cref{ex - standard t-structures} and we define $\B$ and $\A$ as $\hinj^{+}(S)$ and $\hinj^{+}(R)$ respectively. The $t$-structures are left bounded, non-degenerate, have as hearts the module category on the zero-th cohomology of the dg-algebra, and have coaisles that are closed under countable products. Both pretriangulated $t$-dg-categories also have enough derived injectives by \Cref{ex - conmpactly gen non-degenerate} because they are compactly generated (see \cite{Porta}[Corollary 7.3]) and their hearts are Grothendieck. So both $\B$ and $\A$ lie in $\Hqe^{t+}_{\Pi}$. Consider now the restriction dg-functor,
    $$\theta_{*}: \dgm(S) \to \dgm(R),$$
    given by precomposition with $\theta$. 
    %By \cite{OrlovNCSchemes}[\S 1.3], it has a left (and right) adjoint dg-functor $\theta^{*}$, which preserves acyclics. Therefore, $\theta_{*}$ (co)restricts to
    By applying it to h-injective dg-$S$-modules and taking h-injective resolutions, we obtain the quasi-functor
    $$\theta_{*}: \hinj(S) \to \hinj(R).$$
    The induced functor $H^{0}(\theta_{*}): \D(S) \to \D(R)$ is $t$-exact and thus (co)restricts to $\theta_{*}: \B \to \A$. This is a $t$-deformation by \cite{GLVdB2}[(5.9)]. 
    % For $n = 2$, $m = 1$ and $d = 0$, so $\theta: k[\epsilon]/(\epsilon^{2}) \to k$, the induced abelian deformation between the hearts is $$\id: \Vector(k) \to \Mod(k[\epsilon]/(\epsilon^{2}),$$ see \Cref{ex - standard t-structures}. Since $k$ is not flat as a $k[\epsilon]/(\epsilon^{2})$-module, this abelian deformation is seen to be non-flat using \cite{AbelianDeformations}[Proposition 3.7].
\end{example}

\subsection{The induced dg-deformation of the derived injectives} \label{pardefset4}

Our aim is to provide the converse result to \Cref{thm - GLV2}. This requires two things: constructing a dg-deformation in $\Hqe^{\DGInj}_{\Pi}$ out of a $t$-deformation in $\Hqe^{t+}_{\Pi}$ and showing that $\Hqe^{t+}_{\Pi}$ is closed under $t$-deformations. In this subsection, we treat the first point.

Recall from \cite{ThesisFrancesco}[Proposition 7.1] that a quasi-functor $G: \A \to \B$ is a right adjoint quasi-functor to $F: \B \to \A$ if, viewed as a right quasi-representable dg-functor into $\dgm_{R}(\B)$, there is for any $A \in \A$ and $B \in \B$ a natural quasi-isomorphism
\begin{equation} \label{eq - quasiadj}
G(A)(B) \sim \B(B,\Phi_{G}(A)) \sim \A(\Phi_{F}(B),A),% \sim (F(B)(A))^{\op},
\end{equation}
meaning $G \sim \mb{L}L(F)$ cf. \cite{ThesisFrancesco}[\S 7.1]. Consider a $t$-deformation $F: \B \to \A$ and put $\tb{I} = \DGInj \A$ and $\tb{J} = \DGInj \B$. 
Let $F: \B \to \dgm_R(\A)$ be the dg-functor corresponding to the quasi-functor $F$.
We define the dg functor $G: \tb{I} \to \dgm_R(\tb{J})$ in such a way that it could be the restriction of a quasi-functor that is right adjoint to $F$ according to \eqref{eq - quasiadj}. 
More precisely, we put $$G(A)(B) = F^{\op}(B)(A), \tn{ for } A \in \tb{I}, B \in \tb{J}^{\op}.$$

\begin{proposition}\label{proppartialadjoint}
Suppose $\B \in \Hqe^{t+}$.
    The dg-functor $G$ defines a quasi-functor $G: \tn{\tb{I}} \to \tn{\tb{J}}$ such that
    $H^0(G)$ makes the following diagram commute:
\begin{equation}\label{eqinjcomm}
    \xymatrix{{\Inj H^0(\A)^{\heartsuit}} \ar[d]_-{\homom_{H^0(R)}({H^0(S)},-)} \ar[r]^-L_-{\cong} & {H^0(\tn{\tb{I}})} \ar[d]^-{H^0(G)} \\
    {\Inj H^0(\B)^{\heartsuit}} \ar[r]^-L_-{\cong} & {H^0(\tn{\tb{J}}).}
    }
\end{equation}
\end{proposition}

\begin{proof}
    We are to show that $G$ is right quasi-representable.
On a derived injective $A = L(I) \in \A$ and an object $B \in \tb{J} \subseteq \B$, we have that
\begin{align*}
    G(L(I))(B) &= F^{\op}(B)(L(I)) \\
    &\sim \A(\Phi_{F}(B),L(I)) \\
    &\sim \A^{\heartsuit}(H^{0}_{t}\Phi_{F}(B),I) \\
    &\sim \A^{\heartsuit}((H^{0}_{t}B)_{H^{0}(R)},I) \\
    &\sim \B^{\heartsuit}(H^{0}_{t}B,\homom_{H^{0}(R)}(H^{0}(S),I)) \\
    &\sim \B(B,L(\homom_{H^{0}(R)}(H^{0}(S),I))) \\
    &= \tb{J}(B,L(J))
\end{align*}
for $J = \homom_{H^{0}(R)}(H^{0}(S),I)$, where we used \cite{GLVdB2}[(4.51)], which states that $H^{0}(\A_{(S)})^{\heartsuit} \cong H^{0}(\A)^{\heartsuit}_{(H^{0}(S))}$, and the fact that the following diagram of functors commutes
\begin{equation}\label{eqHt}
\begin{tikzcd}
	{H^{0}(\A_{(S)})} & {H^{0}(\A)} \\
	{H^{0}(\A)^{\heartsuit}_{(H^{0}(S)})} & {H^{0}(\A)^{\heartsuit}.}
	\arrow["{H^{0}(F)}", from=1-1, to=1-2]
	\arrow["{H^{0}_{t}}"', from=1-1, to=2-1]
	\arrow["{H^{0}_{t}}", from=1-2, to=2-2]
	\arrow["{(-)_{H^{0}(R)}}", from=2-1, to=2-2]
\end{tikzcd}
\end{equation}
Thus, $L(J)$ is a representing object for $G(L(I))$ and the diagram \eqref{eqinjcomm} is readily seen to commute. 
\end{proof}
 
\begin{thm}\label{thmdgdef}
Consider a $t$-deformation $F: \B \to \A$ in $\Hqe^{t+}_{\Pi}$ and consider the quasi-functor $G: \tn{\tb{I}} \to \tn{\tb{J}}$ from Prop. \ref{proppartialadjoint}.
    \begin{enumerate}
    \item $G$ is a dg-deformation in $\Hqe^{\DGInj}_{\Pi}$;
        \item through the equivalences $\Tw^{+}(\tn{\tb{I}}) \cong \A$ and $\Tw^{+}(\tn{\tb{J}}) \cong \B$, $\Tw^+(G)$ defines a right adjoint quasi-functor to $F$.
    \end{enumerate}
    Given a $t$-deformation $F: \Tw^{+}(\tn{\tb{J}}) \to \Tw^{+}(\tn{\tb{I}})$ with $\tn{\tb{J}} \in \Hqe^{\DGInj}_{\Pi},\tn{\tb{I}} \in \Hqe^{\DGInj}$ and the quasi-functor $G: \tn{\tb{I}} \to \tn{\tb{J}}$ from Prop. \ref{proppartialadjoint}, we have that
    \begin{enumerate}
        \item[(3)] $\Tw^{+}(G)$ defines a right adjoint quasi-functor to $F$.
    \end{enumerate}
\end{thm}

\begin{proof}
We start by showing that the quasi-functor $G$ is a morphism in $\Hqe^{\DGInj}$. Let $L(J) \in H^0(\tb{J})$ be a derived injective. 
We are to show that the $H^0(\tb{I})^{\op}$-module $H^0(\tb{J})(L(J),H^0 G(-))$ is finitely presented. By adjunction, we have that
$$H^0(\tb{J})(L(J),H^0 G(-)) \cong H^0(\A)(H^0 F(L(J)), -).$$ 
Hence, according to the upper isomorphism in \eqref{eqinjcomm}, we may equivalently show that $H^0(\A)(H^0F L(J), L(-))$ is finitely presented as an $(\Inj H^0(\A)^{\heartsuit})^{\op}$-module. Using \eqref{eqHt}, this module is isomorphic to
$$H^0(\A)^{\heartsuit}(J_{H^0(R)}, -),$$
which is indeed finitely presented by \cite{GLVdB1}[(1.3)].

We now continue with (2) and (3). By \Cref{thm - correspondence}, we may extend $G$ to a quasi-functor $\Tw^+(G): \A \to \B$, which is seen to be right adjoint to $F$ using \eqref{eq - quasiadj}.

Finally, we complete the proof of (1) and thus of the theorem. 
Consider hereto the factorization of the quasi-functor $G$ as 
\begin{equation} \label{eq - factorization G}
    \xymatrix{ {\tb{I}} \ar[r]^-{\rho} & {S \otimes^{\mb{L}}_R \tb{I}} \ar[r]^-{\varphi} & {\tb{J}.}
}
\end{equation}
According to \Cref{lemtensinj}, $\rho$ lies in $\Hqe^{\DGInj}_{\Pi}$. Furthermore, since $H^0(\rho)$ is surjective and $G$ lies in $\Hqe^{\DGInj}_{\Pi}$, it follows from \Cref{lemepi} that $\varphi$ also lies in 
$\Hqe^{\DGInj}_{\Pi}$. 
Under the equivalence of \Cref{thm - correspondence}, the quasi-functor $\rho$ corresponds to the right adjoint of a $t$-deformation $$\A_{(S)} \cong \Tw^+(S \otimes^{\mb{L}}_R \tb{I}) \to \Tw^+(\tb{I}) \cong \A.$$
We thus obtain a factorization of $F$ as 
$$\xymatrix{
{\B} \ar[r]_-{\Phi} & {\A_{(S)}} \ar[r]_-{\iota} & {\A,}
}$$
in which $\Phi$ is a $t$-exact equivalence since $F$ is a $t$-deformation. Then follows by \Cref{thm - correspondence} that $\varphi$ is an equivalence, so $G$ is a dg-deformation.
\end{proof}

%All the necessary props and lemmas

The rest of this subsection is comprised of lemmas that were necessary for the proof of \Cref{thmdgdef}. We also used (and will use) the following result for abelian deformations, of which a version for flat deformations is stated in \cite{Obstruction}[Proposition 5.5] and \cite{AbelianDeformations}[Proposition A.15].

\begin{proposition}\label{propinjdef}
Let $\mathcal{C}_0 \to \mathcal{C}$ be a (not necessarily flat) nilpotent abelian deformation with enough injectives along $H^{0}(\theta)$. Then the right adjoint $\homom_{H^{0}(R)}(H^{0}(S),-):  \mathcal{C} \to \mathcal{C}_0$ induces a (not necessarily flat) linear deformation
$\homom_{H^{0}(R)}(H^{0}(S),-):  \Inj(\mathcal{C}) \to \Inj(\mathcal{C}_0)$.
\end{proposition}

\begin{proof}
    The fact that for injectives $E, E' \in \Inj(\mathcal{C})$, the canonical morphism $$H^{0}(S) \otimes_{H^{0}(R)} \mathcal{C}(E,E') \to \mathcal{C}_0(\homom_{H^{0}(R)}(H^{0}(S),E), \homom_{H^{0}(R)}(H^{0}(S),E'))$$ is an isomorphism is proven in \cite{AbelianDeformations}[Proposition A.15]. Essential surjectivity of $\homom_{H^{0}(R)}(H^{0}(S),-)$ is shown in \cite{Obstruction}[Proposition 5.5].
\end{proof}

\begin{lemma}\label{lemtensinj}
    Suppose $\tn{\tb{I}}$ is an $R$-linear dg-category in $\Hqe^{\DGInj}_{\Pi}$. Then
    $S \otimes^{\mb{L}}_R \tn{\tb{I}}$ is an $S$-linear dg-category in $\Hqe^{\DGInj}_{\Pi}$ and the canonical quasi-functor $\rho: \tn{\tb{I}} \to S\otimes^{\mb{L}}_R \tn{\tb{I}}$ lies in $\Hqe^{\DGInj}_{\Pi}$.
\end{lemma}

\begin{proof}
First note that by using truncations of inner homs, it easily follows from the fact that $\tb{I}$ is cohomologically concentrated in nonpositive degrees that the same holds for $S \otimes^{\mb{L}}_R \tb{I}$.
By \cite{GLVdB2}[Corollary 5.3.16], $H^0(\tb{I})$ is a (nilpotent) linear deformation of $H^0(S \otimes^{\mb{L}}_R \tb{I})$ along $H^0(\theta)$.
That $H^0(S \otimes^{\mb{L}}_R \tb{I})$ is additive, left coherent and Karoubian, as well as the fact that $\rho$ will lie in the designated category $\Hqe^{\DGInj}_{\Pi}$ as soon as the codomain does, follows from \Cref{lemkar}.
It remains to show that representable left $S \otimes^{\mb{L}}_R \tb{I}$-modules are homotopically finitely presented. This follows by \cite{GLVdB2}[Lemma 5.3.18].
\end{proof}

\begin{lemma}\label{lemkar}
Suppose $\mathfrak{a}$ is an $H^{0}(R)$-linear category which is additive, left coherent and Karoubian. Then the same properties hold for the $H^{0}(S)$-linear category $H^{0}(S) \otimes_{H^{0}(R)} \mathfrak{a}$. Furthermore, a finitely presented $H^{0}(S) \otimes_{H^{0}(R)} \mathfrak{a}^{\op}$-module restricts to a finitely presented $\mathfrak{a}^{\op}$-module.
\end{lemma}

\begin{proof}
Consider the category $\mathcal{C} = \mathsf{mod}(\mathfrak{a}^{\op})^{\op}$. This is an abelian category with enough injectives and $\mathfrak{a} \cong \Inj(\mathcal{C})$.
According to \cite{AbelianDeformations}[Proposition 4.6], we have a base change
$$\homom_{H^{0}(R)}(H^{0}(S),\mathcal{C}) \cong \mathsf{mod}(H^{0}(S) \otimes_{H^{0}(R)} \mathfrak{a}^{\op})^{\op}$$
and by \cite{AbelianDeformations}[Proposition 4.8 (1)], $\homom_{H^{0}(R)}(H^{0}(S),\mathcal{C})$ is then also abelian with enough injectives.
By Proposition \ref{propinjdef}, we have that
$$\Inj(\homom_{H^{0}(R)}(H^{0}(S), \mathcal{C})) \cong H^{0}(S) \otimes_{H^{0}(R)}\Inj(\mathcal{C}) \cong H^{0}(S) \otimes_{H^{0}(R)}\mathfrak{a}.$$
We conclude by \cite{AbelianDeformations}[Proposition A.13].
\end{proof}

\begin{lemma}\label{lemepi}
Consider quasi-functors
$$\xymatrix{ {\tn{\tb{I}}} \ar[r]^{\rho} & {\tn{\tb{I}'}} \ar[r]^{\varphi} & {\tn{\tb{I}''}}
}$$
with $\tn{\tb{I}}$, $\tn{\tb{I}'}$ and ${\tn{\tb{I}''}}$ in $\Hqe^{\DGInj}$.
If $H^0(\varphi)$ is surjective and $\varphi \rho: \tn{\tb{I}} \to \tn{\tb{I}''}$ lies in  $\Hqe^{\DGInj}$, then $\varphi$ lies in $\Hqe^{\DGInj}$.
\end{lemma}

\begin{proof}
    For $X \in \tb{I}''$, consider $M = H^0(\tb{I}'')(X,\varphi(-))$ as an $H^0(\tb{I}')^{\op}$-module. By the assumption on $\varphi \rho$, the restriction $H^0(\tb{I}'')(X,\varphi \rho(-))$ of $M$ along $\rho$ is a finitely presented $H^0(\tb{I})^{\op}$-module. By \cite{GLVdB2}[Sublemma 5.3.7], it follows that $M$ is finitely presented, as desired.
\end{proof}

\subsection{An equivalence of \texorpdfstring{$t$}{t}-deformations with enough derived injectives and the induced dg-deformations} \label{pardefset5}

The categories $\Hqe^{\DGInj}_{\Pi}$ and $\Hqe^{t+}_{\Pi}$ from the correspondence in \S \ref{pardefset2} play an important role in the study of $t$-deformations. For this reason, apart from the deformation pseudofunctors describing dg-deformations and $t$-deformations respectively, we will consider auxiliary pseudofunctors in which we restrict to those categories. More precisely, we consider:

\begin{itemize}
    \item For an $S$-linear dg-category $\tb{J}$, the pseudofunctor $\Def_{\tb{J}}$ of dg-defor-mations of $\tb{J}$;\index{deformation pseudofunctors!$\Def_{\tb{J}}$}
    \item For an $S$-linear dg-category $\tb{J} \in \Hqe^{\DGInj}_{\Pi}$, the pseudofunctor $\Def_{\tb{J}}^{\DGInj}$ of dg-deformations $\tb{I} \to \tb{J}$ that lie in $\Hqe^{\DGInj}_{\Pi}$;\index{deformation pseudofunctors!$\Def_{\tb{J}}^{\DGInj}$}
    \item For an $S$-linear $t$-dg-category $\B$, the pseudofunctor $\Def^t_{\B}$ of $t$-deformations of $\B$;\index{deformation pseudofunctors!$\Def^t_{\B}$}
        \item For an essentially small bounded $S$-linear $t$-dg-category $\B$, the pseudofunctor $\Def^{t,b}_{\B}$ of bounded $t$-deformations of $\B$;\index{deformation pseudofunctors!$\Def^{t,b}_{\B}$}
    \item For an $S$-linear $t$-dg-category $\B \in \Hqe^{t+}_{\Pi}$, the pseudofunctor $\Def^{t,\DGInj}_{\B}$ of $t$-deformations $\B \to \A$ that have a right adjoint that lies in $\Hqe^{t+}_{\Pi}$.\index{deformation pseudofunctors!$\Def^{t,\DGInj}_{\B}$}
    \item For an $S$-linear left bounded locally coherent Grothendieck $t$-dg-category $\B$, the pseudofunctor $\Def^{t,\tn{+lcG}}_{\B}$ of $t$-deformations $\B \to \A$ so that $\A$ is left bounded locally coherent Grothendieck.\index{deformation pseudofunctors!$\Def^{t,\tn{+lcG}}_{\B}$}
\end{itemize}

\begin{remark} \label{remark - universes size}
    Let $\mc{U}$ be the standard universe.\index{universe $\mc{U}$} All deformation functors above are defined on the category $\mc{U}-{\DGRng}/S$\index{universe $\mc{U}$!$\mc{U}-{\DGRng}/S$} of certain commutative dg-rings over $S$ that are elements of $\mc{U}$ (cf. \Cref{def - UDGRng}). In order to arrive at formally sound definitions, we have to introduce two additional universes $\mc{V}$ and $\mc{W}$ with $\mc{U} \in \mc{V}$ and $\mc{V} \in \mc{W}$. Concretely, we define $\Def_{\B}(R)$ as the $\mc{W}$-groupoid of dg-deformations that are elements of $\mc{V}$. These set-theoretic considerations and the independence of the constructions from the choices of universes are discussed in \Cref{appendix - Universe}.
\end{remark}

Suppose that the $t$-dg-category $\B \in \Hqe^{t+}_{\Pi}$ and the dg-category $\tb{J} \in \Hqe^{\DGInj}_{\Pi}$ correspond under \Cref{thm - correspondence}, ergo $\tb{J} = \DGInj(\B)$ and $\B = \Tw^{+}(\tb{J})$. Then we have the following result,

\begin{proposition}\label{diagramdef}
    There is a commutative diagram of pseudonatural transformations 
    \begin{equation}\label{eq - diagramdef}
        \begin{tikzcd}
	{\Def^{\DGInj}_{\tn{\tb{J}}}} & {\Def^{t,\DGInj}_{\B}} \\
	{\Def_{\tn{\tb{J}}}} & {\Def^{t}_{\B}}
	\arrow["\cong", from=1-1, to=1-2]
	\arrow[from=1-1, to=2-1]
	\arrow[from=1-2, to=2-2]
	\arrow[from=2-1, to=2-2]
\end{tikzcd}
    \end{equation}
    % \begin{equation} \label{eq - diagramdef}
    %     \xymatrix{{\Def^{\DGInj}_{\tn{\tb{J}}}}\ar[r]^{\cong} \ar[d]_{\cong}  &{\Def^{t,\DGInj}_{\B}} \ar[d] \\ {\Def_{\tn{\tb{J}}}} \ar[r] & {\Def^t_{\B}}
    %     }
    % \end{equation}
    in which the upper and left arrow are equivalences. Furthermore, any $\B \to \A$ in $\Def^t_{\B}$ for which $\A \in \Hqe^{t+}_{\Pi}$, lies in the image of the right arrow.
\end{proposition}

\begin{proof}
    The vertical arrows are the obvious inclusions, and the arrow on the left is an equivalence by \Cref{thm - GLV2} and \Cref{lemtensinj}. The lower arrow exists by \Cref{thm - GLV2} and its restriction to the upper arrow exists and is moreover an equivalence by \Cref{thm - GLV2}, \Cref{proppartialadjoint} and \Cref{thm - correspondence}. The claim about the right arrow amounts to the existence of the right adjoint proven in \Cref{proppartialadjoint} and \Cref{thmdgdef}.
\end{proof}

The question now beckons, is the right arrow in \eqref{eq - diagramdef} an equivalence? This does not seem to hold in general. For this reason, we restrict to left bounded locally coherent Grothendieck $t$-dg-categories in the next subsection.

\subsection{An equivalence of \texorpdfstring{$t$}{t}-deformations with enough derived injectives and left bounded locally coherent Grothendieck \texorpdfstring{$t$}{t}-deformations}
\label{pardefset6}

In this subsection, we assume that $K = \Kernel(\theta)$ is nilpotent of degree 2. In that case, $K$ has a dg-$S$-action defined through preimages under $\theta$. Recall that if $K$ is nilpotent of degree $n > 2$, a deformation along $\theta$ can be written as the `composition' of deformations of nilpotency degree 2 (cf. \cite{AbelianDeformations}[Remark 6.2]), so this is a mild condition.

Let $\B$ be a left bounded locally coherent Grothendieck $t$-dg-category, meaning that $\B = \hproj^{+}(\C)$ for $\C \cong \hfp^{b}(\B)$ an essentially small strongly pretriangulated dg-category with a bounded $t$-structure.

\begin{proposition} \label{prop - tlc in hqet+pi}
    $\B$ lies in $\Hqe^{t+}_{\Pi}$.
    \begin{proof}
        By \Cref{prop - properties induced t-structure}, $\B$ is a $t$-dg-category with a left bounded, non-degenerate $t$-structure. Since it is left bounded locally coherent Grothendieck, it has enough derived injectives, see \Cref{cor - lcGt enough DGInj}.
        %By \Cref{prop - heart inddg}, Brown representability and the fact that derived injectives are bounded below, see \Cref{prop - properties dginj}, $\B$ has enough derived injectives. 
        There remains to be shown that the $t$-structure is closed under countable products. This follows by \Cref{remark - closed products} and the fact that $H^{0}(\B) = \D^{+}(\mathcal{C})$ has countable products. 
    \end{proof}
\end{proposition}

\Cref{prop - tlc in hqet+pi} and \Cref{thmdgdef} (2) tell us that there is a pseudonatural transformation
\begin{equation} \label{eq - tlcvsDGINJ}
    \Def^{t,\tn{+lcG}}_{\B} \hookrightarrow \Def^{t,\DGInj}_{\B}.
\end{equation}
The aim is to show that \eqref{eq - tlcvsDGINJ} is an equivalence of deformation pseudofunctors. We thus depart from a $t$-deformation $F = F_{\A}: \B \to \A$ in $\Hqe^{t+}_{\Pi}$ and wish to show that $\A$ is again left bounded locally coherent Grothendieck. We assume that $\A$ is h-flat by replacing it with a h-flat resolution if necessary so that we can use the results of \S \ref{pardefset1}. 

\subsubsection{Lifting along \texorpdfstring{$t$}{t}-deformations with enough derived injectives}

To lift the characteristic properties of the left bounded locally coherent Grothendieck $t$-dg-category $\B$ to $\A$, we need a canonical exact triangle \eqref{eq - base change exact triangle}. This will exist for any $t$-deformation $F: \B \to \A$ in $\Hqe^{t+}$, where $\B$ is strongly pretriangulated and essentially small. In that case, we can apply the homotopy ind-dg-completion and we can enlarge our universe (and use \Cref{remark - strongly pretriang}) to ensure this. Let $\Tw^{+}(G): \A \to \B$ denote the right adjoint quasi-functor to $F$ of \Cref{thmdgdef} (3).

\begin{lemma} \label{lemma - diagram extend basechangetriangle}
    The following diagram of quasi-functors (cf. \Cref{remark - quasif RHom})  commutes,
    \[\begin{tikzcd}[row sep = small]
	{\B \cong \A_{(S)}} && \A \\
	\\
	{\hproj(\A)_{(S)}} && {\hproj(\A).}
	\arrow["{F_{\A}}", shift left, from=1-1, to=1-3]
	\arrow["{\Y_{(S)}}"', hook, from=1-1, to=3-1]
	\arrow["{\Tw^{+}(G)}", shift left, from=1-3, to=1-1]
	\arrow["\Y", hook, from=1-3, to=3-3]
	\arrow["{F_{\hproj(\A)}}", shift left, from=3-1, to=3-3]
	\arrow["{\RHom_{R}(S,-)}", shift left, from=3-3, to=3-1]
\end{tikzcd}\]
    \begin{proof}
        Commutativity of the square in the right direction follows by naturality of the counit of the adjunction $(-)_{R} \dashv (-)_{(S)}$. For the left direction, take $A \in \A$ and $N \in \B$. Then we have natural quasi-equivalences,
        \begin{align*}
            &\quad \hproj(\A)_{(S)}(\Y_{(S)}(N),\RHom_{R}(S,-)\circ \Y(A)) \\
            &\sim \hproj(\A)(F_{\hproj(\A)} \circ \Y_{(S)}(N),\Y(A)) \\
            &\sim \hproj(\A)(\Y \circ F_{\A}(N),\Y(A)) \\
            &\sim \A(F_{\A}(N),A) \\
            &\sim \B(N,\Tw^{+}(G)(A)) \\
            &\sim \hproj(\A)_{(S)}(\Y_{(S)}(N),\Y_{(S)} \circ \Tw^{+}(G)(A)).
        \end{align*}
        Using the base change formula \eqref{eq - base change} and the description in terms of filtered homotopy colimit of representables of \Cref{thm - homotopy dg completion closed}, we can conclude by Yoneda.
    \end{proof}
\end{lemma}

\begin{corollary} \label{cor - basechangetriangle}
    Given an object $A \in \A$, there is an exact triangle\index{base change functors!base change exact triangle}
    \begin{equation} \label{eq - base change exact triangle} \begin{tikzcd}
	{F_{\A} \circ \Tw^{+}(G) (A)} & A & {F_{\A} \circ \Psi(A)} & {}
	\arrow[from=1-1, to=1-2]
	\arrow[from=1-2, to=1-3]
	\arrow["{+}", from=1-3, to=1-4]
\end{tikzcd}\end{equation}
    that is natural in $A$, also called the base change exact triangle. Here, $\Psi$ is the quasi-functor that arises because $$\RHom_{R}(K,-) \circ \Y_{(S)} \circ \Tw^{+}(G): \A \to \hproj(\A)_{(S)}$$
    factorizes over $\B$, i.e. is quasi-isomorphic to $\Y_{(S)} \circ \Psi$ for some quasi-functor $\Psi: \A \to \B$. 
    \begin{proof}
        By \Cref{lemma - diagram extend basechangetriangle} and \Cref{prop - balancing}, we have the exact triangle\footnote{We can interpret $\RHom_{R}(K,-)$ as a quasi-functor to $\hproj(\A)_{(S)}$ because $K$ has a dg-$S$-action due to the nilpotency degree 2.}
        \[\begin{tikzcd}[column sep = small]
	{F_{\hproj(\A)} \circ \RHom_{R}(S,\Y(A))} & {\Y(A)} & {F_{\hproj(\A)} \circ \RHom_{R}(K,\Y(A))} & {\quad.}
	\arrow[from=1-1, to=1-2]
	\arrow[from=1-2, to=1-3]
	\arrow["{+}", from=1-3, to=1-4]
\end{tikzcd}\]
        Using again \Cref{lemma - diagram extend basechangetriangle}, this rewrites as
        \[\begin{tikzcd}[column sep = small]
	{\Y \circ F_{\A} \circ \Tw^{+}(G)(A)} & {\Y(A)} & {F_{\hproj(\A)} \circ \RHom_{R}(K,\Y(A))} & {\quad.}
	\arrow[from=1-1, to=1-2]
	\arrow[from=1-2, to=1-3]
	\arrow["{+}", from=1-3, to=1-4]
\end{tikzcd}\]
        Since the first two objects are representable, so is the third. We can then use \Cref{lemma - diagram extend basechangetriangle} again to conclude that $\RHom_{R}(K,\Y(A))$ lies in the quasi-essential image of $\Y_{(S)}$.
    \end{proof}
\end{corollary}

\begin{lemma} \label{lemma - Psi right t-exact}
    $\Tw^{+}(G): \A \to \B$ and $\Psi: \A \to \B$ are right $t$-exact for the induced $t$-structure.
    \begin{proof}
        $\Tw^{+}(G)$ is right $t$-exact as the right adjoint to a $t$-exact functor. It suffices to show right $t$-exactness of $\RHom_{R}(K,-)$ in \Cref{cor - basechangetriangle} to conclude that $\Psi$ is right $t$-exact via \Cref{cor - basechangetriangle}. 
        Remark first that $$\RHom_{R}(K,-) \sim \RHom_{S}(K,-) \circ \RHom_{R}(S,-).$$ The latter is right $t$-exact as the right adjoint to a $t$-exact functor, so that leaves right $t$-exactness of $\RHom_{S}(K,-): \hproj(\A)_{(S)} \to \hproj(\A)_{(S)}$. Consider an object $X \in (\hproj(\A)_{(S)})_{\geq 1} = (\Y(\A_{\leq 0})^{\perp})_{(S)}$ and $A \in \A_{\leq 0}$. Making use of the quasi-equivalence $\varphi: \hproj_{R}(\A)_{(S)} \cong \hproj_{S}(\A \otimes_{R} S)$ of \eqref{eq - base change}, then
        \begin{align*}
            \RHom_{S}(K,\varphi(X))(A,\diamond) &= \dgm(S)(Q^{\hproj}(K),\varphi(X)(A,\diamond)) \\
            &\sim \dgm(S)(Q^{\hproj}(K),X(\diamond)(A)) \\
            &\sim 0
        \end{align*}
       by our choice of $X$ and $A$.
    \end{proof}
\end{lemma}

Although $\Tw^{+}(G)$ is not left $t$-exact, so it does not (co)restrict to the hearts, we have the following result:

\begin{lemma} \label{lemma - base change exact sequence}
    Let $\homom_{H^{0}(R)}(H^{0}(S),-)$ be the right adjoint to the abelian deformation $F_{\tn{ab}}: H^{0}(\B)^{\heartsuit} \to H^{0}(\A)^{\heartsuit}$ of \cite{AbelianDeformations}[\S 2.3]. Then
    $$H^{0}_{t} \circ \Tw^{+}(G) \cong \tau_{\leq 0} \circ \Tw^{+}(G): H^{0}(\A)^{\heartsuit} \to H^{0}(\B)^{\heartsuit}$$
    is a right adjoint to $F_{\tn{ab}}$ and therefore isomorphic to $\homom_{H^{0}(R)}(H^{0}(S),-)$. Moreover, given an object $A \in H^{0}(\A)^{\heartsuit}$, there is a base change exact sequence in the heart $H^{0}(\A)^{\heartsuit}$,
    \begin{equation} \label{eq - base change triangle heart} \begin{tikzcd}
	{F_{\tn{ab}} \circ H^{0}_{t} \circ \Tw^{+}(G) (A)} & A & {F_{\tn{ab}} \circ H^{0}_{t} \circ \Psi (A).}
	\arrow[hook, from=1-1, to=1-2]
	\arrow[two heads, from=1-2, to=1-3]
\end{tikzcd}\end{equation}
    \begin{proof}
        Let $B \in H^{0}(\B)^{\heartsuit}$ and $A \in H^{0}(\A)^{\heartsuit}$. Then the first claim follows from
        \begin{align*}
            H^{0}(\B)^{\heartsuit}(B,H^{0}_{t}\Tw^{+}(G)(A)) &\cong H^{0}(\B)(B,\tau_{\geq 0} \Tw^{+}(G)(A)) \quad \tn{\textcolor{white}{u}} \iota_{\leq 0} \dashv \tau_{\leq 0} \\
            &\cong H^{0}(\B)(B,\Tw^{+}(G)(A)) \qquad \quad \tn{\Cref{lemma - Psi right t-exact}} \\
            &\cong H^{0}(\A)^{\heartsuit}(F_{\tn{ab}}(B),A).
        \end{align*}
        For the second part, consider the base change exact triangle
        \[\begin{tikzcd}
	{F\circ \Tw^{+}(G) (A)} & A & {F\circ \Psi (A)} & {\quad.}
	\arrow[from=1-1, to=1-2]
	\arrow[from=1-2, to=1-3]
	\arrow["{+}", from=1-3, to=1-4]
\end{tikzcd}\]
        Since $A$ lies in the heart, $F$ is $t$-exact and $\Psi$ is right $t$-exact by \Cref{lemma - Psi right t-exact}, we have that 
        \begin{align*}
            H^{0}(\A)(A,F \circ \Psi(A)) &\cong H^{0}(\A)(A,(\tau_{\leq 0} \circ F \circ \Psi)(A)) \\
            &\cong H^{0}(\A)(A,(F \circ \tau_{\leq 0} \circ \Psi)(A)) \\
             &\cong H^{0}(\A)(A,(F \circ H^{0}_{t} \circ \Psi)(A)).
        \end{align*}
        By plugging in the second map of the base change exact triangle and completing the result to an exact triangle, we obtain the following morphism of exact triangles:
        \[\begin{tikzcd}[row sep = small]
	{F\circ \Tw^{+}(G) (A)} & A & {F\circ \Psi (A)} & { } \\
	Z & A & {F\circ H^{0}_{t} \circ \Psi (A)} & {\quad.}
	\arrow[from=1-1, to=1-2]
	\arrow[from=1-2, to=1-3]
	\arrow[Rightarrow, no head, from=1-2, to=2-2]
	\arrow["{+}", from=1-3, to=1-4]
	\arrow[dashed, from=2-1, to=1-1]
	\arrow[from=2-1, to=2-2]
	\arrow[from=2-2, to=2-3]
	\arrow[from=2-3, to=1-3]
	\arrow["{+}", from=2-3, to=2-4]
\end{tikzcd}\]
        Then $Z \in H^{0}(\A)^{\heartsuit}$ so that the bottom triangle is an exact sequence, and
        \begin{align*}
            Z \cong H^{0}_{t}(Z) 
            &\cong (H^{0}_{t} \circ F \circ \Tw^{+}(G))(A) \\
            &\cong (F \circ H^{0}_{t} \circ \Tw^{+}(G))(A).
        \end{align*}
        We conclude exactness of \eqref{eq - base change triangle heart}.
    \end{proof}
\end{lemma}

Just as in the abelian case, the right adjoint preserves coproducts because $S$ is finitely presented in $\mc{Z}^{0}(\dgm(R))$.

\begin{lemma} \label{lemma - RHom coproducts}
    $\RHom_{R}(S,-): \hproj(\A) \to \hproj(\A)_{(S)}$ preserves coproducts in $\hproj^{+}(\A)$.
    \begin{proof}
        It suffices to show that $\RHom_{R}(S,-): \hproj^{+}(\A) \to \hproj^{+}(\A)$ preserves coproducts as $F$ preserves them and the universal map to a cocone in $\hproj(\A)_{(S)}$ is necessarily dg-$S$-linear.

        The idea is to first write $S$ as a homotopy colimit of perfect dg-$R$-modules. Since $S$ is finitely presented in $\mc{Z}^{0}(\dgm(R))$, we can find an integer $n > 0$ and an exact sequence
        \[\begin{tikzcd}
	{K_{0}} & {R^{\oplus n}} & {S.}
	\arrow[hook, from=1-1, to=1-2]
	\arrow[two heads, from=1-2, to=1-3]
\end{tikzcd}\]
        Since $K$ is finitely generated, we can find an integer $l_{0} > 0$ and an epimorphism $R^{\oplus l_{0}} \twoheadrightarrow K_{0}$ and complete this into a map of exact triangles in $\D(R)_{\leq 0}$:
        \[\begin{tikzcd}[row sep = small]
	{K_{0}} & {R^{\oplus n}} & S & { } \\
	{R^{\oplus l_{0}}} & {R^{\oplus n}} & {C_{0}} & {\quad.}
	\arrow[hook, from=1-1, to=1-2]
	\arrow[two heads, from=1-2, to=1-3]
	\arrow["{+}", from=1-3, to=1-4]
	\arrow[two heads, from=2-1, to=1-1]
	\arrow[from=2-1, to=2-2]
	\arrow[Rightarrow, no head, from=2-2, to=1-2]
	\arrow[from=2-2, to=2-3]
	\arrow["{\gamma_{0}}"', dashed, from=2-3, to=1-3]
	\arrow["{+}", from=2-3, to=2-4]
\end{tikzcd}\]
        Then $H^{0}(\gamma_{0})$ is an isomorphism and $H^{-1}\gamma_{0}$ is an epimorphism. If we thus consider the exact triangle
        \[\begin{tikzcd}
	{K_{1}} & {C_{0}} & S & { \quad,}
	\arrow[from=1-1, to=1-2]
	\arrow["{\gamma_{0}}", from=1-2, to=1-3]
	\arrow["{+}", from=1-3, to=1-4]
\end{tikzcd}\]
        Then $K_{1} \in \D(R)_{\leq -1}$. Since all objects involved are cohomologically finitely presented, we can find $l_{1} > 0$ and an epimorphism $$H^{0}(R)^{\oplus l_{1}} \twoheadrightarrow H^{-1}(K_{1}).$$ Since
        \begin{align*}
             \tn{mod}(H^{0}(R))(H^{0}(R)^{\oplus l_{1}}, H^{-1}(K_{1})) 
            &\cong \bigoplus_{l_{1}} \tn{mod}(H^{0}(R))(H^{0}(R),H^{0}(K_{1}[-1])) \\
            &\cong \bigoplus H^{0}(K_{1}[-1]) \\
            &\cong \bigoplus H^{0}(\D_{\dg}(R)(R,K_{1}[-1])) \\
            &\cong \D(R)(R[1]^{\oplus l_{1}},K_{1}),
        \end{align*}
        we obtain a map $\pi_{1}: R[1]^{\oplus l_{1}} \to K_{1}$ so that $H^{-1}\pi_{1}$ is an epimorphism. As before, we can complete it into a morphism of exact triangles:
        \[\begin{tikzcd}[row sep = small]
	{K_{1}} & {C_{0}} & S & { } \\
	{R[1]^{\oplus l_{1}}} & {C_{0}} & {C_{1}} & {\quad.}
	\arrow[from=1-1, to=1-2]
	\arrow["{\gamma_{0}}", from=1-2, to=1-3]
	\arrow["{+}", from=1-3, to=1-4]
	\arrow["{\pi_{1}}", from=2-1, to=1-1]
	\arrow[from=2-1, to=2-2]
	\arrow[Rightarrow, no head, from=2-2, to=1-2]
	\arrow["{\psi_{1}}"', from=2-2, to=2-3]
	\arrow["{\gamma_{1}}"', dashed, from=2-3, to=1-3]
	\arrow["{+}", from=2-3, to=2-4]
\end{tikzcd}\]
        Then $H^{-i}\gamma_{1}$ is an isomorphism for $i \in \{0,1\}$ and an epimorphism for $i=2$ and $H^{-i}\psi_{1}$ is an isomorphism for $i = 0$ and an epimorphism for $i=1$. We can repeat this argument to construct
        \begin{itemize}
            \item[-] objects $\{ C_{n} \}_{n \geq 0}$ in $\dgm(R)$ so that each $C_{n}$ is perfect;
            \item[-] morphisms $\{ \gamma_{n}: C_{n} \to S \}_{n \geq 0}$ so that $H^{-i}\gamma_{n}$ is an isomorphism for $0 \leq i \leq n$ and an epimorphism for $i = 1$;
            \item[-] morphisms $\{ \psi_{n+1}: C_{n} \to C_{n+1} \}_{n \geq 0}$ so that $H^{-i}\psi_{n+1}$ is an isomorphism for $0 \leq i \leq n$ and an epimorphism for $i = n+1$.
        \end{itemize}
        Then $S \cong \hocolim_{n \geq 0} C_{n}$ because $$H^{-i}(\hocolim_{n \geq 0}C_{n}) \cong \colim_{n \geq 0} H^{-i}(C_{n}) \cong H^{-i}(C_{i}) \cong S$$ and the standard $t$-structure on $\D(R)$ of \Cref{ex - standard t-structures} is non-degenerate. As a consequence, we get that
        \begin{equation} \label{eq - holim RHom coprod}
            \RHom_{R}(S,-) \cong \holim_{n \geq 0} \RHom_{R}(C_{n},-).
        \end{equation}
        Since $C_{n}$ is perfect, we have that
        $$\RHom_{R}(C_{n},-) \cong C_{n}^{\vee} \otimes^{\mb{L}}_{R} -,$$
        which preserves coproducts. However, we cannot get coproducts past the homotopy limit in \eqref{eq - holim RHom coprod}. To solve this, we will prove that for $M$ in $\hproj(\A)_{\geq m}$,
        \begin{equation} \label{eq - goal 1}
            H^{j}_{t}\RHom_{R}(S,M) \cong H^{j}_{t}\RHom_{R}(C_{n},M), \quad j \leq m+n.
        \end{equation}
        Consider hereto the defining exact triangle
        \[\begin{tikzcd}
	{K_{n+1}} & {C_{n}} & S & {\quad.}
	\arrow[from=1-1, to=1-2]
	\arrow["{\gamma_{n}}", from=1-2, to=1-3]
	\arrow["{+}", from=1-3, to=1-4]
\end{tikzcd}\]
        Then $K_{n+1} \in \D(R)_{\leq -n-1}$. By \Cref{prop - balancing}, this gives rise to the exact triangle
        \[\begin{tikzcd}
	{\RHom_{R}(S,M)} & {\RHom_{R}(C_{n},M)} & {\RHom_{R}(K_{n+1},M)} & {\quad}
	\arrow[from=1-1, to=1-2]
	\arrow[from=1-2, to=1-3]
	\arrow["{+}", from=1-3, to=1-4]
\end{tikzcd}\]
        and it suffices to show that $\RHom_{R}(K_{n+1},M) \in \hproj(\A)_{\geq m+n+1}$ to obtain \eqref{eq - goal 1}. Consider hereto $A \in \A_{\leq m+n}$. Then $$\hproj(\A)(\Y(A[n+1]),M) \in \D(R)_{\geq 1}$$
        and also
        $$M(A) \cong \hproj(\A)(\Y(A),M) \in \D(R)_{\geq -n}.$$
        As desired, we find that
        \begin{align*}
            \RHom_{R}(K_{n+1},M)(A) &\cong \homom_{R}(Q^{\hproj}(K_{n+1}),M(A)) \\
            &\cong \D(R)(K_{n+1},M(A)) \\
            &= 0. \qedhere
        \end{align*}
    \end{proof}
\end{lemma}

\begin{lemma} \label{lemma - right adjoint coproducts}
    Assume that $H^{0}_{t}: H^{0}(\A) \to H^{0}(\A)^{\heartsuit}$ preserves coproducts. Then $\Tw^{+}(G): \A \to \B$ preserves coproducts.
    \begin{proof}
        As in \Cref{lemma - RHom coproducts}, it suffices to show that $F \circ \Tw^{+}(G)$ preserves coproducts. Let $\coprod_{i \in I} A_{i}$ be a coproduct in $H^{0}(\A)_{\geq 0}$. Then there is a natural morphism,
        \begin{align*}
            F(\coprod_{i \in I} \Tw^{+}(G)(A_{i})) \xrightarrow{\psi} F \circ \Tw^{+}(G)(\coprod_{i \in I} A_{i}).
        \end{align*}
        Due to the non-degeneracy of the $t$-structure, it suffices to show that $H^{i}_{t}\psi$ is an isomorphism for every $i \in \mb{Z}$, or equivalently that each $H^{i}_{t}\phi$ is an isomorphism with
        \begin{align*}
            \Y \circ F (\coprod_{i \in I} \Tw^{+}(G)(A_{i})) \xrightarrow{\phi} \Y \circ F \circ \Tw^{+}(G)(\coprod_{i \in I} A_{i})
            \cong \RHom_{R}(S,\Y(\coprod_{i \in I} A_{i})).
        \end{align*}
        Using \Cref{lemma - RHom coproducts}, we have that
        \begin{align*}
            H^{i}_{t}\RHom_{R}(S,\Y(\coprod_{i \in I} A_{i})) &\cong H^{i}_{t}\RHom_{R}(C_{n},\Y(\coprod_{i \in I} A_{i})).
        \end{align*}
        for $n$ large enough. Since $C_{n}$ is perfect, it can be written as a finite extension of shifts of finite coproducts of $R$. We can pull these operations into Yoneda and denote the resultion functor by $\widetilde{C_{n}}(-)$. Then
        \begin{align*}
        H^{i}_{t}\RHom_{R}(S,\Y(\coprod_{i \in I} A_{i})) &\cong H^{i}_{t}\RHom_{R}(C_{n},\Y(\coprod_{i \in I} A_{i})) \\
        &\cong H^{i}_{t}\Y(\widetilde{C_{n}}(\coprod_{i \in I} A_{i})) \\
        &\cong H^{i}_{t}\Y(\coprod_{i \in I} \widetilde{C_{n}}(A_{i})) \\
        &\cong \Y(\coprod_{i \in I} H^{i}_{t} \widetilde{C_{n}}(A_{i})) \\
        &\cong \Y(\coprod_{i \in I} H^{i}_{t} (F \circ \Tw^{+}(G))(A_{i})) \\
        &\cong H^{i}_{t}(\Y \circ F)(\coprod_{i \in I} \Tw^{+}(G)(A_{i})),
        \end{align*}
        where we used \eqref{eq - goal 1} in the fifth line.
    \end{proof}
\end{lemma}

\begin{lemma} \label{lemma - countable products}
    $H^{0}(\A)_{\geq 0}$ has countable products if $H^{0}(\B)_{\geq 0}$ has them.
    \begin{proof}
        Let $\{ A_{n} \}_{n \in \mb{N}}$ be a set of objects in $H^{0}(\A)_{\geq 0}$. Then we can consider for each $n$ the base change exact triangle in $H^{0}(\A)_{\geq 0}$,
        \[\begin{tikzcd}
	{F\circ \Tw^{+}(G)(A_{n}) } & {A_{n}} & {F \circ \Psi(A_{n})} & {\quad.}
	\arrow[from=1-1, to=1-2]
	\arrow[from=1-2, to=1-3]
	\arrow["{+}", from=1-3, to=1-4]
\end{tikzcd}\]
        In $\D(\A)$, we then have the exact triangle
        \[\begin{tikzcd}[column sep = small, font = \small]
	{{\displaystyle \prod_{n \in \mb{N}}} \A(-,F\circ \Tw^{+}(G)(A_{n}))} & {{\displaystyle \prod_{n \in \mb{N}}} \A(-,A_{n})} & {{\displaystyle \prod_{n \in \mb{N}}} \A(-,F \circ \Psi(A_{n}))} & {\quad.}
	\arrow[from=1-1, to=1-2]
	\arrow[from=1-2, to=1-3]
	\arrow["{+}", from=1-3, to=1-4]
\end{tikzcd}\]
        Since $\Y$ preserves limits, $F$ preserves products (since it is essentially the evaluation $\ev_{\diamond}$), $\Tw^{+}(G)$ and $\Psi$ are right $t$-exact by \Cref{lemma - Psi right t-exact} and $H^{0}(\B)_{\geq 0}$ has countable products, this rewrites as
        \[\begin{tikzcd}[column sep = small, font = \small]
	{\A(-,F({\displaystyle \prod_{n \in \mb{N}}}\Tw^{+}(G)(A_{n})))} & {{\displaystyle \prod_{n \in \mb{N}}} \A(-,A_{n})} & {\A(-,F ({\displaystyle \prod_{n \in \mb{N}}} \Psi(A_{n})))} & {\quad.}
	\arrow[from=1-1, to=1-2]
	\arrow[from=1-2, to=1-3]
	\arrow["{+}", from=1-3, to=1-4]
\end{tikzcd}\]
        The first and third object are representable. Then so is the second. 
    \end{proof}
\end{lemma}

\begin{remark} \label{remark - hqet+pi not tlc}
    Up to this point, we have not used the fact that $\B$ is left bounded locally coherent Grothendieck. The results above hold for any $t$-deformation of an (essentially small strongly pretriangulated) $t$-dg-category in $\Hqe^{t+}$.
\end{remark}

\subsubsection{Lifting left bounded locally coherent Grothendieck \texorpdfstring{$t$}{t}-dg-categories along \texorpdfstring{$t$}{t}-deformations with enough injectives}

Let $F: \B \to \A$ be a $t$-deformation in $\Hqe^{t+}_{\Pi}$ with $\B$ a (strongly pretriangulated essentially small) left bounded locally coherent Grothendieck $t$-dg-category, as in the beginning of this subsection. We start by lifting Grothendieckness of the heart.

\begin{lemma} \label{lemma - exact filtered colimits}
    The heart $H^{0}(\A)^{\heartsuit}$ has filtered colimits and they are exact.
    \begin{proof}
        This follows from \cite{AbelianDeformations}[Lemmas 6.26 and 6.28], since flatness (which we may not have here) is not used there. One could also adapt the proofs using the base change short exact sequence of \Cref{lemma - base change exact sequence}.
        % We start by showing that $H^{0}(\A)^{\heartsuit}$ has small coproducts and that they are exact. Since $H^{0}(\A)^{\heartsuit}$ has enough injectives, we can reuse the proof of \cite{AbelianDeformations}[Lemma 6.26] by replacing the short exact sequence there by our base change short exact sequence.
        % Since $H^{0}(\A)^{\heartsuit}$ has small coproducts and coequalizers, it is cocomplete so filtered colimits exist. Exactness of them follows by the same proof as \cite{AbelianDeformations}[Lemma 6.28] using again our base change exact sequence instead.
    \end{proof}
\end{lemma}

\begin{proposition} \label{prop - heart G}
    The heart $H^{0}(\A)^{\heartsuit}$ is a Grothendieck abelian category.
    \begin{proof}
        By assumption, $H^{0}(\A)^{\heartsuit}$ has enough injectives. By \Cref{lemma - exact filtered colimits}, it also has exact filtered colimits. There remains to be shown that $H^{0}(\A)^{\heartsuit}$ has a small set of generators. As in the lemmas above, we can reuse the results of \cite{AbelianDeformations}[\S 6.5], namely \cite{AbelianDeformations}[Lemmas 6.23 and 6.24], since flatness is not used there. That $\mathcal{C}_{\kappa}$ is essentially small follows from the first part of \cite{AbelianDeformations}[Lemma 6.8], where flatness is again not used.
    \end{proof}
\end{proposition}

Next, we lift some properties on the triangulated level.

\begin{lemma} \label{lemma - coproducts}
    $H^{0}(\A)_{\geq 0}$ has small coproducts and $H^{0}_{t}: H^{0}(\A) \to H^{0}(\A)^{\heartsuit}$ preserves coproducts.
    \begin{proof}
        Let $\{ X_{i} \}_{i \in I}$ be a set of objects in $H^{0}(\A)_{\geq 0}$. Then we can consider for each $i \in I$ the base change exact triangle in $H^{0}(\A)_{\geq 0}$,
        \begin{equation} \label{eq - triangle in heart} \begin{tikzcd}
	{F\circ \Tw^{+}(G)(X_{i}) } & {X_{i}} & {F \circ \Psi(X_{i})} & {\quad.}
	\arrow[from=1-1, to=1-2]
	\arrow[from=1-2, to=1-3]
	\arrow["{+}", from=1-3, to=1-4]
\end{tikzcd}\end{equation}
        In $\D(\A^{\op})$, we then have the exact triangle
        \[\begin{tikzcd}[column sep = small, font = \small]
	{{\displaystyle \prod_{i \in I}} \A(F\circ \Psi(X_{i}),-)} & {{\displaystyle \prod_{i \in I}}\A(X_{i},-)} & {{\displaystyle \prod_{i \in I}} \A(F \circ \Tw^{+}(G)(X_{i}),-)} & {\quad.}
	\arrow[from=1-1, to=1-2]
	\arrow[from=1-2, to=1-3]
	\arrow["{+}", from=1-3, to=1-4]
\end{tikzcd}\]
        Since $\Y$ preserves limits, $F$ preserves colimits, $\Tw^{+}(G)$ and $\Psi$ are right $t$-exact by \Cref{lemma - Psi right t-exact} and $H^{0}(\B)_{\geq 0}$ has small coproducts, this rewrites as
        \[\begin{tikzcd}[column sep = small, font = \small]
	{\A(F({\displaystyle \coprod_{i \in I}} \Psi(X_{i})),-)} & {{\displaystyle \prod_{i \in I}} \A(X_{i},-)} & {\A(F ({\displaystyle \coprod_{i \in I}} \Tw^{+}(G)(X_{i})),-)} & {\quad.}
	\arrow[from=1-1, to=1-2]
	\arrow[from=1-2, to=1-3]
	\arrow["{+}", from=1-3, to=1-4]
\end{tikzcd}\]
        Since $\Y$ is an embedding, the middle object is also corepresented by an object in $H^{0}(\A)$, which is the coproduct.

        There remains to be shown that $H^{0}_{t}: H^{0}(\A) \to H^{0}(\A)^{\heartsuit}$ preserves coproducts. Consider hereto a set $\{ X_{i} \}_{i \in I}$ of objects in $H^{0}(\A)$ so that $\coprod_{i \in I} X_{i}$ exists in $H^{0}(\A)$. Then we can look to the exact triangle
        \[\begin{tikzcd}[column sep = small]
	{F(\coprod_{i \in I} \Tw^{+}(G) (X_{i}))} & {\coprod_{i \in I} X_{i}} & {F(\coprod_{i \in I} \Psi (X_{i}))} & {\quad}
	\arrow[from=1-1, to=1-2]
	\arrow[from=1-2, to=1-3]
	\arrow["{+}", from=1-3, to=1-4]
\end{tikzcd}\]
        and the long exact $t$-cohomology sequence,
        \begin{equation} \label{eq - LES 1} \begin{tikzcd}[column sep = small, font = \small]
	{...} & {H^{0}_{t}[F(\coprod_{i \in I} \Tw^{+}(G) (X_{i}))]} & {H^{0}_{t}[\coprod_{i \in I} X_{i}]} & {H^{0}_{t}[F(\coprod_{i \in I} \Psi (X_{i}))]} & {...}
	\arrow[from=1-1, to=1-2]
	\arrow[from=1-2, to=1-3]
	\arrow[from=1-3, to=1-4]
	\arrow[from=1-4, to=1-5]
\end{tikzcd}\end{equation}
        Since $H^{0}_{t} \circ F = F_{\tn{ab}} \circ H^{0}_{t}$ and $F, F_{\tn{ab}}$ and $H^{0}_{t}: H^{0}(\B) \to H^{0}(\B)^{\heartsuit}$ preserve coproducts, and since coproducts are exact in $H^{0}(\A)^{\heartsuit}$ by \Cref{lemma - exact filtered colimits}, we find that 
        \begin{equation*}
            H^{0}_{t}(\coprod_{i \in I} X_{i}) \cong \coprod_{i \in I} H^{0}_{t}(X_{i}). \qedhere
        \end{equation*}
    \end{proof}
\end{lemma}

\begin{lemma} \label{lemma - right complete}
    $H^{0}(\A)$ is right complete.
    \begin{proof}
    By \Cref{lemma - coproducts}, we only have to prove condition (3) in \Cref{def - right complete}. Let $A$ be an object in $H^{0}(\A)_{\geq 0}$. By \Cref{lemma - exact filtered colimits} and \Cref{remark - H0t hocolim to colim AB5} (2), $H^{0}_{t}$ maps sequential homotopy colimits to sequential colimits. Then the cone of any map $\hocolim_{n} \tau_{\leq n}A \to A$ has no $t$-cohomology, which by non-degeneracy of the $t$-structure, implies it is zero. Therefore,
    \begin{equation*}
        \hocolim_{n} \tau_{\leq n}A \cong A. \qedhere
    \end{equation*}
%     So consider a diagram $X: \mb{N} \to H^{0}(\A)$ that stabilizes in cohomology, meaning that for each $N \geq 0$, there exists some $n_{N} \geq 0$ so that
%     $$H^{i}_{t}(X_{n}) \cong H^{i}_{t}(X_{n+1}), \forall n \geq n_{N}, i \leq N.$$
%     We then claim that the diagrams $\Tw^{+}(G) \circ X$ and $\Psi \circ X: \mb{N} \to H^{0}(\B)$ stabilize in cohomology. Given $N \geq 0$ and $n \geq n_{N}$, we can consider the short exact triangle
%     \[\begin{tikzcd}
% 	{X_{n}} & {X_{n+1}} & {Z_{n}} & {\quad.}
% 	\arrow["{\alpha_{n}}", from=1-1, to=1-2]
% 	\arrow[from=1-2, to=1-3]
% 	\arrow["{+}", from=1-3, to=1-4]
% \end{tikzcd}\]
%     Then $Z_{n} \in H^{0}(\A)_{\geq N+1}$. By right $t$-exactness of $\Tw^{+}(G)$ and $\Psi$, either as a right adjoint to a $t$-exact functor or by \Cref{lemma - Psi right t-exact}, then also $\Tw^{+}(G)(Z_{n})$ and $\Psi(Z_{n})$ lie in the coaisle $H^{0}(\B)_{\geq N+1}$. Consequently, $\tau_{\leq N}\Tw^{+}(G)(\alpha_{n})$ and $\tau_{\leq N}\Psi(\alpha_{n})$ are isomorphisms in $H^{0}(\B)$ so that $\hocolim_{n} \Tw^{+}(G) \circ X$ and $\hocolim_{n} \Psi \circ X$ exist, as $H^{0}(\B)$ is right complete. In $\D(\A^{\op})$, we then have the exact triangle
%     \[\begin{tikzcd}[column sep=small, font=\small]
% 	{\A(F(\hocolim_{n} \Psi(X_{n})),-)} & {\holim_{n} \A(X_{n},-)} & {\A(F(\hocolim_{n} \Tw^{+}(G)(X_{n})),-)} & {\quad.}
% 	\arrow[from=1-1, to=1-2]
% 	\arrow[from=1-2, to=1-3]
% 	\arrow["{+}", from=1-3, to=1-4]
% \end{tikzcd}\]
%     Since the first and third object are corepresentable, so is the second.
    \end{proof}
\end{lemma}

With these results in hand, we can go on to lift locally coherence of the heart. Note that the following lemma is where we divert from the classic approach of \cite{AbelianDeformations}. Indeed, the argument in \cite{AbelianDeformations}[Proposition 6.34] requires flatness of the abelian deformation, something we do not have here (cf. \Cref{ex - example}). This is remedied by the following:

\begin{lemma} \label{lemma - condition higher exts fp coproducts} 
    $\fp(H^{0}(\A))^{\heartsuit} \subseteq H^{0}(\A)^{c}$.
    \begin{proof}
        By considering the base change exact sequence of \Cref{lemma - base change exact sequence} and using \Cref{lemma - hfpb coextension} below, it suffices to show that $$F(\fp(H^{0}(\B)^{\heartsuit})) \subseteq H^{0}(\A)^{c}.$$ So consider $\coprod_{i \in I} A_{i}$ in $H^{0}(\A)$ and $B \in \fp(H^{0}(\B)^{\heartsuit})$. Then
        \begin{align*}
            H^{0}(\A)(F(B),\coprod_{i \in I}A_{i}) &\cong H^{0}(\B)(B, \Tw^{+}(G)(\coprod_{i \in I}A_{i})) \\
            &\cong H^{0}(\B)(B, \coprod_{i \in I} \Tw^{+}(G)(A_{i})) \\
            &\cong \coprod_{i \in I} H^{0}(\B)(B, \Tw^{+}(G)(A_{i})) \\
            &\cong \coprod_{i \in I} H^{0}(\A)(F(B), A_{i})
        \end{align*}
        by $F \dashv \Tw^{+}(G)$, \Cref{lemma - right adjoint coproducts} and $\fp(H^{0}(\B)^{\heartsuit}) \subseteq H^{0}(\B)^{c}$.
    \end{proof}
\end{lemma}

\begin{corollary} \label{cor - 6.35'} 
    Let $B \in \fp(H^{0}(\B)^{\heartsuit})$. Then 
    $H^{1}(\A(F_{\tn{ab}}(B),-))$ preserves filtered colimits in $H^{0}(\A)^{\heartsuit}$.
    \begin{proof} 
        This follows from \Cref{lemma - condition higher exts fp coproducts} (cf. (4') in \Cref{def - left tlc}) and \Cref{lemma - hfpb coextension} below, together with the fact that filtered colimits in $H^{0}(\A)^{\heartsuit}$ can be written in terms of small coproducts and coequalizers.
    \end{proof}
\end{corollary}

\begin{proposition} \label{prop - heart lcG}
    The heart $H^{0}(\A)^{\heartsuit}$ is a locally coherent Grothendieck abelian category.
    \begin{proof}
        By \Cref{prop - heart G}, $H^{0}(\A)^{\heartsuit}$ is a Grothendieck abelian category. For locally coherence, we follow \cite{AbelianDeformations}[Theorem 6.36]. Let $\{ B_{i} \}_{i \in I}$ be a small set of finitely presented generators of $H^{0}(\B)^{\heartsuit}$. The idea is to carry out the construction of \cite{AbelianDeformations}[Lemma 6.23], which does not require flatness, for specific maps $k$. So suppose we constructed the $G^{''}_{f}$ as in \cite{AbelianDeformations}[Lemma 6.23], meaning that for a non-zero map $f: A \to F_{\tn{ab}}(B)$ where $B \in H^{0}(\B)^{\heartsuit}$, we constructed the following commutative diagram:
        \[\begin{tikzcd}[font=\small]
	{A''_{f}} & {A'} & {A'} & A & {B.} \\
	{B_{j}} & {H^{0}(S)\otimes_{H^{0}(R)}A'} & {B_{i}} & {H^{0}(S)\otimes_{H^{0}(R)}A}
	\arrow["{h'}", from=1-1, to=1-2]
	\arrow["{(\eta_{A'})'}"', two heads, from=1-1, to=2-1]
	\arrow["{(P2)}"{description}, draw=none, from=1-1, to=2-2]
	\arrow[Rightarrow, no head, from=1-2, to=1-3]
	\arrow["{\eta_{A'}}"', two heads, from=1-2, to=2-2]
	\arrow["{g'}", from=1-3, to=1-4]
	\arrow["{(\eta_{A})'}"', two heads, from=1-3, to=2-3]
	\arrow["{(P1)}"{description}, draw=none, from=1-3, to=2-4]
	\arrow["f", from=1-4, to=1-5]
	\arrow["{\eta_{A}}"', two heads, from=1-4, to=2-4]
	\arrow["h"', from=2-1, to=2-2]
	\arrow["{(\eta_{A})''}"', from=2-2, to=2-3]
	\arrow["g"', from=2-3, to=2-4]
	\arrow["{f'}"', from=2-4, to=1-5]
\end{tikzcd}\]
        Here we omitted writing the forgetful functor $F_{\tn{ab}}$. Denoting $I := \Kernel(H^{0}(\theta))$, we have the short exact sequence
        \begin{equation*}
            F_{\tn{ab}}(IA') \hookrightarrow A' \twoheadrightarrow F_{\tn{ab}}(H^{0}(S) \otimes_{H^{0}(R)} A')
        \end{equation*}
        and, as $A''_{f}$ is the pullback of $(h,\eta_{A'})$, also the short exact sequence
        \begin{equation*} 
            F_{\tn{ab}}(IA') \hookrightarrow A''_{f} \twoheadrightarrow F_{\tn{ab}}(B_{j})
        \end{equation*}
        or exact triangle,
        \begin{equation} \label{eq - triangle LC} \begin{tikzcd}
	{F_{\tn{ab}}(B_{j}[-1])} & {F_{\tn{ab}}(IA')} & {A''_{f}} & {\quad.}
	\arrow[from=1-1, to=1-2]
	\arrow[from=1-2, to=1-3]
	\arrow["{+}", from=1-3, to=1-4]
\end{tikzcd}\end{equation}
        Since $IA' \in H^{0}(\B)^{\heartsuit}$, it can be written as a filtered colimit $\colim_{i} X_{i}$ of coherent objects $X_{i}$ in $H^{0}(\B)^{\heartsuit}$. Then $F_{\tn{ab}}(IA') = \colim_{i} F_{\tn{ab}}(X_{i})$ is a filtered colimit of finitely presented objects in $H^{0}(\A)^{\heartsuit}$. By \Cref{cor - 6.35'},
        \begin{align*}
            H^{0}(\A)(F_{\tn{ab}}(B_{j})[-1],F_{\tn{ab}}(IA')) &= H^{1}(\A(F_{\tn{ab}}(B_{j}),F_{\tn{ab}}(IA'))) \\
            &\cong \colim_{i} H^{1}(\A(F_{\tn{ab}}(B_{j}),F_{\tn{ab}}(X_{i}))) \\
            &= \colim_{i} H^{0}(\A)(F_{\tn{ab}}(B_{j}[-1]),F_{\tn{ab}}(X_{i})).
        \end{align*}
        So we obtain a morphism of exact triangles in $H^{0}(\A)$ for some $i$,
        \[\begin{tikzcd}[row sep = small]
	{F_{\tn{ab}}(B_{j})[-1]} & {F_{\tn{ab}}(IA')} & {A''_{f}} & { } \\
	{F_{\tn{ab}}(B_{j})[-1]} & {F_{\tn{ab}}(X_{i})} & {A_{f}} & {\quad.}
	\arrow[from=1-1, to=1-2]
	\arrow[from=1-2, to=1-3]
	\arrow["{+}", from=1-3, to=1-4]
	\arrow[Rightarrow, no head, from=2-1, to=1-1]
	\arrow[from=2-1, to=2-2]
	\arrow[from=2-2, to=1-2]
	\arrow[from=2-2, to=2-3]
	\arrow[dashed, from=2-3, to=1-3]
	\arrow["{+}", from=2-3, to=2-4]
\end{tikzcd}\]
        This determines a morphism of short exact sequences in $H^{0}(\A)^{\heartsuit}$,
        \[\begin{tikzcd}[row sep = small]
	{F_{\tn{ab}}(IA')} & {A''_{f}} & {F_{\tn{ab}}(B_{j})} \\
	{F_{\tn{ab}}(X_{i})} & {A_{f}} & {F_{\tn{ab}}(B_{j})}
	\arrow[hook, from=1-1, to=1-2]
	\arrow[two heads, from=1-2, to=1-3]
	\arrow[Rightarrow, no head, from=1-3, to=2-3]
	\arrow[from=2-1, to=1-1]
	\arrow[hook, from=2-1, to=2-2]
	\arrow["k"', dashed, from=2-2, to=1-2]
	\arrow[two heads, from=2-2, to=2-3]
\end{tikzcd}\]
        By \cite{AbelianDeformations}[Lemma 6.23], the $\{ A_{f} \}_{f}$ form a set of generators for $H^{0}(\A)^{\heartsuit}$. All objects of $\{ A_{f} \}_{f}$ are finitely presented since each $F_{\tn{ab}}(X_{i})$ and $F_{\tn{ab}}(B_{j})$ are finitely presented. It is a small set because:
        \begin{itemize}
            \item $\fp(H^{0}(\B)^{\heartsuit})$ is small;
            \item $H^{1}(\A(F(-),-)) \cong H^{1}(\B(-,\Tw^{+}(G)(-)))$. \qedhere
        \end{itemize}
    \end{proof}
\end{proposition}

\begin{thm}
    Let $F: \B \to \A$ be a $t$-deformation in $\Hqe^{t+}_{\Pi}$ of a left bounded locally coherent Grothendieck $t$-dg-category. Then $\A$ is left bounded locally coherent Grothendieck.
    \begin{proof}
    We know that $H^{0}(\A)^{\heartsuit}$ is a locally coherent Grothendieck abelian category by \Cref{prop - heart lcG}, that $H^{0}(\A)$ is right complete by \Cref{lemma - right complete}, that $H^{0}_{t}: H^{0}(\A) \to H^{0}(\A)^{\heartsuit}$ preserves coproducts by \Cref{lemma - coproducts} and that $\fp(H^{0}(\A)^{\heartsuit}) \subseteq H^{0}(\A)^{c}$ by \Cref{lemma - condition higher exts fp coproducts}. We can thus invoke \Cref{thm - categorical hfpb inddg}. 
    \end{proof}
\end{thm}

\begin{corollary}
    \label{cor - equivalence t-def t-locally coherent} 
    Suppose that $\B$ is a left bounded locally coherent Grothendieck $t$-dg-category.
    Then there is an equivalence of deformation pseudofunctors
    $$\Def^{t,\tn{+lcG}}_{\B} \cong \Def^{t,\DGInj}_{\B}.$$
\end{corollary}

\subsection{An equivalence of bounded \texorpdfstring{$t$}{t}-deformations of bounded \texorpdfstring{$t$}{t}-dg-categories and left bounded locally coherent Grothendieck \texorpdfstring{$t$}{t}-deformations of their dg-derived categories}
\label{pardefset7}

In this subsection we will show that one can always move to the setting of \S \ref{pardefset6} by using the equivalence of \Cref{thm - categorical hfpb inddg}.

\begin{thm} \label{thm - compatibility hoinddg coextension}
Let $\A$ be an essentially small h-flat strongly pretriangulated $R$-linear dg-category with a bounded $t$-structure. Then there is a quasi-equivalence
    $$\Ind^{\dg,Q,+}(\A_{(S)}) \cong \Ind^{\dg,Q,+}(\A)_{(S)}$$
    that is $t$-exact for the induced $t$-structure and the $t$-structure of \eqref{eq - tstructure coextension}.
    \begin{proof}
        Consider the embedding $\Y: \A \hookrightarrow \hproj^{+}(\A)$. Then we can apply the coextension of scalars $(-)_{(S)}$ to obtain a dg-functor
        $$\Y_{(S)}: \A_{(S)} \to \hproj^{+}(\A)_{(S)}.$$
        By the base change formula of \eqref{eq - base change} and the proof of \Cref{thm - homotopy dg completion closed}, the codomain has all filtered homotopy colimits, so we can apply the universal property of \Cref{cor - universal prop inddgho} to find a dg-functor $\Gamma := \Ind^{\dg,Q,+}(\Y_{(S)})$ that makes the diagram below commute,
        \[\begin{tikzcd}[row sep = small]
	{\A_{(S)}} & {\hproj^{+}(\A)_{(S)}.} \\
	{\Ind^{\dg,Q,+}(\A_{(S)})}
	\arrow["{\Y_{(S)}}", from=1-1, to=1-2]
	\arrow["\Y"', hook, from=1-1, to=2-1]
	\arrow["\Gamma"', dashed, from=2-1, to=1-2]
\end{tikzcd}\]
        From the proof of \Cref{cor - universal prop enriched}, we know that $\Gamma$ is defined as
        \[\begin{tikzcd}
	{\Gamma:\colim{}^{Q(c_{P}^{\op})} \Y \circ X} & {\colim{}^{Q(c_{P}^{\op})} \Y_{(S)} \circ X.}
	\arrow[maps to, from=1-1, to=1-2]
\end{tikzcd}\]

        Let us check that $\Gamma$ is a $t$-exact quasi-equivalence. 
        
        To see that $\Gamma$ is quasi-faithful, observe the commutative diagram below, where $F_{\A}$ and $F_{\Ind^{\dg,Q,+}(\A)}$ are the counits of \eqref{eq - coext}.
            \begin{equation} \label{eq - cdiagram qf} \begin{tikzcd}
	{\A} & {\hproj^{+}(\A)} \\
	{\A_{(S)}} & {\Ind^{\dg,Q,+}(\A_{(S)})} & {\hproj^{+}(\A)_{(S)}}
	\arrow["{\Y}"', hook, from=2-1, to=2-2]
	\arrow["{F_{\A}}", from=2-1, to=1-1]
	\arrow["{\Y}", hook, from=1-1, to=1-2]
	\arrow["{\D^{+}(F_{\A})}", from=2-2, to=1-2]
	\arrow["\Gamma"', from=2-2, to=2-3]
	\arrow["{F_{\hproj^{+}(\A)}}"', from=2-3, to=1-2]
\end{tikzcd}\end{equation}
            Since the counits are quasi-faithful, so is $\D^{+}(F_{\A})$ and thus also $\Gamma$. 
            
            Next, we show that $\Gamma$ is quasi-full. By \Cref{thm - homotopy dg completion closed} (2), it suffices to consider
            \begin{align*}
                f \in &\Ind^{\dg,Q,+}(\A)_{(S)}(\colim{}^{Q(c_{P^{\tn{op}}})} \Y_{(S)} \circ X, \colim{}^{Q(c_{P'^{\tn{op}}})} \Y_{(S)} \circ X').
            \end{align*}
            Then $f$ is essentially just an $S$-linear map $f_{\diamond}$, obtained from $f$ by applying $F_{\Ind^{\dg,Q,+}(\A)}$. By commutativity of \eqref{eq - cdiagram qf}, it corresponds to a map
            \begin{align*}
                \tilde{f} \in &\Ind^{\dg,Q,+}(\A)(\D^{+}(F_{\A})(\colim{}^{Q(c_{P^{\tn{op}}})} \Y \circ X), \D^{+}(F_{\A})(\colim{}^{Q(c_{P'^{\tn{op}}})} \Y \circ X'))
            \end{align*}
            that is $S$-linear, which means it has a preimage under $\D^{+}(F_{\A})$. This preimage is then mapped to $f$ by $\Gamma$. We conclude that $\Gamma$ is quasi-fully faithful.
            
            We continue with (right) $t$-exactness of $H^{0}(\Gamma)$. So consider  $\colim{}^{Q(c_{P}^{\op})} \Y \circ X$ in the coaisle $\D^{+}(\A_{(S)})_{\geq 1}$. Then $$\colim{}^{Q(c_{P}^{\op})} \Y \circ X \cong \colim{}^{Q(c_{P}^{\op})} \tau_{\geq 1} \circ X.$$
            Because cohomology preserves filtered colimits and quasi-functors $F$ in $\hproj^{+}(\A)_{(S)}$ yield functors $H^{0}(F): H^{0}(S) \to \D^{+}(\A)$, we have that $$(\colim{}^{Q(c_{P}^{\op})} \Y_{(S)} \circ X)(\diamond) \sim \Y(\colim{}^{Q(c_{P}^{\op})} F_{\A} \circ X).$$
            By definition of the induced $t$-structure on $\D^{+}(\A)$, it suffices to remark that
            \begin{align*}
                \colim{}^{Q(c_{P}^{\op})} F_{\A} \circ X &= \D^{+}(F_{\A})(\colim{}^{Q(c_{P}^{\op})} \Y \circ X) \\
                &\cong \D^{+}(F_{\A})(\colim{}^{Q(c_{P}^{\op})} \tau_{\geq 1} \circ X) \\
                &\cong \colim{}^{Q(c_{P}^{\op})} F_{\A} \circ \tau_{\geq 1} \circ X \\
                &\cong \colim{}^{Q(c_{P}^{\op})} \tau_{\geq 1} \circ F_{\A} \circ X,
            \end{align*}
            where we used $t$-exactness of $F_{\A}$ and $\D^{+}(F_{\A})$ (see \Cref{cor - D(F) t-exact}). The argument for left $t$-exactness is analogous. 
            
         Finally, to see that $\Gamma$ is quasi-essentially surjective, consider an object $N$ in $H^{0}(\Ind^{\dg,Q,+}(\A)_{(S)})$ (we once again use \Cref{thm - homotopy dg completion closed} (2)). Note that by \Cref{prop - properties induced t-structure} and \eqref{eq - tstructure coextension}, the $t$-structure is non-degenerate, and by \Cref{prop - Lurie} also right complete. Hence,
            $$N \cong \hocolim_{n \geq 0} \tau_{\leq n}N.$$
            Since $H^{0}(\Gamma)$ preserves filtered homotopy colimits, it suffices to show that each $\tau_{\leq n}N$ lies in the essential image. Since $H^{0}(\Gamma)$ also preserves shifts, we can assume that $n = 0$. All in all, the claim is reduced to showing essential surjectivity of $H^{0}(\Gamma)$ on $P \in (\Ind^{\dg,Q,b}(\A)_{(S)})_{\leq 0}$. The strategy is to approximate $P$ by objects in the essential image via an inductive argument. Consider hereto $M \geq 0$ so that $\tau_{\leq -M}P = 0$. Then $\tau_{\leq -M+1}P \cong H^{-M+1}_{t}P$. Consequently, $$\tau_{\leq -M+1}P \in H^{0}(\Ind^{\dg,Q,+}(\A)_{(S)})^{\heartsuit}.$$ Since $H^{0}(\Gamma)$ induces equivalences
            \begin{align*}
                H^{0}(\Ind^{\dg,Q,+}(\A)_{(S)})^{\heartsuit} &\cong H^{0}(\Ind^{\dg,Q,+}(\A))^{\heartsuit}_{H^{0}(S)} \\
                &\cong (\Ind H^{0}(\A)^{\heartsuit})_{H^{0}(S)} \\
                &\cong \Ind (H^{0}(\A)^{\heartsuit}_{H^{0}(S)}) \\
                &\cong \Ind H^{0}(\A_{(S)})^{\heartsuit} \\
                &\cong H^{0}(\Ind^{\dg,Q,+}(\A_{(S)}))^{\heartsuit}.
            \end{align*}
            using also \cite{AbelianDeformations}[Proposition 4.5], $\tau_{\leq -M+1}P$ lies in the essential image of $H^{0}(\Gamma)$. We can extend this by considering the exact triangle
            \[\begin{tikzcd}[row sep = small]
	{\tau_{\leq -M+1}P} & {\tau_{\leq -M+2}P} & {H^{-M+2}_{t}P[M-2]} & {\quad.}
	\arrow["{+}", from=1-3, to=1-4]
	\arrow[from=1-2, to=1-3]
	\arrow[from=1-1, to=1-2]
\end{tikzcd}\]
            Since the first and third object lie in the essential image of $H^{0}(\Gamma)$, so does the second. Using this argument, we can build our way up to $\tau_{\leq 0} P \cong P$. This shows that $H^{0}(\Gamma)$ is essentially surjective. \qedhere
    \end{proof}
\end{thm}

\begin{remark}
    Note that the quasi-functor $\Gamma$ is still defined, $t$-exact and quasi-fully faithful when considering $\Ind^{\dg,Q,+}(\mathcal{A})$ for any essentially small h-flat strongly pretriangulated $t$-dg-category $\A$. For the quasi-essentially surjectivity however, we require non-degeneracy and right completeness as in the setting of \Cref{prop - properties induced t-structure}.
\end{remark}

The converse to \Cref{thm - compatibility hoinddg coextension} -- keeping in mind \Cref{thm - categorical hfpb inddg} -- was proven in \cite{GLVdB2}[(5.5)].

\begin{lemma} \label{lemma - hfpb coextension}
    Let $\A$ be a pretriangulated $R$-linear $t$-dg-category. Then there is a $t$-exact quasi-equivalence
    $$\hfp^{b}(\A_{(S)}) \cong \hfp^{b}(\A)_{(S)}.$$
\end{lemma}

\begin{thm} \label{thm - equivalent deformation pseudofunctors bounded} 
    Let $\B$ be an essentially small strongly pretriangulated $S$-linear dg-category with a bounded $t$-structure. Then there is an equivalence of deformation pseudofunctors
    $$\Def^{t,b}_{\B} \xlongrightarrow{\cong} \Def^{t,\tn{+lcG}}_{\Ind^{\dg,Q,+}(\B)}.$$
    \begin{proof}
        Given a $t$-deformation $F: \B \to \A$ in $\Def^{t,b}_{\B}$, we can assume that $\A$ is strongly pretriangulated by \Cref{remark - not strongly pretriangulated}, h-flat by cofibrantly replacing it and essentially small by \Cref{prop - universes} (3). By \Cref{thm - compatibility hoinddg coextension}, $F$ induces a $t$-deformation,
        $$\overline{F} := \Ind^{\dg,Q,+}(F): \Ind^{\dg,Q,+}(\B) \to \Ind^{\dg,Q,+}(\A).$$
        Mapping $F$ to $\overline{F}$ (or rather their equivalence classes) gives rise to a pseudonatural transformation $$\lambda: \Def^{t,b}_{\B} \to \Def^{t,\tn{+lcG}}_{\Ind^{\dg,Q,+}(\B)}.$$ The question is then whether for each $\theta:R \to S \in \mc{U}-\DGRng/S$, $\lambda_{\theta}$ is an equivalence of categories. It suffices to give a quasi-inverse. We hereto combine \Cref{lemma - hfpb coextension} with the equivalence of \Cref{thm - categorical hfpb inddg} to construct from $\ol{F}$ the $t$-deformation,
        $$\hfp^{b}(\Ind^{\dg,Q,+}(\B)) \to \hfp^{b}(\Ind^{\dg,Q,+}(\A)).$$
        This defines the quasi-inverse $\lambda_{\theta}^{-1}$ to $\lambda_{\theta}$. %One can check that the $\lambda_{\theta}^{-1}$ glue to a pseudonatural transformation $\lambda^{-1}: \Def^{*}_{\B^{b}} \to \Def^{*}_{\B}$.
    \end{proof}
\end{thm}

We recapitulate using the deformation equivalences of this chapter.

\begin{corollary} \label{cor - recap defequiv}
    Let $\B$ be an essentially small strongly pretriangulated $S$-linear dg-category with a bounded $t$-structure and $\tn{\tb{J}} = \DGInj(\hproj^{+}(\B))$. There is a commutative diagram of equivalences of deformation pseudofunctors,
    \[\begin{tikzcd}[row sep = small]
	{\Def^{\DGInj}_{\tn{\tb{J}}}} & {\Def^{t,\DGInj}_{\hproj^{+}(B)}} & {\Def^{t,\tn{+lcG}}_{\hproj^{+}(B)}} \\
	{\Def_{\tn{\tb{J}}}} && {\Def^{t,b}_{\B}.}
	\arrow["\cong", from=1-1, to=1-2]
	\arrow["\cong"', from=1-1, to=2-1]
	\arrow["\cong"', from=1-3, to=1-2]
	\arrow["\cong"', from=2-3, to=1-3]
\end{tikzcd}\]
    \begin{proof}
        Combine \Cref{diagramdef}, \Cref{cor - equivalence t-def t-locally coherent} and \Cref{thm - equivalent deformation pseudofunctors bounded}.
    \end{proof}
\end{corollary}

If we take $S$ to be a field $k$, we can understand the Hochschild cohomology groups of $\B$ in terms of bounded $t$-deformations along an appropriate dg-$k$-algebra morphism.

\begin{corollary} \label{cor - HHn bounded tdg body}
    Let $\B$ be an essentially small strongly pretriangulated dg-$k$-category with a bounded $t$-structure. For $n \geq 2$,
    \begin{equation*}
        \tn{HH}^{n}_{\dg}(\B) \cong \begin{Bmatrix} \tn{bounded $t$-deformations} \\ \tn{of $\B$ along } \theta_{2-n} \end{Bmatrix}/ \sim
    \end{equation*}
    where $\theta_{2-n}: k[\epsilon]/(\epsilon^{2}) \to k$ with $\vert \epsilon \vert = 2-n$.
    \begin{proof}
        See \Cref{cor - HHn bounded tdg} in \Cref{appendix - Hochschild}.
    \end{proof}
\end{corollary}

\ver

% Appendix
\appendix
\counterwithout{thm}{subsection} 
\counterwithin{thm}{section}
\section{Universes and preservation of size} \label{appendix - Universe}

Let $\mc{U}$ be a universe.

\begin{definition} \label{def - UDGRng}
    We denote by $\mc{U}-\DGRng$ the category with as data:
    \begin{itemize}
        \item[-] objects are commutative dg-rings in $\mc{U}$, meaning that each graded piece lies in $\mc{U}$, that are strictly concentrated in nonpositive degrees and homotopically coherent;
        \item[-] morphisms are strictly surjective dg-ring maps $\theta: R \to S$ satisfying the conditions set forth in the Conventions.
    \end{itemize}
    For a fixed object $S$ in $\mc{U}-\DGRng$, we can then consider the slice category $\mc{U}-\DGRng/S$. 
\end{definition}

\begin{definition} \label{def - u-small dgcat}
    An $S$-linear dg-category $\tb{J}$ is an $S$-linear \ti{$\mc{U}$-dg-category} if each graded piece of any inner hom of $\tb{J}$ lies in $\mc{U}$. It is \ti{$\mc{U}$-small} if moreover $\Ob(\tb{J}) \in \mc{U}$ and \ti{essentially $\mc{U}$-small} if it is quasi-equivalent to a $\mc{U}$-small category. \index{universe $\mc{U}$!$\mc{U}$-dg-category} \index{universe $\mc{U}$!$\mc{U}$-small} \index{universe $\mc{U}$!essentially $\mc{U}$-small}
\end{definition}

Consider a universe $\mc{V}$ so that $\mc{U} \in \mc{V}$ and a universe $\mc{W}$ so that $\mc{V} \in \mc{W}$.

\begin{definition}
    Let $\tb{J}$ be an $S$-linear $\mc{U}$-dg-category, $\B$ be an $S$-linear $\mc{U}$-$t$-dg-category in $\Hqe^{t+}_{\Pi}$ and $\B'$ be a $\mc{U}$-small bounded $S$-linear $t$-dg-category. Assume moreover that both are essentially $\mc{V}$-small. Let $\theta: R \to S$ be an object in $\mc{U}-\DGRng/S$. We define the groupoids:
    \begin{itemize}
        \item The $\mc{W}$-groupoid $\mc{V}-\Def_{\tb{J}}(\theta:R \to S)$ with as objects the dg-deformations of $\tb{J}$ along $\theta$ that are elements of $\mc{V}$ and as morphisms the equivalences of dg-deformations modulo natural isomorphisms of quasi-functors. An \ti{equivalence (resp. isomorphism) of dg-deformations} $G_{1}: \tb{I}_{1} \to \tb{J}$ to $G_{2}: \tb{I}_{2} \to \tb{J}$ of $\tb{J}$ over $\theta$ is a quasi-equivalence of $R$-linear dg-categories $\Gamma: \tb{I}_{2} \to \tb{I}_{1}$ so that $G_{1} \circ \Gamma$ and $G_{2}$ are naturally isomorphic (resp. equal);

        \item Similarly, the $\mc{W}$-groupoid $\mc{V}-\Def_{\tb{J}}^{\DGInj}(\theta:R \to S)$ of dg-deformations of $\tb{J}$ in $\Hqe^{\DGInj}_{\Pi}$ along $\theta$ that are elements of $\mc{V}$;

        \item The $\mc{W}$-groupoid $\mc{V}-\Def^{t}_{\B}(\theta:R \to S)$ with as objects the $t$-deformations $F: \B \to \A$ of $\B$ along $\theta$ that are elements of $\mc{V}$. Its morphisms are the equivalences of $t$-deformations modulo the natural isomorphisms of quasi-functors. An \ti{equivalence (resp. isomorphism) of $t$-deformations} $F_{1}: \B \to \A_{1}$ to $F_{2}: \B \to \A_{2}$ of $\B$ over $\theta$ is a $t$-exact quasi-equivalence of $R$-linear dg-categories $\Lambda: \A_{1} \to \A_{2}$ so that $\Lambda \circ F_{1}$ and $F_{2}$ are naturally isomorphic (resp. equal);

        \item Similarly, the $\mc{W}$-groupoid $\mc{V}-\Def_{\B'}^{t,b}(\theta:R \to S)$ of bounded $t$-deformations of $\B'$ along $\theta$ that are elements of $\mc{V}$;

        \item The $\mc{W}$-groupoid $\mc{V}-\Def_{\B}^{t,\DGInj}(\theta:R \to S)$ of $t$-deformations of $\B$ in $\Hqe^{t+}_{\Pi}$ along $\theta$ that are elements of $\mc{V}$;

        \item If $\B$ is left bounded locally coherent Grothendieck, the $\mc{W}$-groupoid $\mc{V}-\Def_{\B}^{t,\tn{+lcG}}(\theta:R \to S)$ of $t$-deformations $F:\B \to \A$ of $\B$ in $\Hqe^{t+}_{\Pi}$ along $\theta$ that are elements of $\mc{V}$ and so that $\A$ is left bounded locally coherent Grothendieck.
    \end{itemize}
\end{definition}

These give rise to pseudofunctors,
\begin{itemize}
    \item For an $S$-linear $\mc{U}$-dg-category $\tb{J}$, the pseudofunctor
    $$\mc{V}-\Def_{\tb{J}}: \mc{U}-\DGRng/S \to \mc{W}-\Gpd$$ of dg-deformations of $\tb{J}$ in $\mc{V}$;
    \item For an $S$-linear $\mc{U}$-dg-category $\tb{J} \in \Hqe^{\DGInj}_{\Pi}$, the pseudofunctor 
    $$\mc{V}-\Def_{\tb{J}}^{\DGInj}: \mc{U}-\DGRng/S \to \mc{W}-\Gpd$$ of dg-deformations $\tb{I}$ in $\mc{V}$ for which $\tb{I} \to \tb{J}$ lies in $\Hqe^{\DGInj}_{\Pi}$;
    \item For an $S$-linear $\mc{U}$-$t$-dg-category $\B$, the pseudofunctor 
    $$\mc{V}-\Def^t_{\B}: \mc{U}-\DGRng/S \to \mc{W}-\Gpd$$ of $t$-deformations in $\mc{V}$;
    \item For a $\mc{U}$-small $S$-linear $t$-dg-category $\B$, the pseudofunctor 
    $$\mc{V}-\Def^{t,b}_{\B}: \mc{U}-\DGRng/S \to \mc{W}-\Gpd$$ of bounded $t$-deformations in $\mc{V}$;
    \item For an $S$-linear $\mc{U}$-$t$-dg-category $\B \in \Hqe^{t+}_{\Pi}$, the pseudofunctor $$\mc{V}-\Def^{t,\DGInj}_{\B}: \mc{U}-\DGRng/S \to \mc{W}-\Gpd$$ of $t$-deformations $\A$ in $\mc{V}$ for which $\B \to \A$ has a left adjoint in $\Hqe^{t+}_{\Pi}$;
    \item For an $S$-linear left bounded locally coherent Grothendieck $\mc{U}$-$t$-dg-category $\B$, the pseudofunctor 
    $$\mc{V}-\Def^{t,\tn{+lcG}}_{\B}: \mc{U}-\DGRng/S \to \mc{W}-\Gpd$$
    of $t$-deformations $\B \to \A$ in $\mc{V}$ so that $\A$ is left bounded locally coherent Grothendieck.
\end{itemize}

The following proposition allows us to restate the results of \S \ref{pardefset5}-\ref{pardefset7} in terms of these pseudofunctors. As in \S \ref{pardefset6}, we assume that $K = \Kernel(\theta)$ is nilpotent of degree 2 so that $K$ is a dg-$S$-module. As recalled there, the results can be extended to a nilpotency degree $n > 2$.

\begin{proposition} \label{prop - universes}
    Let $\theta: R \to S \in \mc{U}-\DGRng/S$.
    \begin{enumerate}
        \item Let $\tn{\tb{J}}$ be an $S$-linear $\mc{U}$-dg-category (resp. essentially $\mc{U}$-small) and $G: \tn{\tb{I}} \to \tn{\tb{J}}$ be a dg-deformation along $\theta$. Then $\tn{\tb{I}}$ is a $\mc{U}$-dg-category (resp. essentially $\mc{U}$-small).

        \item Let $\mc{B}$ be an $S$-linear pretriangulated $\mc{U}$-$t$-dg-category (resp. essentially $\mc{U}$-small) that lies in $\Hqe^{t+}_{\Pi}$. Then $\DGInj \B$ is $\mc{U}$-dg-category (resp. essentially $\mc{U}$-small) if and only if $\B$ is a $\mc{U}$-$t$-dg-category (resp. essentially $\mc{U}$-small).

        \item Let $\mc{B}$ be an essentially $\mc{U}$-small $S$-linear strongly pretriangulated $t$-dg-category with a bounded $t$-structure and $F: \mc{B} \to \mc{A}$ be a $t$-deformation so that $\A$ is bounded. Then $\mc{A}$ is an essentially $\mc{U}$-small $t$-dg-category.
    \end{enumerate}

    \begin{proof}
    Using the same argument as in \Cref{ex - example}, we can assume nilpotency of degree two. Recall from \cite{SGA4I}[Exposé I.0] that the existence of a universum $\mc{U}$ is equivalent to the existence of a strong inaccessible cardinal $\kappa$ so that we can describe $\mc{U}$ as the closed class $V_{\kappa}$ of cardinals $\lambda < \kappa$.
    
        \begin{enumerate}
            \item   We can assume that $\tb{I}$ is cofibrant and that $\tb{J} = S \otimes_{R} \tb{I}$ because the bar construction stays within $\mc{U}$. Then we have for any $m,n \in \tb{I}$ the short exact sequence
                    $$K \tb{I}(m,n) \hookrightarrow \tb{I}(m,n) \twoheadrightarrow S \otimes_{R} \tb{I}(m,n) = \tb{J}(G(m),G(n)),$$
                    where $K \tb{I}(m,n)$ denotes the image of the map $$K \otimes_{R} \tb{I}(m,n) \to \tb{I}(m,n).$$ Consequently,
                    $$\vert \tb{I}(m,n) \vert \leq \vert K \tb{I}(m,n) \vert \cdot \vert \tb{J}(G(m),G(n)) \vert.$$
                    It thus suffices to show that $\vert K \tb{I}(m,n) \vert < \kappa$. By our conventions, $K$ is finitely presented in $\mc{Z}^{0}(\dgm(R))$ and thus also in $\mc{Z}^{0}(\dgm(S))$. It is then certainly finitely generated, so we can find an epimorphism $S^{n} \twoheadrightarrow K$. Applying $- \otimes_{S} (S \otimes_{R} \tb{I}(m,n))$ to it yields
                    $$(S \otimes_{R} \tb{I}(m,n))^{n} \twoheadrightarrow K \otimes_{R} \tb{I}(m,n).$$
                    Then
                    $$ \vert K \tb{I}(m,n) \vert \leq \vert K \otimes_{R} \tb{I}(m,n) \vert \leq \vert (S \otimes_{R} \tb{I}(m,n))^{n} \vert < \kappa.$$

                    Now assume that $\tb{J}$ is moreover $\mc{U}$-small, i.e. that $\Ob(\tb{J}) \in \mc{U}$. We claim that
                    $$\vert \Sk(\tb{J}) \vert = \vert \Sk(\tb{I}) \vert.$$ To prove this, it suffices to show that $G$ reflects isomorphisms. Consider hereto a morphism $h: m \to n$ in $\tb{I}$ so that $G(h)$ is an isomorphism in $\tb{J}$. Then we can find an inverse, which by fullness, is $G(g)$ for a morphism $g: n \to m$ in $\tb{I}$. Then $g \circ h - \id_{m}$ and $h \circ g - \id_{n}$ lie in $K \tb{I}(m,n)$. Due to the nilpotency of degree two, this implies that 
                    \begin{align*}
                        (g \circ h - \id_{m})^{2} &= 0, \\
                        (h \circ g - \id_{n})^{2} &= 0.
                    \end{align*}
                    % $(g \circ h - \id_{m})^{2} = 0$ and $(h \circ g - \id_{n})^{2} = 0$. 
                    These equations yield the inverse $2 g - g \circ h \circ g$ for $h$. Consequently, the cardinalities of the skeletons are the same and thus strictly smaller than $\kappa$.
                    
            \item   By \Cref{thm - correspondence}, we know that $\B = \Tw^{+}(\DGInj \B)$. So for any $(X^{\bullet},q_{X}),(Y^{\bullet},q_{Y}) \in \B$, we have that
            $$\B((X^{\bullet},q_{X}),(Y^{\bullet},q_{Y})) = \prod_{k} \prod_{j} \DGInj \B(X^{i},Y^{j}) \in \mc{U}.$$
            The converse is clear. 

            Now for the $\mc{U}$-small case: since $\Ob(\DGInj \B) \subseteq \Ob(\B)$, one direction is again clear. For the other, remark that a bounded above twisted complex in $\DGInj \B$ consists of the data:
            \begin{itemize}
                \item $\forall i \in \mb{Z}$ an object $X_{i} \in \DGInj \B$,
                \item $\forall i,j \in \mb{Z}, i \leq j$ a morphism $X_{i} \to X_{j}$.
            \end{itemize}
            This shows that $\Ob(\B) \in \mc{U}$ if $\Ob(\DGInj \B) \in \mc{U}$. 

            \item Consider the induced $t$-deformation
            $$\hproj^{\dg,Q,+}(\B) \to \hproj^{+}(\A)$$
            using \Cref{thm - compatibility hoinddg coextension}. Note that the theorem still holds, though we do not know whether $\A$ is essentially $\mc{U}$-small (not even a $\mc{U}$-$t$-dg-category), so $\hproj^{+}(\A)$ does not have to be a $\mc{U}$-$t$-dg-category. However, the induced $t$-deformation satisfies the conditions of (2) due to the boundedness of the $t$-structures involved. Since $\B$ is an essentially $\mc{U}$-small $t$-dg-category, $\hproj^{+}(\B)$ is a $\mc{U}$-$t$-dg-category. Statement (2) then tells us that so is $\hproj^{+}(\A)$. Then $\A$ is a $\mc{U}$-$t$-dg-category through $\A \hookrightarrow \hproj^{+}(\A)$.

            Now for the $\mc{U}$-smallness, by \Cref{lemma - base change exact sequence}, an object $A \in \A$ is determined by 
            a triple $(B,B',f)$ for 
            \begin{itemize}
                \item objects $B, B' \in \B$,
                \item a morphism $f \in H^{1}(\A(F(B),F(B')))$,
             \end{itemize}
            by using the base change exact triangle \eqref{eq - base change exact triangle} on the level of $\hproj$. Since $\A$ is a $\mc{U}$-$t$-dg-category by the previous, we know that $$\vert H^{1}(\A(F(B),F(B'))) \vert < \kappa,$$ so that $\A$ is essentially $\mc{U}$-small. \qedhere
        \end{enumerate}
    \end{proof}
\end{proposition}

\ver

\section{Brown representability in the left bounded setup} \label{appendix - Brown}

The following is an adaptation of \cite{NeemanGrothendieckDuality}[Theorem 3.1].

\begin{thm} \label{thm - Brown}
    Let $\T$ be a triangulated category with a $t$-structure $(\T_{\leq 0}, \T_{\geq 0})$. Suppose that $\T_{\geq 0}$ is closed under countable coproducts and that $\T$ is generated by a set $S \subseteq \T^{c}$ in the sense that $S^{\perp} = 0$. Let $$H: \T^{\op} \to \Ab$$ be a cohomological functor that maps coproducts to products and so that $H \circ \tau_{\geq 0} \cong H$. Then $H$ is representable and the representing object lies in $\T_{\geq 0}$.
    \begin{proof}
        The idea is to construct the representable object as a sequential homotopy colimit of a diagram. First, we define
        $$S_{\geq 0}:= \{ \tau_{\geq 0}s \mid s \in S\} \subseteq \T_{\geq 0}.$$
        We claim that $S_{\geq 0}$ generates $\T_{\geq 0}$. Consider hereto an object $X \in \T_{\geq 0}$ and assume that 
        $$\forall s \in S_{\geq 0}: \T(s,X) = 0.$$
        Using the adjoint pair $\tau_{\geq 0} \dashv \iota_{\geq 0}$, we then have that
        $$\forall s \in S: \T(s,X) \cong \T(\tau_{\geq 0}s,X) = 0.$$
        Consequently, $X = 0$ so that $S_{\geq 0}$ generates $\T_{\geq 0}$. We now start constructing the diagram $\mb{N} \to \T_{\geq 0}$. 

        \ver
        
        \tb{The base case:} Let
        \begin{align*}
            U_{0} &:= \bigcup_{s \in S_{\geq 0}} H(s) = \{ (\alpha,s) \mid s \in S_{\geq 0}, \alpha \in H(s) \}, \\
            X_{0} &:= \coprod_{(\alpha,s) \in U_{0}} s.
        \end{align*}
        Then
        $$H(X_{0}) = \prod_{(\alpha,s) \in U_{0}} H(s)$$
        and there is an obvious element in $H(X_{0})$, namely the element that is $\alpha \in H(s)$ for $(\alpha,s) \in U_{0}$. We call this element $\alpha_{0}$. The construction is so that if $s \to X_{0}$ is the inclusion of $s$ into $X_{0}$ corresponding to $(\alpha,s) \in U_{0}$ (so into that component), then the induced map $H(X_{0}) \to H(s)$ takes $\alpha_{0}$ to $\alpha$. So we have this special element $\alpha_{0} \in H(X_{0})$, which by Yoneda corresponds to a natural transformation
        $$\phi_{0}:\homom_{\T_{\geq 0}}(-,X_{0}) \to H\vert_{\T_{\geq 0}},$$
        and we have seen that $\phi_{0}(s)$ is surjective for all $s \in S_{\geq 0}$.

        \ver

        \tb{The (first) induction step:} 
        Next, we define $X_{1}$ and $\phi_{1}$ to illustrate the induction step. Let
        \begin{align*}
            U_{1} &:= \bigcup_{s \in S_{\geq 0}} \ker(\phi_{0}(s)) = \{ (f,s) \mid s \in S_{\geq 0}, f \in \homom_{\T_{\geq 0}}(s,X_{0}) : \phi_{0}(s) = 0 \}, \\
            K_{1} &:= \coprod_{(f,s) \in U_{1}} s.
        \end{align*}
        Let
        $$\gamma: K_{1} \to X_{0}$$
        be the map that is $f$ on the $s$-component with index $(f,s)$. Consider then the exact triangle in $\T$:
        \[\begin{tikzcd}
	{K_{1}} & {X_{0}} & {X'_{1}} & {\quad.}
	\arrow["\gamma", from=1-1, to=1-2]
	\arrow[from=1-2, to=1-3]
	\arrow["{+}", from=1-3, to=1-4]
\end{tikzcd}\]
        Note that $X'_{1} \in \T_{\geq -1}$ and not necessarily $\T_{\geq 0}$! Hence, we wish to truncate it and define $X_{1} := \tau_{\geq 0}X'_{1}$. Consider hereto the morphism of exact triangles,
        \[\begin{tikzcd}
	{K_{1}} & {X_{0}} & {X'_{1}} & \quad \\
	{N_{1}} & {X_{0}} & {\tau_{\geq 0}X'_{1}} & {\quad.}
	\arrow["\gamma", from=1-1, to=1-2]
	\arrow["\mu"', from=1-1, to=2-1]
	\arrow["\xi", from=1-2, to=1-3]
	\arrow[Rightarrow, no head, from=1-2, to=2-2]
	\arrow["{+}", from=1-3, to=1-4]
	\arrow["\pi", from=1-3, to=2-3]
	\arrow["\rho"', from=2-1, to=2-2]
	\arrow["\chi"', from=2-2, to=2-3]
	\arrow["{+}", from=2-3, to=2-4]
\end{tikzcd}\]
        Consider then the long exact sequences induced by $H$:
        \[\begin{tikzcd}
	{...} & {H(X'_{1})} & {H(X_{0})} & {H(K_{1})} & {...} \\
	{...} & {H(\tau_{\geq 0}X'_{1})} & {H(X_{0})} & {H(N_{1})} & {...}
	\arrow[from=1-1, to=1-2]
	\arrow["{\overline{\xi}}", from=1-2, to=1-3]
	\arrow["{\overline{\gamma}}", from=1-3, to=1-4]
	\arrow[Rightarrow, no head, from=1-3, to=2-3]
	\arrow[from=1-4, to=1-5]
	\arrow[from=2-1, to=2-2]
	\arrow["{\overline{\pi}}", from=2-2, to=1-2]
	\arrow["{\overline{\chi}}"', from=2-2, to=2-3]
	\arrow["{\overline{\rho}}"', from=2-3, to=2-4]
	\arrow["{\overline{\mu}}", from=2-4, to=1-4]
	\arrow[from=2-4, to=2-5]
\end{tikzcd}\]
        Since $H(\pi) = \overline{\pi}$ is an isomorphism, we obtain the following exact sequence as part of the top row:
        \[\begin{tikzcd}
	{H(\tau_{\geq 0}X'_{1})} & {H(X'_{1})} & {H(K_{1}).}
	\arrow["{\overline{\xi} \circ \overline{\pi}}", from=1-1, to=1-2]
	\arrow["{\overline{\gamma}}", from=1-2, to=1-3]
\end{tikzcd}\]
        Now observe that $\overline{\gamma}(\alpha_{0}) = 0$ because the $f: s \to X_{0}$ are part of the kernel of $\phi_{0}$. Therefore, there exists some $\alpha_{1} \in H(\tau_{\geq 0}X'_{1}) = H(X_{1})$, which then corresponds under Yoneda to a natural transformation
        $$\phi_{1}:\homom_{\T_{\geq 0}}(-,X_{1}) \to H\vert_{\T_{\geq 0}}$$
        so that $$\phi_{1} \circ \homom_{\T_{\geq 0}}(-,\pi \circ \xi) = \phi_{0}.$$

        We can inductively proceed this way and construct objects $X_{i} \in \T_{\geq 0}$ (which are $\tau_{\geq 0}$-truncations of objects $X'_{i} \in \T_{\geq -1}$) and natural transformations $$\phi_{i}: \homom_{\T_{\geq 0}}(-,X_{i}) \to H\vert_{\T_{\geq 0}}$$ that are compatible with each other (i.e. yield a sequential diagram). 
        \ver

        \tb{The homotopy colimit:} We define $$X := \hocolim_{n \geq 0} X_{n}.$$
        Then $X \in \T_{\geq 0}$ because $\T_{\geq 0}$ is closed under countable coproducts. By considering the exact triangle
        \[\begin{tikzcd}
	{\coprod X_{i}} & {\coprod X_{i}} & X & {\quad}
	\arrow["{1 - \tn{shift}}", from=1-1, to=1-2]
	\arrow[from=1-2, to=1-3]
	\arrow["{+}", from=1-3, to=1-4]
\end{tikzcd}\]
        and applying the cohomological functor $H$, we get the exact sequence
        \[\begin{tikzcd}
	{H(X)} & {\coprod H(X_{i})} & {\coprod H(X_{i}).}
	\arrow[from=1-1, to=1-2]
	\arrow["{1-\tn{shift}}", from=1-2, to=1-3]
\end{tikzcd}\]
        The element $\{ \alpha_{i} \}$ lies in the kernel of $1-\tn{shift}$, so there exists some $\alpha \in H(X)$ mapping to it. Again by Yoneda, this corresponds to a natural transformation
        $$\phi: \homom_{\T_{\geq 0}}(-,X) \to H\vert_{\T_{\geq 0}}.$$

        \ver
        
        \tb{$\phi$ is a natural isomorphism:} We start by showing that $\phi(s)$ is an isomorphism for every $s \in S_{\geq 0}$. Observe hereto the commutative diagram
        \[\begin{tikzcd}
	{\homom_{\T_{\geq 0}}(s,X_{0})} \\
	{\homom_{\T_{\geq 0}}(s,X)} & {H(s).}
	\arrow[from=1-1, to=2-1]
	\arrow["{\phi_{0}(s)}", from=1-1, to=2-2]
	\arrow["{\phi(s)}"', from=2-1, to=2-2]
\end{tikzcd}\]
        Since $\phi_{0}(s)$ is surjective (see the base case), so is $\phi(s)$. That leaves injectivity. So let $f \in \ker(\phi(s))$. Then 
        \begin{align*}
            f \in \homom_{\T_{\geq 0}}(s,X) &= \T(s,\hocolim X_{i}) \cong \colim \T(s,X_{i})
        \end{align*}
        by compactness of $s$. So there exists some $f_{i}: s \to X_{i}$ so that the composite
        $$s \xrightarrow{f_{i}} X_{i} \to X$$
        is $f$. Since the following diagram commutes:
        \[\begin{tikzcd}
	{\homom_{\T_{\geq 0}}(s,X_{0})} \\
	{\homom_{\T_{\geq 0}}(s,X)} & {H(s),}
	\arrow["{j_{s}}"', from=1-1, to=2-1]
	\arrow["{\phi_{0}(s)}", from=1-1, to=2-2]
	\arrow["{\phi(s)}"', from=2-1, to=2-2]
\end{tikzcd}\]
        and $j_{s}(f_{i}) = f$ while $\phi(s)(f) = 0$, also $\phi_{0}(s)(f_{i}) = 0$. Hence, $(f_{i},s) \in U_{i+1}$ (see the induction step). Consequently, $f_{i}$ factors through $\gamma_{i}$ in the triangle
        \[\begin{tikzcd}
	{K_{i+1}} & {X_{i}} & {X'_{i+1}} & {\quad.}
	\arrow["{\gamma_{i}}", from=1-1, to=1-2]
	\arrow["{\xi_{i}}", from=1-2, to=1-3]
	\arrow["{_+}", from=1-3, to=1-4]
\end{tikzcd}\]
        Consequently, $\xi_{i} \circ f_{i} = 0$. Since
        \[\begin{tikzcd}
	s & {X_{i}} & {X_{i+1}} \\
	&& {X,}
	\arrow["{f_{i}}", from=1-1, to=1-2]
	\arrow["f"', from=1-1, to=2-3]
	\arrow["{\pi_{i} \circ \xi_{i}}", from=1-2, to=1-3]
	\arrow[from=1-2, to=2-3]
	\arrow[from=1-3, to=2-3]
\end{tikzcd}\]
        we see that $f=0$. This shows injectivity. Consequently, $\phi(s)$ is an isomorphism for all $s \in S_{\geq 0}$. 

        Now, let $\mc{S}$ be the smallest full subcategory of $\T$ containing $S_{\geq 0}$, closed under negative shifts and $\T_{\geq 0}$-coproducts and cones$[-1]$. Let $X$ be an arbitrary object of $\T_{\geq 0}$. We can then carry out the previous construction for $\mc{S}$ and $\homom_{\T_{\geq 0}}(-,X)$ and find an object $Z$ and a natural transformation
        $$\homom_{\mc{S}}(-,Z) \to \homom_{\T_{\geq 0}}(-,X)$$
        that is a natural isomorphism on $\mc{S}$. By Yoneda, there is a morphism $Z \to X$ in $\T_{\geq 0}$ inducing it. We can complete it to a triangle,
        \[\begin{tikzcd}
	Z & X & C & {\quad.}
	\arrow[from=1-1, to=1-2]
	\arrow[from=1-2, to=1-3]
	\arrow["{+}", from=1-3, to=1-4]
\end{tikzcd}\]
        Then $$\forall s \in S_{\geq 0}: \homom_{\T_{\geq 0}}(s,C[-1]) = 0.$$ Hence $C = 0$ because $S_{\geq 0}$ generates $\T_{\geq 0}$. Consequently, $Z \cong X$ so that $X \in \mc{S}$ and also $\mc{S} = \T_{\geq 0}$. Now remark that $\mc{S}$ is contained in the full subcategory of $\T_{\geq 0}$ consisting of the objects on which $\phi$ is an isomorphism. Consequently, $\phi$ is a natural isomorphism. 

        \ver

        \tb{From $\T_{\geq 0}$ to $\T$:} We showed that
        $$\T_{\geq 0}(-,X) \cong H\vert_{\T_{\geq 0}}: \T_{\geq 0} \to \Ab.$$
        Then
        \begin{align*}
            \T(-,X) &\cong \T_{\geq 0}(\tau_{\geq 0}(-),X) \\
            &\cong H\vert_{T_{\geq 0}}(\tau_{\geq 0}(-)) \\
            &\cong H(-). \qedhere
        \end{align*}
    \end{proof}
\end{thm}

\begin{corollary} \label{cor - enough DGInj}
    Let $\T$ be a triangulated category with a $t$-structure $(\T_{\leq 0},\T_{\geq 0})$. Suppose that $\T_{\geq 0}$ is closed under countable coproducts, $\T$ is generated by a set $S \subseteq \T^{c}$, $\T^{\heartsuit}$ has enough injectives and $H^{0}_{t}: \T \to \T^{\heartsuit}$ preserves coproducts. Then $\T$ has enough derived injectives.
    \begin{proof}
        Let $I$ be an injective object in the heart $\T^{\heartsuit}$. Then
        $$\T^{\heartsuit}(H^{0}_{t}(-),I): \T \to \Ab$$
        is a cohomological functor that maps coproducts to products. We clearly also have that $H \circ \tau_{\geq 0} \cong H$. Thus we can apply \Cref{thm - Brown} to obtain
        $$\T^{\heartsuit}(H^{0}_{t}(-),I) \cong \T(-,L(I)),$$
        and the representing object is the derived injective associated to $I$.
    \end{proof}
\end{corollary}

\ver

\section{Comparing the Hochschild complexes}
\label{appendix - Hochschild}

The goal of this appendix is to prove \Cref{cor - HHn bounded tdg}, which states that for $n \geq 2$, the $n$-th Hochschild\index{Hochschild}\index{Hochschild!Hochschild complex $\Cc(-)$}\index{Hochschild!Hochschild cohomology $\tn{HH}^{*}(-)$} cohomology group $\tn{HH}^{n}_{\dg}(\A)$ of an essentially small strongly pretriangulated bounded $t$-dg-category $\A$ linear over a field $k$, parameterizes the bounded $t$-deformations of $\A$ along the dg-$k$-algebra morphism $$\theta_{2-n}: k[\epsilon]/(\epsilon^{2}) \to k$$ with $\vert \epsilon \vert = 2-n$, up to equivalence.

Observe first that in the particular case of $\D^{b}_{\dg}(\mb{A})$ for an abelian $k$-linear category $\mb{A}$, we have the following result of \cite{LVdBHH}[Theorem 6.1].

\begin{thm} \label{thm - HH A DbA}
    Let $\Aa$ be an abelian category. Then the Hochschild complexes of the abelian category $\Aa$, as defined in \cite{LVdBHH}, and of the dg-category $\D_{\dg}^{b}(\Aa)$ are isomorphic in $\tn{Ho}(B_{\infty})$\index{$\tn{Ho}(B_{\infty})$}, the homotopy category of $B_{\infty}$-algebras,
    $$\Cc(\Aa) \cong \Cc(\D^{b}(\Aa)).$$
\end{thm}

Since the Hochschild complex of an abelian category $\mb{A}$ is defined as the Shukla complex of $\Inj(\Ind(\mb{A}))$ (see \cite{LVdBHH}[\S 2.4]), \Cref{lemma - HHn dg parametrization} below together with \Cref{cor - recap defequiv} yields \Cref{cor - HHn bounded tdg} in the special case of $\A := \D^{b}_{\dg}(\mb{A})$.

\begin{lemma} \label{lemma - HHn dg parametrization}
    Let $\mf{a}$ be a dg-$k$-category concentrated in nonpositive degrees. For $n \geq 2$,
    \begin{equation} \label{eq - HHn dg neg}
        \tn{HH}^{n}_{\dg}(\mf{a}) \cong \begin{Bmatrix} \tn{dg-deformations of $\mf{a}$ along } \theta_{2-n} \end{Bmatrix}/ \sim.
    \end{equation}
    \begin{proof}
        We adapt the proof of \cite{Wendy}[Proposition 4.11], which covers the case $n=2$. So assume that $n \geq 3$ and let $\mu$ be the $A_{\infty}$-structure on $\mf{a}$ (consisting only of $\mu_{1}$ and $\mu_{2}$) and consider $\phi \in \mc{Z}^{n}\Cc(\mf{a})$. Then we map $\phi$ to the $A_{\infty}$-category $\mf{a}_{\phi} := (\mf{a}[\epsilon],\mu + \phi \epsilon)$. Let us check that $\mu + \phi \epsilon$ defines a strictly unital $A_{\infty}$-structure on $\mf{a}[\epsilon]$. Note that there is no curvature, i.e. $(\mu + \phi \epsilon)_{0} = 0$, because $\mf{a}$ is concentrated in nonpositive degrees and $\vert \epsilon \vert < 0$. To show that $\mu+\phi \epsilon$ defines an $A_{\infty}$-structure on $\mf{a}[\epsilon]$, it suffices to show that
        \begin{equation} \label{eq - Ainfty condition}
            (\mu+\phi \epsilon) \{ \mu+\phi \epsilon \} = 0,
        \end{equation}
        in the notation of \cite{Wendy}. We can rewrite this as
        $$\mu \{ \mu \} + ( \mu \{ \phi \} + \phi \{ \mu \})\epsilon.$$
        Since $\mu$ defines an $A_{\infty}$-structure on $\mf{a}$, we have that $\mu \{ \mu \} = 0$. Since $\phi$ is a Hochschild cocycle of $\mf{a}$, we also have that $\mu \{ \phi \} + \phi \{ \mu \} = 0$, proving \eqref{eq - Ainfty condition}.
        For the strict unitality of $\mf{a}_{\phi}[\epsilon]$, let  $\eta: k \to \mf{a}$ be the unit of $\mf{a}$. Then the induced degree zero morphism $\widetilde{\eta}: k[\epsilon]/(\epsilon^{2}) \to \mf{a}_{\phi}[\epsilon]$ is a unit for $\mf{a}_{\phi}[\epsilon]$ because $\eta$ is a unit for $(\mf{a}, \mu)$ and $\epsilon$ squares to zero.
 
        Consider now a morphism of cocycles $h:\phi \to \phi'$. This is mapped to the morphism $1 +h \epsilon: B\mf{a}_{\phi}[\epsilon] \to B\mf{a}_{\phi'}[\epsilon]$. If $\delta(h) = \phi'-\phi$, so $[\phi] = [\phi']$, then $\mf{a}_{\phi}[\epsilon] \sim \mf{a}_{\phi'}[\epsilon]$ via $1 + h \epsilon$. So we found a map
        $$\tn{HH}^{n}_{\dg}(\mf{a}) \to \begin{Bmatrix} \tn{strictly unital $A_{\infty}$-deformations of $\mf{a}$ along }\theta_{2-n} \end{Bmatrix}/ \sim$$
        and one can check that it is a bijection. Since the homotopy category of strictly unital $A_{\infty}$-categories is equivalent to the homotopy category of dg-categories, see for example \cite{CanoStel} and \cite{Pascaleff}, we obtain \eqref{eq - HHn dg neg}.
    \end{proof}
\end{lemma}

We now treat the general case. To find a general $t$-dg-analogue for \Cref{thm - HH A DbA}, we use the following result\footnote{Note that \Cref{thm - Keller} was also used to prove \cite{LVdBHH}[Theorem 6.1].} from \cite{KellerDerivedHochschild}, stated more generally as in \cite{LVdBHH}[Theorem 4.1.1].

\begin{thm} \label{thm - Keller}
    Let $\mf{a}$ and $\mf{b}$ be cofibrant dg-$R$-categories. Given a cofibrant $\mf{a}-\mf{b}$-dg-bimodule $M: \mf{a}^{\op} \otimes \mf{b} \to \dgm(R)$, if
    \begin{align} 
        \begin{split} \label{eq - condition 1 Keller}
        &\forall b,b' \in \mf{b}: \quad \textcolor{white}{u} \RHom_{\dgm(\mf{a})}(M(-,b),M(-,b')) \cong \mf{b}(b,b'), 
        \end{split} \\
        \begin{split} \label{eq - condition 2 Keller}
        &\forall a,a' \in \mf{a}: \quad \textcolor{white}{i} \RHom_{\dgm(\mf{b}^{\tn{op}})}(M(a,-),M(a',-)) \cong \mf{a}(a',a),
        \end{split}
    \end{align}
    then the Hochschild complexes of $\mf{a}$ and $\mf{b}$ are isomorphic in $\tn{Ho}(B_{\infty})$:
    $$\Cc(\mf{a}) \cong \Cc(\mf{b}).$$
\end{thm}

\begin{thm} \label{thm - HH A DGInjInddgA agree}
    Let $\A$ be an essentially small strongly pretriangulated $t$-dg-$R$-category. Then there is an isomorphism in $\tn{Ho}(B_{\infty})$,
    $$\Cc(\A) \cong \Cc(\DGInj(\Ind^{\dg,Q,+}(\A))).$$
    In particular, the Hochschild cohomologies agree:
    $$\tn{HH}_{\dg}^{*}(\A) \cong \tn{HH}_{\dg}^{*}(\DGInj(\Ind^{\dg,Q,+}(\A))).$$
    \begin{proof}
        We wish to apply \Cref{thm - Keller}. Let $\J := \DGInj(\Ind^{\dg,Q,+}(\A))$. By replacing $\hproj^{+}(\A)$ by a cofibrant replacement, we can ensure that the full dg-subcategories $\Ind^{\dg,Q,+}(\A)$, $\J$ and $\A$ are cofibrant. As $\A$-$\J$-dg-bimodule we take
        $$M: \A^{\tn{op}} \otimes \J \to \dgm(R): (A,I) \mapsto \underline{\homom}_{\dgm(\A)}(\Y(A),I).$$
        It is cofibrant because it is essentially the inner hom in $\hproj^{+}(\A)$. Note that by Yoneda,
        \begin{equation} \label{eq - M def}
            M(A,I) \cong I(A).
        \end{equation}
        We have to show that conditions \eqref{eq - condition 1 Keller} and \eqref{eq - condition 2 Keller} hold, i.e. that
        \begin{align} 
        \begin{split} \label{eq - condition 1}
        &\forall I,I' \in \J: \quad \textcolor{white}{u} \RHom_{\dgm(\A)}(M(-,I),M(-,I')) \cong \J(I,I'),
        \end{split} \\
        \begin{split} \label{eq - condition 2}
        &\forall A,A' \in \A: \quad \RHom_{\dgm(\J^{\tn{op}})}(M(A,-),M(A',-)) \cong \A(A',A),
        \end{split}
    \end{align}
        We start with condition \eqref{eq - condition 1}. Let $I,I' \in \J$. By \eqref{eq - M def} and the fact that $\J \subseteq \hproj(\A)$, we have that 
        $$\RHom_{\dgm(\A)}(M(-,I),M(-,I')) = \underline{\homom}_{\dgm(\A)}(I,I') = \J(I,I').$$
        Now for condition \eqref{eq - condition 2}, consider $A,A' \in \A$. The idea is to resolve $\Y(A)$. By \Cref{prop - tlc in hqet+pi} and \Cref{thm - correspondence}, $\Ind^{\dg,Q,+}(\A) \cong \Tw^{+}(\J)$. By \cite{GLVdB1}[\S 4.1] there is then a left bounded twisted complex $(I,q_{I}) \in \Tw^{+}(\J)$ and a morphism $\Y(A) \to (I,q_{I})$ that induces a quasi-isomorphism $$\Y(A) \cong \tn{Tot}(I,q_{I}),$$ also called a derived injective resolution of $\Y(A)$. Then
        $$(X,q_{X}) := (\oplus_{i \in \mb{N}} \underline{\homom}_{\dgm(\A)}(I_{i},-)[-i], q)$$
        with
        $$q = \{q^{j}_{i}:= - \circ (q_{I})^{-i}_{-j}: \underline{\homom}_{\dgm(\A)}(I_{i},-)[-i] \to \underline{\homom}_{\dgm(\A)}(I_{j},-)[-j] \}_{i,j \in \mb{Z}}$$
        is a right bounded twisted complex\footnote{This is the dual of \Cref{def - lb twisted complex}, see also \cite{GLVdB1}[Definition 4.4].} in $\Tw^{-}(\dgm(\J^{\tn{op}}))$, and there is a quasi-isomorphism
        \begin{equation} \label{eq - hproj res}
            \tn{Tot}(X,q_{X}) \cong \underline{\homom}_{\dgm(\A)}(\Y(A),-) = M(A,-).
        \end{equation}
        Since $\dgm(\J^{\tn{op}})$ is closed under direct sums and cones, $\tn{Tot}$ is just the totalisation in $\dgm(\J^{\tn{op}})$ and we claim that it moreover lies in $\hproj(\J^{\tn{op}})$. To see this, observe first that every component of $(X,q_{X})$ lies in $\hproj(\J^{\tn{op}})$. By \cite{GLVdB1}[Proposition 4.13],
        $$\tn{Tot}(X,q_{X}) = \hocolim_{p \geq 0} \tn{Tot}(\sigma_{\geq -p} X),$$
        where $\sigma_{\geq -p}$ are the stupid truncations of \cite{GLVdB1}[\S 4.3]. The $\tn{Tot}(\sigma_{\geq -p} X)$ are h-projective as iterated extensions of the h-projective components. By \Cref{prop - classic hocolim} and \Cref{thm - homotopy dg completion closed} (4), this homotopy colimit remains h-projective. All in all, we found that \eqref{eq - hproj res} is a h-projective resolution of $M(A,-)$ in $\dgm(\J^{\tn{op}})$. We can then rewrite the left-hand side of \eqref{eq - condition 2} as
        \begin{align*}
            \RHom_{\dgm(\J^{\tn{op}})}(M(A,-),M(A',-)) &\cong \underline{\homom}_{\dgm(\J^{\tn{op}})}(\tn{Tot}(X,q_{X}),M(A',-)) \\
            &\cong \tn{Tot}(\oplus_{i \in \mb{N}} M(A',I_{i})[i],q') \\
            &\cong \tn{Tot}(\oplus_{i \in \mb{N}} I_{i}(A'),q'') \\
            &\cong \tn{Tot}(X,q_{X})(A') \\
            &\cong \Y(A)(A') \\
            &= \A(A',A),
        \end{align*}
        where $q',q''$ denote the induced families and where we used the fact that $\tn{Tot}: \Tw^{+}(\J) \to \hproj^{+}(\A)$ commutes with evaluation in the fourth line.
    \end{proof}
\end{thm}

We can now combine \Cref{thm - HH A DGInjInddgA agree} and \Cref{lemma - HHn dg parametrization} with the deformation equivalences of \Cref{cor - recap defequiv} to show that $\tn{HH}^{n \geq 2}_{\dg}(\A)$ for an essentially small strongly pretriangulated dg-$k$-category $\A$ with a bounded $t$-structure, parametrizes the bounded $t$-deformations of $\A$ along $\theta_{2-n}$.

\begin{corollary} \label{cor - HHn bounded tdg}
    Let $\A$ be an essentially small strongly pretriangulated dg-$k$-category with a bounded $t$-structure. For $n \geq 2$,
    \begin{equation} \label{eq - HHn A small}
        \tn{HH}^{n}_{\dg}(\A) \cong \begin{Bmatrix} \tn{bounded $t$-deformations} \\ \tn{of $\A$ along } \theta_{2-n} \end{Bmatrix}/ \sim.
    \end{equation}
    \begin{proof}
        By \Cref{thm - HH A DGInjInddgA agree}, we have that
        $$\tn{HH}^{*}_{\dg}(\A) \cong \tn{HH}^{*}_{\dg}(\DGInj(\Ind^{\dg,Q,+}(\A))).$$
        For $n \geq 2$, the Hochschild cohomology group $\tn{HH}^{n}_{\dg}(\DGInj(\Ind^{\dg,Q,+}(\A)))$ on the right-hand side parametrizes dg-deformations along $\theta_{2-n}$ by \Cref{lemma - HHn dg parametrization}. Thus, we obtain
        $$\tn{HH}^{n}_{\dg}(\A) \cong \begin{Bmatrix} \tn{dg-deformations of} \\ \tn{$\DGInj(\Ind^{\dg,Q,+}(\A))$ along } \theta_{2-n} \end{Bmatrix}/ \sim.$$
        By \Cref{cor - recap defequiv}, this implies \eqref{eq - HHn A small}.
    \end{proof}
\end{corollary}

\begin{remark} 
    More generally, if $k$ is a field of characteristic zero, then one may deduce from \Cref{thm - HH A DGInjInddgA agree} and \Cref{cor - recap defequiv} that the Hochschild complex $\Cc(\A)[1]$ controls the $t$-dg-deformation theory of $\A$ for \ti{local artinian dg-$k$-algebras}\footnote{Here we mean augmented dg-$k$-algebras $R$ whose augmentation maps $\theta: R \to k$ satisfy the conditions set forth in the Conventions and appropriate maps between them, see also the reference cited there.} $R$ via the Maurer-Cartan formalism (cf. \Cite{LVdBHH}[Remark 3.5]), in the sense that we have a bijection
    \begin{equation} \label{eq - MC}
        \tn{Sk}(\Def^{t,b}(\A)) \cong \tn{MC}[\Cc(\A)[1] \otimes \Kernel(\theta)]/\sim,
    \end{equation}
     where the right-hand side denotes the Maurer-Cartan elements up to gauge equivalence \cite{Gauge} with $\theta: R \to k$ the augmentation map.
    %This implies \Cref{cor - HHn bounded tdg} -- which holds for an arbitrary field -- since $$\tn{HH}^{n}_{\dg}(\A) = \tn{MC}[\Cc(\A)[1] \otimes \Kernel(\theta_{2-n})]/\sim,$$
    This recovers \Cref{cor - HHn bounded tdg} since the right-hand side of \eqref{eq - MC} for $\theta_{2-n}$ equals $\tn{HH}^{n}_{\dg}(\A)$.
\end{remark}

% Taking $\A := \D^{b}(\Aa)$ provides a deformation-theoretic interpretation of $\tn{HH}^{n \geq 3}_{\tn{ab}}(\Aa)$ in terms of bounded $t$-deformations of $\D^{b}(\Aa)$ along $\theta_{2-n}$.

% \begin{corollary} \label{cor - HHn ab Dba}
%     Let $\Aa$ be an abelian category. For $n \geq 3$,
%      $$\tn{HH}^{n}_{\tn{ab}}(\Aa) \cong \begin{Bmatrix} \tn{bounded $t$-deformations} \\ \tn{of $\D^{b}(\Aa)$ along } \theta_{2-n} \end{Bmatrix}/ \sim.$$
%      \begin{proof}
%          Combine \Cref{thm - HHn bounded tdg} and \Cref{thm - HH A DbA}.
%      \end{proof}
% \end{corollary}

% References --------------------------------------------------------------------
\newpage
%\addcontentsline{toc}{section}{References}
\printbibliography

@book {AdamekRosicky,
    AUTHOR = {Ad\'{a}mek, Ji\v{r}\'{\i} and Rosick\'{y}, Ji\v{r}\'{\i}},
     TITLE = {Locally presentable and accessible categories},
    SERIES = {London Mathematical Society Lecture Note Series},
    VOLUME = {189},
 PUBLISHER = {Cambridge University Press, Cambridge},
      YEAR = {1994},
     PAGES = {xiv+316},
}

@article{Beilinson,
    AUTHOR = {Beilinson, A.A.},
    TITLE = {Coherent sheaves on {$P^{n}$} and problems of linear algebra.},
    Year = {1978},
    Pages = {214-216},
    Number = {12},
    Journal = {Funct. Anal. Its Appl.},
    FJournal = {Functional Analysis and Its Applications},
}

@article {BondalKapranov,
    AUTHOR = {Bondal, A. I. and Kapranov, M. M.},
     TITLE = {Framed triangulated categories},
   JOURNAL = {Mat. Sb.},
  FJOURNAL = {Matematicheski\u{\i} Sbornik},
    VOLUME = {181},
      YEAR = {1990},
    NUMBER = {5},
     PAGES = {669--683},
}

@article {Lopez,
    AUTHOR = {Alonso Tarrio, Leovigildo and Jeremias Lopez, Ana and Souto
              Salorio, Maria Jose},
     TITLE = {Construction of {$t$}-structures and equivalences of derived
              categories},
   JOURNAL = {Trans. Amer. Math. Soc.},
  FJOURNAL = {Transactions of the American Mathematical Society},
    VOLUME = {355},
      YEAR = {2003},
    NUMBER = {6},
     PAGES = {2523--2543},
}

@article {Antieau,
    AUTHOR = {Antieau, Benjamin and Gepner, David and Heller, Jeremiah},
     TITLE = {{$K$}-theoretic obstructions to bounded {$t$}-structures},
   JOURNAL = {Invent. Math.},
  FJOURNAL = {Inventiones Mathematicae},
    VOLUME = {216},
      YEAR = {2019},
    NUMBER = {1},
     PAGES = {241--300},
}

@article {BMR,
    AUTHOR = {Barthel, Tobias and May, J. P. and Riehl, Emily},
     TITLE = {Six model structures for {DG}-modules over {DGA}s: model
              category theory in homological action},
   JOURNAL = {New York J. Math.},
  FJOURNAL = {New York Journal of Mathematics},
    VOLUME = {20},
      YEAR = {2014},
     PAGES = {1077--1159},
}

@incollection {BBD,
    AUTHOR = {Beilinson, A. A. and Bernstein, J. and Deligne, P.},
     TITLE = {Faisceaux pervers},
 BOOKTITLE = {Analysis and topology on singular spaces, {I} ({L}uminy,
              1981)},
    SERIES = {Ast\'{e}risque},
    VOLUME = {100},
     PAGES = {5--171},
 PUBLISHER = {Soc. Math. France, Paris},
      YEAR = {1982},
}

@article {BondalVdB,
    AUTHOR = {Bondal, A. and van den Bergh, M.},
     TITLE = {Generators and representability of functors in commutative and
              noncommutative geometry},
   JOURNAL = {Mosc. Math. J.},
  FJOURNAL = {Moscow Mathematical Journal},
    VOLUME = {3},
      YEAR = {2003},
    NUMBER = {1},
     PAGES = {1--36, 258},
}

@incollection {CalaqueGrivaux,
    AUTHOR = {Calaque, Damien and Grivaux, Julien},
     TITLE = {Formal moduli problems and formal derived stacks},
 BOOKTITLE = {Derived algebraic geometry},
    SERIES = {Panor. Synth\`eses},
    VOLUME = {55},
     PAGES = {85--145},
 PUBLISHER = {Soc. Math. France, Paris},
      YEAR = {2021},
}

@article {CanoStel,
    AUTHOR = {Canonaco, Alberto and Ornaghi, Mattia and Stellari, Paolo},
     TITLE = {Localizations of the category of {$A_\infty$} categories and
              internal {H}oms},
   JOURNAL = {Doc. Math.},
  FJOURNAL = {Documenta Mathematica},
    VOLUME = {24},
      YEAR = {2019},
     PAGES = {2463--2492},
}

@article {GLVdB1,
    AUTHOR = {Genovese, Francesco and Lowen, Wendy and Van den Bergh,
              Michel},
     TITLE = {t-structures and twisted complexes on derived injectives},
   JOURNAL = {Adv. Math.},
  FJOURNAL = {Advances in Mathematics},
    VOLUME = {387},
      YEAR = {2021},
     PAGES = {Paper No. 107826, 70},
}

@article {ThesisFrancesco,
    AUTHOR = {Genovese, Francesco},
     TITLE = {Adjunctions of quasi-functors between {DG}-categories},
   JOURNAL = {Appl. Categ. Structures},
  FJOURNAL = {Applied Categorical Structures. A Journal Devoted to
              Applications of Categorical Methods in Algebra, Analysis,
              Order, Topology and Computer Science},
    VOLUME = {25},
      YEAR = {2017},
    NUMBER = {4},
     PAGES = {625--657},
}

@misc {FrancescoPrivate,
    AUTHOR = {Genovese, Francesco},
    TITLE = {Hocolim dg},
    Note = {Private communication},
    Month = {11},
    Year = {2022},
}

@misc {GLVdB2,
    AUTHOR = {Genovese, Francesco and Lowen, Wendy and Van den Bergh,
              Michel},
     TITLE = {t-structures on dg-categories and derived deformations},
}

@article{JuliaFrancesco,
   title={A Derived Gabriel–Popescu Theorem for t-Structures via Derived Injectives},
   volume={2023},
   ISSN={1687-0247},
   number={6},
   journal={International Mathematics Research Notices},
   publisher={Oxford University Press (OUP)},
   author={Genovese, Francesco and Ramos González, Julia},
   year={2022},
   month=jan, pages={4695–4760} }

@misc{KellerDerivedHochschild, 
    AUTHOR = {Keller, Bernhard},
    TITLE = {Derived invariance of higher structures on the Hochschild complex},
    YEAR = {2003},
    URL = {https://webusers.imj-prg.fr/~bernhard.keller/publ/dih.pdf},
}

@article {KellerWendy,
    AUTHOR = {Keller, Bernhard and Lowen, Wendy},
     TITLE = {On {H}ochschild cohomology and {M}orita deformations},
   JOURNAL = {Int. Math. Res. Not. IMRN},
  FJOURNAL = {International Mathematics Research Notices. IMRN},
      YEAR = {2009},
    NUMBER = {17},
     PAGES = {3221--3235},
}

@article {KellerWendyNicolas,
    AUTHOR = {Keller, Bernhard and Lowen, Wendy and Nicol\'as, Pedro},
     TITLE = {On the (non)vanishing of some ``derived'' categories of curved
              dg algebras},
   JOURNAL = {J. Pure Appl. Algebra},
  FJOURNAL = {Journal of Pure and Applied Algebra},
    VOLUME = {214},
      YEAR = {2010},
    NUMBER = {7},
     PAGES = {1271--1284},
}

@book {Kelly,
    AUTHOR = {Kelly, Gregory Maxwell},
     TITLE = {Basic concepts of enriched category theory},
    SERIES = {London Mathematical Society Lecture Note Series},
    VOLUME = {64},
 PUBLISHER = {Cambridge University Press, Cambridge-New York},
      YEAR = {1982},
     PAGES = {245},
}

@article {KrauseAuslander,
    AUTHOR = {Krause, Henning},
     TITLE = {Deriving {A}uslander's formula},
   JOURNAL = {Doc. Math.},
  FJOURNAL = {Documenta Mathematica},
    VOLUME = {20},
      YEAR = {2015},
     PAGES = {669--688},
}

@misc{LehmannLowen,
      title={Filtered derived categories of curved deformations}, 
      author={Alessandro Lehmann and Wendy Lowen},
      year={2024},
      eprint={2402.08660},
      archivePrefix={arXiv},
      primaryClass={math.KT},
}

@misc{Lehmann,
      title={Hochschild cohomology parametrizes curved Morita deformations}, 
      author={Alessandro Lehmann},
      year={2024},
      eprint={2406.04945},
      archivePrefix={arXiv},
      primaryClass={math.KT},
}

@article {Wendy,
    AUTHOR = {Lowen, Wendy},
     TITLE = {Hochschild cohomology, the characteristic morphism and derived
              deformations},
   JOURNAL = {Compos. Math.},
  FJOURNAL = {Compositio Mathematica},
    VOLUME = {144},
      YEAR = {2008},
    NUMBER = {6},
     PAGES = {1557--1580},
}

@article {LVdBHH,
    AUTHOR = {Lowen, Wendy and Van den Bergh, Michel},
     TITLE = {Hochschild cohomology of abelian categories and ringed spaces},
   JOURNAL = {Adv. Math.},
  FJOURNAL = {Advances in Mathematics},
    VOLUME = {198},
      YEAR = {2005},
    NUMBER = {1},
     PAGES = {172--221},
}

@misc{LVdBCurv,
      title={The curvature problem for formal and infinitesimal deformations}, 
      author={Lowen, Wendy and Van den Bergh, Michel},
      year={2015},
      eprint={1505.03698},
      archivePrefix={arXiv},
      primaryClass={math.KT},
}

@article {Obstruction,
    AUTHOR = {Lowen, Wendy},
     TITLE = {Obstruction theory for objects in abelian and derived
              categories},
   JOURNAL = {Comm. Algebra},
  FJOURNAL = {Communications in Algebra},
    VOLUME = {33},
      YEAR = {2005},
    NUMBER = {9},
     PAGES = {3195--3223},
}

@article {AbelianDeformations,
    AUTHOR = {Lowen, Wendy and Van den Bergh, Michel},
     TITLE = {Deformation theory of abelian categories},
   JOURNAL = {Trans. Amer. Math. Soc.},
  FJOURNAL = {Transactions of the American Mathematical Society},
    VOLUME = {358},
      YEAR = {2006},
    NUMBER = {12},
     PAGES = {5441--5483},
}

@misc{LurieDAGI,
    AUTHOR = {Lurie, J.},
     TITLE = {Derived Algebraic Geometry I: Stable ${\infty}$-categories},
       URL = {https://people.math.harvard.edu/~lurie/papers/DAG-I.pdf},
}

@article{Gauge, 
    Title = {Deformation quantization of algebraic varieties},
    volume={56},
   number={3},
   journal={Letters in Mathematical Physics},
   publisher={Springer Science and Business Media LLC},
   author={Kontsevich, Maxim},
   year={2001},
   pages={271–294} }

@article {NeemanBokstedt,
    AUTHOR = {B\"okstedt, Marcel and Neeman, Amnon},
     TITLE = {Homotopy limits in triangulated categories},
   JOURNAL = {Compositio Math.},
  FJOURNAL = {Compositio Mathematica},
    VOLUME = {86},
      YEAR = {1993},
    NUMBER = {2},
     PAGES = {209--234},
}

@book {NeemanBook,
    AUTHOR = {Neeman, Amnon},
     TITLE = {Triangulated categories},
    SERIES = {Annals of Mathematics Studies},
    VOLUME = {148},
 PUBLISHER = {Princeton University Press, Princeton, NJ},
      YEAR = {2001},
     PAGES = {viii+449},
}

@misc{NeemanNLC,
      title={Non-left-complete derived categories}, 
      author={Amnon Neeman},
      year={2011},
      eprint={1103.5539},
      archivePrefix={arXiv},
      primaryClass={math.CT}
}

@article {NeemanGrothendieckDuality,
    AUTHOR = {Neeman, Amnon},
     TITLE = {The {G}rothendieck duality theorem via {B}ousfield's
              techniques and {B}rown representability},
   JOURNAL = {J. Amer. Math. Soc.},
  FJOURNAL = {Journal of the American Mathematical Society},
    VOLUME = {9},
      YEAR = {1996},
    NUMBER = {1},
     PAGES = {205--236},
}

@misc{NeemanFunctors,
      title={The category $\big[{\mathcal T}^c\big]^{\text{op}}$ as functors on ${\mathcal T}^b_c$}, 
      author={Amnon Neeman},
      year={2018},
      eprint={1806.05777},
      archivePrefix={arXiv},
      primaryClass={math.CT}
}

@article {OrlovSing,
    AUTHOR = {Orlov, D. O.},
     TITLE = {Triangulated categories of singularities and {D}-branes in
              {L}andau-{G}inzburg models},
   JOURNAL = {Tr. Mat. Inst. Steklova},
  FJOURNAL = {Trudy Matematicheskogo Instituta Imeni V. A. Steklova},
    VOLUME = {246},
      YEAR = {2004},
     PAGES = {240--262},
}

@article {Pascaleff,
    AUTHOR = {Pascaleff, James},
     TITLE = {Remarks on the equivalence between differential graded
              categories and {${\rm A}$}-infinity categories},
   JOURNAL = {Homology Homotopy Appl.},
  FJOURNAL = {Homology, Homotopy and Applications},
    VOLUME = {26},
      YEAR = {2024},
    NUMBER = {1},
     PAGES = {275--285},
}

@article {NonFM,
    AUTHOR = {Rizzardo, Alice and Van den Bergh, Michel and Neeman, Amnon},
     TITLE = {An example of a non-{F}ourier-{M}ukai functor between derived
              categories of coherent sheaves},
   JOURNAL = {Invent. Math.},
  FJOURNAL = {Inventiones Mathematicae},
    VOLUME = {216},
      YEAR = {2019},
    NUMBER = {3},
     PAGES = {927--1004},
}

@article {Porta,
    AUTHOR = {Porta, Marco},
     TITLE = {The {P}opescu-{G}abriel theorem for triangulated categories},
   JOURNAL = {Adv. Math.},
  FJOURNAL = {Advances in Mathematics},
    VOLUME = {225},
      YEAR = {2010},
    NUMBER = {3},
     PAGES = {1669--1715},
}

@book{SGA4I,
     TITLE = {Théorie des topos et cohomologie étale des schémas. {T}ome 1:
              {T}héorie des topos},
    AUTHOR = {
Artin, Michael and Grothendieck, Alexander and Verdier, Jean-Louis},
    SERIES = {Lecture Notes in Mathematics, Vol. 269},
      NOTE = {Séminaire de Géométrie Algébrique du Bois-Marie 1963--1964
              (SGA 4),
              Dirigé par M. Artin, A. Grothendieck, et J. L. Verdier. Avec
              la collaboration de N. Bourbaki, P. Deligne et B. Saint-Donat},
 PUBLISHER = {Springer-Verlag, Berlin-New York},
      YEAR = {1972},
     PAGES = {xix+525},
}

\newpage
\printindex

\end{document}